\documentclass[twoside,11pt,a4paper]{article}
\usepackage{fullpage}
\usepackage{graphicx}
\usepackage{amsmath,amsthm}
\usepackage{amssymb,amsfonts,amsbsy}
\usepackage{dsfont,eufrak,mathrsfs}
\usepackage{calligra}
\usepackage{color}
\usepackage[T1]{fontenc}

\DeclareMathAlphabet{\mathpzc}{OT1}{pzc}{m}{it}
\DeclareFontShape{T1}{calligra}{m}{n}{<->s*[1.2]callig15}{}
\DeclareMathAlphabet{\mathcalligra}{T1}{calligra}{m}{n}

\newtheorem{Def}{Definition}[section]

\newtheorem{Teo}[Def]{Theorem}
\newtheorem{Lema}[Def]{Lemma}
\newtheorem{Prop}[Def]{Proposition}
\newtheorem{Cor}[Def]{Corollary}

\newtheorem{Rem}[Def]{Remark}

\newtheorem{Theo}{Theorem}

\newcommand*{\fun}[3]{#1: #2 \rightarrow #3}
\newcommand*{\ccomp}[1]{\mathcal{C}(#1)}

\newcommand*{\Mod}[1]{\mathrm{Mod}(#1)}
\newcommand*{\EMod}[1]{\mathrm{Mod}^{*}(#1)}
\newcommand*{\stab}[1]{\mathrm{stab}_{pt}(#1)}
\newcommand*{\nat}{\mathbb{N}}
\newcommand*{\X}{\mathfrak{X}}
\newcommand*{\Cf}{\mathscr{C}}
\newcommand*{\Cfn}{\Cf_{0}}
\newcommand*{\Cff}{\Cf_{f}}
\newcommand*{\Bf}{\mathscr{B}}
\newcommand*{\Bfn}{\Bf_{0}}
\newcommand*{\Sf}{\mathscr{S}}
\newcommand*{\Mf}{\mathscr{A}}
\newcommand*{\Gf}{\mathscr{G}}
\newcommand*{\Df}{\mathscr{D}}
\newcommand*{\Ef}{\mathscr{E}}
\newcommand*{\Hff}{\mathscr{H}}
\newcommand*{\Hcf}{H_{\Cf}}
\newcommand*{\Hcp}{\Hcf^{+}}

\newcommand*{\N}[3]{N_{#1}^{#2, \hspace{0.05cm} #3}}
\newcommand*{\lk}[1]{\mathcalligra{lk} \hspace{0.12cm}(#1)}

\newcommand*{\eps}[2]{\epsilon^{#1, \hspace{0.05cm} #2}}
\newcommand*{\epsp}[3]{\epsilon_{#1}^{(#2,#3)}}
\newcommand*{\ColonEqq}{\mathrel{\mathop:}=}
\newcommand*{\biprime}{\prime\prime}

\newcommand*{\Aut}[1]{\mathrm{Aut}(#1)}

\author{Jes\'{u}s Hern\'{a}ndez Hern\'{a}ndez}
\title{Exhaustion of the curve graph via rigid expansions}
\date{}

\begin{document}
\maketitle
\begin{abstract}
 For an orientable surface $S$ of finite topological type with genus $g \geq 3$, we construct a finite set of curves whose union of iterated rigid expansions is the curve graph $\ccomp{S}$. The set constructed, and the method of rigid expansion, are closely related to Aramayona and Leiniger's finite rigid set in \cite{Ara1} and \cite{Ara2}, and in fact a consequence of our proof is that Aramayona and Leininger's set also exhausts the curve graph via rigid expansions.
\end{abstract}
\section*{Introduction}
\indent In this article we consider an orientable surface $S_{g,n}$ of finite topological type with genus $g \geq 3$ and $n \geq 0$ punctures. The mapping class group of $S_{g,n}$, denoted by $\Mod{S_{g,n}}$ is the group of orientation preserving self-homeomorphisms of $S_{g,n}$. The extended mapping class group of $S_{g,n}$, denoted by $\EMod{S_{g,n}}$ is the group of isotopy classes of self-homeomorphisms of $S_{g,n}$.\\
\indent In order to study these groups, Harvey in 1979 (see \cite{Harvey}) introduced the curve complex of a surface as the simplicial complex whose vertices are isotopy classes of essential curves, and simplices are defined by disjointness (see Section \ref{prelim} for details). We call the $1$-skeleton of the curve complex the \textit{curve graph}, which we denote by $\ccomp{S_{g,n}}$.\\
\indent There is a natural link between the curve complex and $\Mod{S_{g,n}}$ and $\EMod{S_{g,n}}$. Ivanov (in \cite{Ivanov}) linked the curve complex to $\EMod{S_{g,n}}$ via simplicial automorphisms, while Harer (in \cite{Harer}) linked the curve complex with $\Mod{S_{g,n}}$ by their (co-)homology.\\
\indent On one hand, in \cite{Ivanov}, \cite{Korkmaz} and \cite{Luo} it was proved that for most surfaces every automorphism of the curve graph is induced by a homeomorphism of $S_{g,n}$, with the well-known exception of $S_{1,2}$. Later on, there were generalizations of this result for larger classes of simplicial maps (see \cite{Irmak1}, \cite{Irmak2}, \cite{Irmak3}, \cite{BehrMar}), until Shackleton (see \cite{Shack}) proved that any locally injective self-map of the curve graph is induced by a homeomorphism (for surfaces of high-enough complexity).\\
\indent Thereafter, Aramayona and Leininger introduced in \cite{Ara1} the concept of a rigid set of the curve graph, which is a full subgraph $Y$ such that any locally injective map from $Y$ to $\ccomp{S_{g,n}}$ is the restriction to $Y$ of an automorphism, unique up to the pointwise stabilizer of $Y$ in $\Aut{\ccomp{S_{g,n}}}$. By Shackleton's result, the curve graph itself is a rigid set. In \cite{Ara1} they also construct a finite rigid set for any orientable surface of finite topological type. See Section \ref{chap3} below.\\
\indent On the other hand, it is a well-known result by Harer \cite{Harer} that the curve complex is homotopically equivalent to a bouquet of spheres, which is used to determine the virtual cohomological dimension of the mapping class group.\\
\indent Later on, Birman, Broaddus and Menasco in \cite{Broaddus} proved that Aramayona and Leininger's finite rigid set either is (for $g = 0$ and $n \geq 5$) or contains (for $g \geq 1$ and $n \leq 1$) a $\Mod{S_{g,n}}$-module generator of the reduced homology of the curve complex. Thus, they link the (co-)homological and simplicial sides of the study of the mapping class group and curve complexes.\\
\indent Afterwards, Aramayona and Leininger proved in \cite{Ara2} that for almost all surfaces of finite topological type, there exists an increasing sequence of finite rigid sets that exhaust the curve graph, each of which has trivial pointwise stabilizer in $\EMod{S_{g,n}}$. Note that this is \textbf{not trivial}, given that there exist examples of supersets of a rigid set that are not rigid themselves.\\
\indent While their proof is \textit{effective} for the result, it does not lend itself to improving other results concerning simplicial maps. In this work, we prove a similar result to theirs; however, we use a method developed in \cite{Ara2} for expanding subgraphs. This method can be used to obtain new results concerning edge-preserving maps; the details of these results are given in \cite{Thesis} and will appear in a second paper \cite{JHH2}. We call this method \textit{rigid expansion}.\\
\indent We define the first rigid expansion of a subgraph $Y$, denoted as $Y^{1}$, as the union of $Y$ with all the curves uniquely determined by subsets of $Y$, where a curve $\beta$ is uniquely determined by a subset $B$ of $\ccomp{S_{g,n}}$ if it is the unique curve disjoint from every element in $B$. We also define $Y^{0} = Y$ and, inductively, $Y^{k} = (Y^{k-1})^{1}$.\\
\indent Note in particular that if $\beta$ is uniquely determined by $B$, then for every $h \in \EMod{S_{g,n}}$ we have that $h(\beta)$ is uniquely determined by $h(B)$.\\
\indent Now we can state the main result of this work.
\begin{Theo}\label{TheoA}
 Let $S_{g,n}$ be an orientable surface of finite topological type with genus $g \geq 3$, $n \geq 0$ punctures, and empty boundary. There exists a finite subgraph of $\ccomp{S_{g,n}}$ whose union of iterated rigid expansions is equal to $\ccomp{S_{g,n}}$.
\end{Theo}
\indent The proof of Theorem \ref{TheoA} is divided into two cases: the closed surface case (see Theorem \ref{Thm2} in Section \ref{chap1}) and the punctured surface case (see Theorem \ref{Thm3} in Section \ref{chap2}). We begin by defining a particular set of curves (based on the rigid set introduced in \cite{Ara1}) that we call the \textit{principal set}, and use a Humphries-Lickorish generating set of $\Mod{S_{g,n}}$ to see that the positive and negative translations of the principal set are contained in some rigid expansion of it (the principal set); afterwards, the iterated use of this result allows us to see that most topological types of curves are in some rigid expansion, while the rest of the topological types are uniquely determined by finite sets of curves from the previous cases.\\
\indent Afterwards, in Section \ref{chap3} we reintroduce the rigid set of \cite{Ara1}, denoted by $\X(S_{g,n})$. Note that while Birman, Broaddus and Menasco's homological spheres in \cite{Broaddus} (which is a subset of $\X(S_{g,n})$ for $g \geq 1$ and $n \leq 1$) are not contained in the principal set of a closed surface, they \textit{are} contained in their first rigid expansion. Then, we use Theorem \ref{TheoA} to obtain an analogous result for Aramayona and Leininger's finite rigid set.
\begin{Theo}\label{Xexhausts}
 Let $S_{g,n}$ be an orientable surface of genus $g \geq 3$, $n \geq 0$ punctures and empty boundary. Then $\bigcup_{i \in \nat} \X(S_{g,n})^{i} = \ccomp{S_{g,n}}$.
\end{Theo}
\indent We must remark that this work is the published version of the first two chapters of the author's Ph.D. thesis, and as was mentioned before these results are used to obtain new results on simplicial maps of different graphs. In particular we use these results in \cite{JHH2} to prove that under certain conditions on the surfaces, all edge-preserving maps between a priori different curve graphs are actually induced by homeomorphisms between the underlying surfaces.\\[0.3cm]
\textbf{Acknowledgements:} The author thanks his Ph.D. advisors, Javier Aramayona and Hamish Short, for their very helpful suggestions, talks, corrections, and specially for their patience while giving shape to this work.
\section{Preliminaries}\label{prelim}
\indent We suppose $S_{g,n}$ is an orientable surface of finite topological type with empty boundary, genus $g \geq 3$ and $n$ punctures. The \textit{mapping class group of} $S_{g,n}$, denoted by $\Mod{S_{g,n}}$, is the group of isotopy classes of orientation preserving self-homeomorphisms of $S_{g,n}$; the \textit{extended mapping class of} $S_{g,n}$, denoted by $\EMod{S_{g,n}}$, is the group of isotopy classes of \textit{all} self-homeomorphisms of $S_{g,n}$. Note that $\Mod{S_{g,n}}$ is an index $2$ subgroup of $\EMod{S_{g,n}}$.\\
\indent A \textit{curve} $\alpha$ is the topological embedding of the unit circle into the surface. We often abuse notation and call ``curve'' the embedding, its image on $S_{g,n}$ or its isotopy class. The context makes clear which use we mean.\\
\indent A curve is \textit{essential} if it is neither null-homotopic nor homotopic to the boundary curve of a neighbourhood of a puncture.\\
\indent The (geometric) intersection number of two (isotopy classes of) curves $\alpha$ and $\beta$ is defined as follows: $$i(\alpha,\beta) \ColonEqq \min \{|a \cap b| : a \in \alpha, b \in \beta\}.$$
\indent Let $\alpha$ and $\beta$ be two curves on $S_{g,n}$. As a convention for this work, we say $\alpha$ and $\beta$ are \textit{disjoint} if $i(\alpha,\beta) = 0$ \textbf{and} $\alpha \neq \beta$.\\
\indent Under the conditions on $S_{g,n}$ imposed above, we define the \textit{curve graph of} $S_{g,n}$, denoted by $\ccomp{S_{g,n}}$, as the simplicial graph whose vertices are the isotopy classes of essential curves on $S_{g,n}$, and two vertices span an edge if the corresponding curves are disjoint.\\
\indent Let $\beta$ be an essential curve on $S_{g,n}$ and $B$ a set of curves on $S_{g,n}$. We say $\beta$ is \textit{uniquely determined by} $B$, denoted $\beta = \langle B \rangle$, if $\beta$ is the unique essential curve on $S_{g,n}$ that is disjoint from every element in $B$, i.e. $$\{\beta\} = \bigcap_{\gamma \in B} \lk{\gamma},$$ where $\lk{\gamma}$ denotes the link of $\gamma$ in $\ccomp{S_{g,n}}$.\\
\indent Let $Y \subset \ccomp{S_{g,n}}$; the first \textit{rigid expansion} of $Y$ is defined as $$Y^{1} \ColonEqq Y \cup \{\beta: \beta = \langle B \rangle, B \subset Y\};$$ we also define $Y^{0} = Y$ and, inductively, $Y^{k} = (Y^{k-1})^{1}$.\\
\section{Closed surface case}\label{chap1}
In this section, we suppose that $S$ is a closed surface of genus $g \geq 3$. This section is divided as follows: Subsection \ref{Chap1Sec1} gives some definitions, fixes the principal set, states the main result of the section, and gives the proof of said result pending the proof of a technical lemma; Subsections \ref{subsec4-2}, \ref{subsec4-3}, and \ref{Chap1Sec4} give the proofs of the claims for the technical lemma.
\subsection{Statement and proof of Theorem \ref{Thm2}}\label{Chap1Sec1}
\indent Let $k \in \mathbb{Z}^{+}$ and $C = \{\gamma_{0}, \ldots, \gamma_{k}\}$ be an ordered set of $k+1$ curves in $S$. It is called a \textit{chain} of length $k+1$ if $i(\gamma_{i},\gamma_{i+1}) = 1$ for $0 \leq i \leq k-1$, and $\gamma_{i}$ is disjoint from $\gamma_{j}$ for $|i - j| > 1$. On the other hand, $C$ is called a \textit{closed chain} of length $k+1$ if $i(\gamma_{i},\gamma_{i+1}) = 1$ for $0 \leq i \leq k$ modulo $k+1$, and $\gamma_{i}$ is disjoint from $\gamma_{j}$ for $|i - j| > 1$ (modulo $k+1$); a closed chain is maximal if it has length $2g+2$. A \textit{subchain} is an ordered subset of either a chain or a closed chain which is itself a chain, and its length is its cardinality.\\
\indent Recalling that $k \geq 1$, note that if $C$ is a chain (or a subchain), then every element of $C$ is a nonseparating curve. Also, if $C$ has odd length, a closed regular neighbourhood $N(C)$ has two boundary components; we call these curves the bounding pair associated to $C$.\\
\indent Let $\Cf = \{\alpha_{0}, \ldots, \alpha_{2g+1}\}$ be the closed chain in $S$ depicted in Figure \ref{OriginalChainv2}. Observe it is a maximal closed chain, and given any other maximal closed chain $C$ there exists an element of $\Mod{S}$ that maps $C$ to $\Cf$ (see \cite{FarbMar}).\\
\indent We define the set $\Bf$ as the union of the bounding pairs associated to the subchains of odd length of $\Cf$.
\begin{figure}[h]
\begin{center}
 \includegraphics[width=9cm]{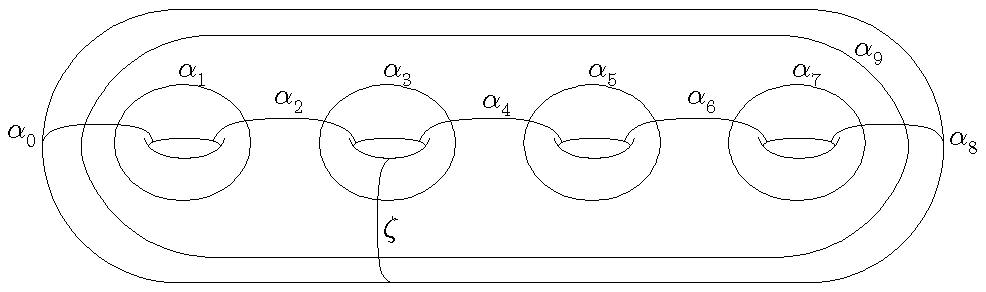} \caption{The set $\Cf = \{\alpha_{0}, \ldots, \alpha_{2g+1}\}$ and the curve $\zeta$ -- one of the curves of the bounding pair associated to $\{\alpha_{0}, \alpha_{1}, \alpha_{2}\}$.} \label{OriginalChainv2} 
\end{center}
\end{figure}\\
\indent Now we are able to state the main result for the closed surface case.
\begin{Teo}\label{Thm2}
 Let $S$ be an orientable closed surface with genus $g \geq 3$, and let $\Cf$ and $\Bf$ be defined as above. Then $\bigcup_{i \in \nat} (\Cf \cup \Bf)^{i} = \ccomp{S}$.
\end{Teo}
\indent The idea of the proof is as follows. Let $\zeta$ be the curve depicted in Figure \ref{OriginalChainv2}; we define the set $\Gf = \{\alpha_{0}, \ldots, \alpha_{2g-1}, \zeta\}$. Note that Humphries and Lickorish proved that the Dehn twists along the elements of $\Gf$ generate $\Mod{S}$ (see \cite{Hump}). Also recall that an essential curve $\alpha$ on $S$ is separating if $S \backslash \{\alpha\}$ is disconnected, and it is called nonseparating otherwise.\\
\indent First we prove that the image of $\Cf \cup \Bf$ under the Dehn twist along any element of $\Gf$ is contained in $(\Cf \cup \Bf)^{4}$. Afterwards we note that any nonseparating curve in $\ccomp{S}$ is the image of an element in $\Gf$ under an orientation preserving mapping class, and thus is contained in $(\Cf \cup \Bf)^{k}$ for some $k$. Finally, we show that every separating curve in $\ccomp{S}$ is uniquely determined by some finite subset of nonseparating curves, and thus also lies in $(\Cf \cup \Bf)^{k}$ for some $k$.\\
\indent Before passing to the proof of Theorem \ref{Thm2}, we give the necessary notation and state a technical lemma.\\
\indent Let $\alpha, \beta \in \ccomp{S}$ and $A, B \subset \ccomp{S}$. We denote by $\tau_{\alpha}(\beta)$ the right Dehn twist of $\beta$ along $\alpha$, $\tau_{\alpha}(B) = \bigcup_{\gamma \in B} \{\tau_{\alpha}(\gamma)\}$ and $\tau_{A}(B) = \bigcup_{\gamma \in A} \tau_{\gamma}(B)$. Observe that if $\alpha$ and $\beta$ are such that $i(\alpha,\beta) = 1$, we have:
\begin{equation}\label{taupm} 
 \tau_{\alpha}(\beta) = \tau_{\beta}^{-1}(\alpha) \hspace{1cm} \tau_{\alpha}^{-1}(\beta) = \tau_{\beta}(\alpha);
\end{equation}
\indent See Proposition 3.9 in \cite{Ara2} or Figure \ref{DehnTwistEqFig} for a proof.\\
\begin{figure}
 \begin{center}
  \resizebox{9cm}{!}{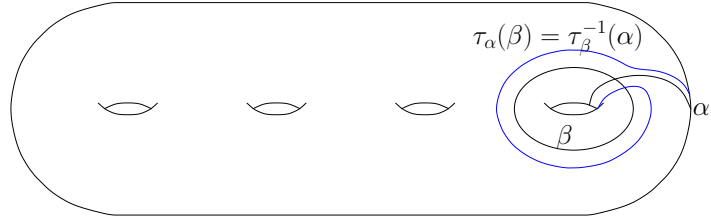}\caption{The curves $\alpha$ and $\beta$ in black, with the curve $\tau_{\alpha}(\beta)$ in blue.}\label{DehnTwistEqFig}
 \end{center}
\end{figure}
\indent The key technical lemma for the proof of Theorem \ref{Thm2} is the following.
\begin{Lema}\label{CUBen4}
 $\tau_{\Gf}^{\pm 1}(\Cf \cup \Bf) \subset (\Cf \cup \Bf)^{4}$.
\end{Lema}
\indent Note that, as was mentioned in the Introduction, if $\beta = \langle B \rangle$ we have that for any $h \in \EMod{S}$, $h(\beta) = \langle h(B) \rangle$. This allows the iterated use of the lemma.\\
\indent Assuming this lemma (which we prove in the following subsections) we embark on the proof of Theorem \ref{Thm2}.
\begin{proof}[\textbf{Proof of Theorem \ref{Thm2}}]
\indent Let $\gamma$ be a nonseparating curve and $\alpha \in \Gf$. There exists an orientation preserving mapping class $h \in \Mod{S}$ such that $\gamma = h(\alpha)$. As was mentioned above, the Dehn twists along the elements of $\Gf$ generate $\Mod{S}$. Thus, for some $\gamma_{1}, \ldots, \gamma_{m} \in \Gf$ and some $n_{1}, \ldots, n_{m} \in \mathbb{Z}$ we have that $\gamma = \tau_{\gamma_{1}}^{n_{1}} \circ \cdots \circ \tau_{\gamma_{m}}^{n_{m}}(\alpha)$. By an inductive use of Lemma \ref{CUBen4}, we have that $\gamma \in (\Cf \cup \Bf)^{4(|n_{1}| + \ldots + |n_{m}|)}$. Hence, every nonseparating curve is an element of $\bigcup_{i \in \nat} (\Cf \cup \Bf)^{i}$.\\
\indent Let $\gamma$ be a separating curve. Note that up to homeomorphism there exist only a finite number of separating curves. Moreover, as can be seen in Figure \ref{ChainsProofFig}, every such curve can be uniquely determined by a pair of chains of cardinalities $2g^{\prime}$ and $2g^{\biprime}$, where $g^{\prime}$ and $g^{\biprime}$ are the genera of the connected components of $S \backslash \{\gamma\}$. Then, there exist chains $C_{1}$ and $C_{2}$ such that $\gamma = \langle C_{1} \cup C_{2}\rangle$. By the previous case, $C_{1} \cup C_{2} \subset (\Cf \cup \Bf)^{k}$ for some $k \in \nat$; thus $\gamma \in (\Cf \cup \Bf)^{k+1}$. Therefore $\ccomp{S} = \bigcup_{i \in \nat} (\Cf \cup \Bf)^{i}$.
\end{proof}
\begin{figure}
 \begin{center}
  \resizebox{9cm}{!}{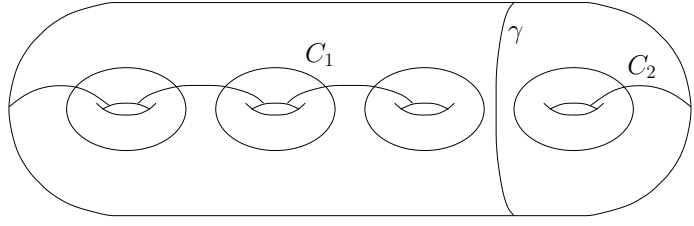}
 \end{center}\caption{A separating curve $\gamma$ and chains $C_{1}$ and $C_{2}$ that uniquely determine it.}\label{ChainsProofFig}
\end{figure}
\indent As stated before, the rest of this section is dedicated to the proof of Lemma \ref{CUBen4}, which (using that $\zeta \in \Bf$) is divided as follows:\\
\textbf{Claim 1:} $\tau_{\Cf}^{\pm 1} (\Cf) \subset (\Cf \cup \Bf)^{2}$.\\
\textbf{Claim 2:} $\tau_{\Cf}^{\pm 1} (\Bf) \cup \tau_{\Bf}^{\pm 1}(\Cf) \subset (\Cf \cup \Bf)^{3}$.\\
\textbf{Claim 3:} $\tau_{\zeta}^{\pm 1}(\Bf) \subset (\Cf \cup \Bf)^{4}$.\\[0.3cm]
\indent Note that since we only need to prove the lemma for Dehn twists along elements of $\Gf$, we only need to prove Claim 3 for $\zeta$.\\
\indent Before going further, we introduce the notation  used in the proofs of said claims.\\
\indent Let $\Cf^{\prime} = \{\gamma_{0}, \ldots, \gamma_{2g+1}\}$ be a maximal closed chain in $S$. The sets $\Cf^{\prime}_{o} = \{\gamma_{i} \in \Cf^{\prime} : i$ is odd$\}$ and $\Cf^{\prime}_{e} = \{\gamma_{i} \in \Cf^{\prime} : i$ is even$\}$, satisfy that $S \backslash \Cf^{\prime}_{e}$ and $S \backslash \Cf^{\prime}_{o}$ have two connected components, each homeomorphic to $S_{0,g+1}$. We denote by $S_{e}^{+}$ and $S_{e}^{-}$ the connected components of $S \backslash \Cf^{\prime}_{e}$, and by $S_{o}^{+}$ and $S_{o}^{-}$ the connected components of $S \backslash \Cf^{\prime}_{o}$. See Figure \ref{Beven} for an example. Let $1 \leq k \leq g-1$, and $\{\gamma_{i}, \ldots, \gamma_{i+2k}\}$ (with the indices modulo $2g+2$) be a subchain of $\Cf^{\prime}$. We denote by $[\gamma_{i}, \ldots, \gamma_{i+2k}]^{+}$ the curve in the associated bounding pair that is contained in either $S_{o}^{+}$ or $S_{e}^{+}$. Analogously, we denote by $[\gamma_{i}, \ldots, \gamma_{i+2k}]^{-}$ the curve in the associated bounding pair contained in either $S_{o}^{-}$ or $S_{e}^{-}$.\\
\begin{figure}[h]
\begin{center}
 \includegraphics[width=9cm]{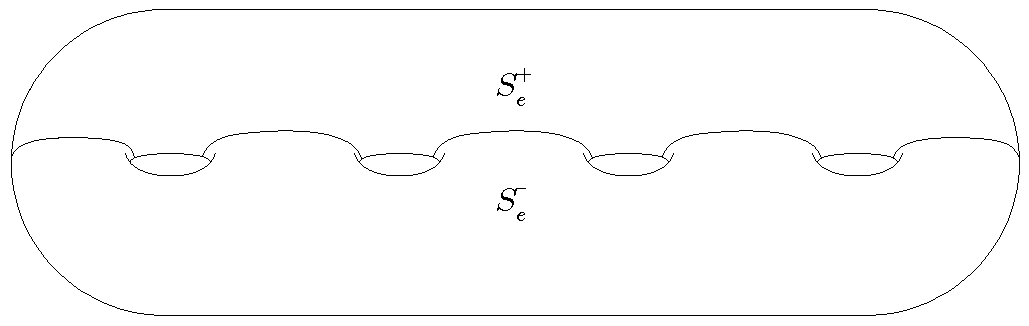} \caption{The set $\Cf^{\prime}_{e}$ and the corresponding $S_{e}^{+}$ and $S_{e}^{-}$.}\label{Beven}
\end{center}
\end{figure}
\begin{Rem}
 Note that according to this notation, $\zeta = [\alpha_{0}, \alpha_{1}, \alpha_{2}]^{-}$.
\end{Rem}
\indent We partition the set $\Bf$ into $\Bf_{o}^{+}$, $\Bf_{o}^{-}$, $\Bf_{e}^{+}$ and $\Bf_{e}^{-}$, depending on whether $\beta \in \Bf$ is contained in $S_{o}^{+}$, $S_{o}^{-}$, $S_{e}^{+}$ or $S_{e}^{-}$ respectively. We write $\Bf^{+} = \Bf_{o}^{+} \cup \Bf_{e}^{+}$ and $\Bf^{-} = \Bf_{o}^{-} \cup \Bf_{e}^{-}$.
\subsection{Proof of Claim 1: $\tau_{\Cf}^{\pm 1}(\Cf) \subset (\Cf \cup \Bf)^{2}$}\label{subsec4-2}
\indent To prove the claim, we start with a pair of particular curves and we show that is enough to prove the claim via the action of a particular subgroup of $\Mod{S}$.\\
\indent The following lemma is heavily based on Lemma 5.3 in \cite{Ara2}. However, its proof has been modified to emphasize the arguments that are used to obtain a more general result which is repeatedly used in the following subsections.
\begin{Lema}\label{DehnlemmaAra}
 $\tau_{\alpha_{2g}}^{\pm 1}(\alpha_{2g-1}) \in (\Cf \cup \Bf)^{2}$.
\end{Lema}
\begin{proof}
\indent Using the set $$C_{+} = \{ \alpha_{2g+1}, \alpha_{1}, \alpha_{2}, \ldots, \alpha_{2g-4}, \alpha_{2g-2}, [\alpha_{2g-3},\alpha_{2g-2},\alpha_{2g-1}]^{+}, [\alpha_{2},\ldots,\alpha_{2g-2}]^{+}\},$$ we obtain the curve $\gamma_{+} \in (\Cf \cup \Bf)^{1}$ as the curve uniquely determined by $C_{+}$, see Figure \ref{Dehntwists}. Then, letting $$C_{+}^{\prime} = \{\alpha_{0}, \ldots, \alpha_{2g-3}, [\alpha_{2g-2},\alpha_{2g-1},\alpha_{2g}]^{+}, [\alpha_{2g-2},\alpha_{2g-1},\alpha_{2g}]^{-}, \gamma_{+}\},$$ we have that $\tau_{\alpha_{2g}}(\alpha_{2g-1}) = \langle C_{+}^{\prime}\rangle \in (\Cf \cup \Bf)^{2}$.\\
\indent Analogously we have that $\tau_{\alpha_{2g}}^{-1}(\alpha_{2g-1}) = \langle C_{-}^{\prime}\rangle \in (\Cf \cup \Bf)^{2}$. See \cite{Thesis} for more details. Therefore $\tau_{\alpha_{2g}}^{\pm 1}(\alpha_{2g-1}) \in (\Cf \cup \Bf)^{2}$.
\end{proof}
\begin{figure}[h]
\begin{center}
 \includegraphics[width=7cm]{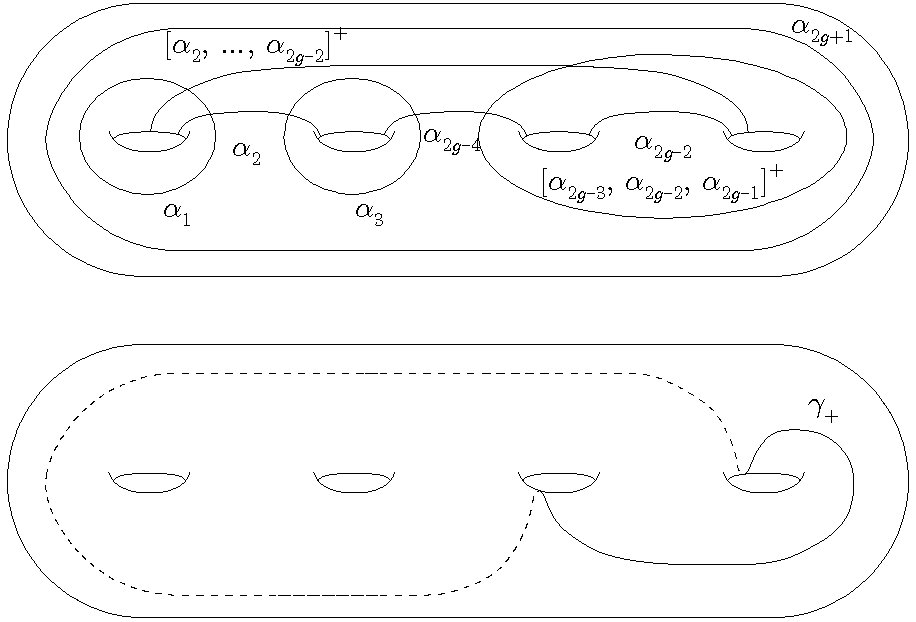} \includegraphics[width=7cm]{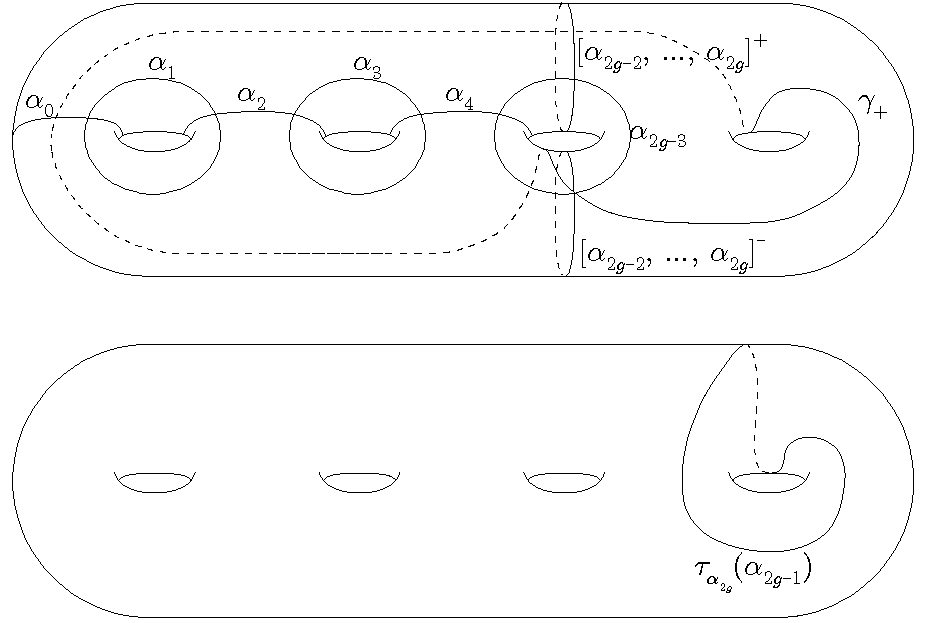}\caption{Above, the sets $C_{+}$ (left) and $C_{+}^{\prime}$ (right). Below, the curves $\gamma_{+}$ (left) and $\tau_{\alpha_{2g}}(\alpha_{2g-1})$ (right), uniquely determined by the sets $C_{+}$ and $C_{+}^{\prime}$ respectively.}\label{Dehntwists}
\end{center}
\end{figure}
\indent Let $H_{\Cf} < \EMod{S}$ be the setwise stabilizer of $\Cf$.
\begin{Rem}\label{reminvolution}
 Observe that $\Hcf(\Bf) = \Bf$, and for all $h \in \Hcf$ we have that $h(\Bf^{+}), h(\Bf^{-}) \in \{\Bf^{+},\Bf^{-}\}$. Moreover, $\Hcf$ can be partitioned as $\Hcf = \Hcp \sqcup \iota\Hcp$, where $\Hcp$ is the subgroup such that $\Hcp(\Bf^{+}) = \Bf^{+}$, and $\iota$ the hyperelliptic involution (which exchanges $S_{o}^{+}$ (resp. $S_{e}^{+}$) and $S_{o}^{-}$(resp. $S_{e}^{-}$)). Also note that $\Hcp$ \emph{acts transitively} on $\Cf$.
\end{Rem}
\begin{Lema}\label{translema}
 Let $h \in H_{\Cf}$, $k \in \nat$ and $\gamma \in (\Cf \cup \Bf)^{k}$. Then $h(\gamma) \in (\Cf \cup \Bf)^{k}$.
\end{Lema}
\begin{proof}
 We proceed by induction; if $k = 0$, we obtain the result by construction. If $k \geq 1$ let $\gamma \in (\Cf \cup \Bf)^{k} \backslash (\Cf \cup \Bf)^{k-1}$, as such $\gamma = \langle C_{0} \rangle$ with $C_{0} \subset (\Cf \cup \Bf)^{k-1}$; then $h(\gamma) = \langle h(C_{0})\rangle$, but by induction $h(C_{0}) \subset (\Cf \cup \Bf)^{k-1}$, thus $h(\gamma) \in (\Cf \cup \Bf)^{k}$.
\end{proof}
\indent Armed with Lemma \ref{translema}, we are ready to prove Claim 1.\\
\begin{proof}[Proof of Claim 1] Let $\alpha_{i}, \alpha_{j} \in \Cf$ with $i \neq j$. We want to prove that $\tau_{\alpha_{i}}^{\pm 1}(\alpha_{j}) \in (\Cf \cup \Bf)^{2}$. If $|i -j| > 1$ (modulo $2g+2$), then the curves are disjoint and we have that $\tau_{\alpha_{i}}^{\pm 1}(\alpha_{j}) = \alpha_{j} \in \Cf$. Suppose then that $|i-j| = 1$. There exists an element $h \in \Hcp$ such that either $h(\alpha_{2g}) = \alpha_{i}$ and $h(\alpha_{2g-1}) = \alpha_{j}$ if $i = j+1$, or $h(\alpha_{2g}) = \alpha_{j}$ and $h(\alpha_{2g-1}) = \alpha_{i}$ if $j = i+1$. Repeating the procedure of the proof of Lemma \ref{DehnlemmaAra}, precomposing by $h$ and using Lemma \ref{translema}, we obtain that $\tau_{\alpha_{i}}^{\pm 1}(\alpha_{j}) \in (\Cf \cup \Bf)^{2}$.
\end{proof}
\indent This finishes the proof of Claim 1. However, the proofs of Lemma \ref{DehnlemmaAra} and Claim 1 give us a slightly more general result, which is often used in the rest of this section. Its objective is to reduce the problems posed in the following claims when finding convenient maximal closed chains and showing that particular curves dependent on said chains are uniquely determined by elements in the expansions of $\Cf \cup \Bf$.
\begin{Lema}\label{Dehnlemma}
 Let $\{\gamma_{0}, \ldots, \gamma_{2g+1}\}$ be a maximal closed chain in $S$ that is contained in $(\Cf \cup \Bf)^{k}$ for some $k \in \nat$. If $[\gamma_{2g-3}, \gamma_{2g-2}, \gamma_{2g-1}]^{+}, [\gamma_{2}, \ldots, \gamma_{2g-2}]^{\pm} \in (\Cf \cup \Bf)^{k+1}$ and $[\gamma_{2g-2},\gamma_{2g-1},\gamma_{2g}]^{\pm} \in (\Cf \cup \Bf)^{k+2}$, then $\tau_{\gamma_{2g}}^{\pm 1}(\gamma_{2g-1}) \in (\Cf \cup \Bf)^{k+3}$. Moreover, $\tau_{f(\gamma_{2g})}^{\pm 1}(f(\gamma_{2g-1})) \in (\Cf \cup \Bf)^{k+3}$ for any $f \in \Hcp$.
\end{Lema}
\begin{proof}
 Given that \textit{up to the action of} $\Mod{S}$, $\Cf$ is the only maximal closed chain, there exists $h \in \Mod{S}$ such that $h(\alpha_{i}) = \gamma_{i}$ for $i \in \{0, \ldots, 2g+1\}$. Then, we repeat the procedure of the proof of Lemma \ref{DehnlemmaAra} precomposing by $h$ and using Lemma \ref{translema}, getting $\tau_{\gamma_{2g}}^{\pm 1}(\gamma_{2g-1}) = \tau_{h(\alpha_{2g})}^{\pm 1}(h(\alpha_{2g-1})) \in (\Cf \cup \Bf)^{k+3}$.\\
 \indent Let $f \in \Hcp$. Using Lemma \ref{translema} we can apply the result above to $fh(\Cf)$ and we get that $\tau_{f(\gamma_{2g})}^{\pm 1}(f(\gamma_{2g-1})) \in (\Cf \cup \Bf)^{k+3}$.
\end{proof}
\subsection{Proof of Claim 2: $\tau_{\Cf}^{\pm 1}(\Bf) \cup \tau_{\Bf}^{\pm 1}(\Cf) \subset (\Cf \cup \Bf)^{3}$}\label{subsec4-3}
\indent As in Subsection \ref{subsec4-2}, we first note that letting $\alpha \in \Cf$ and $\beta \in \Bf$, if $\alpha$ and $\beta$ are disjoint, there is nothing to prove and so we assume $i(\alpha,\beta) \neq 0$. By construction we then have that $i(\alpha,\beta) = 1$. In part 1, we first establish the claim for a particular family of maximal closed chains that verify the conditions of Lemma \ref{Dehnlemma}, proving that $\tau_{[\alpha_{0},\ldots,\alpha_{2l}]^{-}}^{\pm 1}(\alpha_{2l+1}) \in (\Cf \cup \Bf)^{3}$ for all $1 \leq l \leq g-2$; then, via the action of $\Hcp$, we prove that $\tau_{\Bf^{-}}^{\pm 1}(\Cf) \cup \tau_{\Cf}^{\pm 1}(\Bf^{-}) \subset (\Cf \cup \Bf)^{3}$. In part 2, we finish the proof via the action of the hyperelliptic involution $\fun{\iota}{S}{S}$ mentioned in Remark \ref{reminvolution}.\\[0.5cm]
\textbf{Part 1:} Let $1 \leq l \leq g-2$. We define the following maximal closed chain: $$\Cf_{l} = \{\alpha_{1}, \alpha_{2}, \ldots, \alpha_{2l}, \alpha_{2l+1}, [\alpha_{0}, \ldots, \alpha_{2l}]^{-}, \alpha_{2g+1}, \alpha_{2g}, \alpha_{2g-1}, \ldots, \alpha_{2l+3}, [\alpha_{2}, \ldots, \alpha_{2l+2}]^{+}\}.$$
We refer the reader to Figure \ref{SecondChain} for an example of such a maximal closed chain.
\begin{figure}[h]
\begin{center}
 \resizebox{11cm}{!}{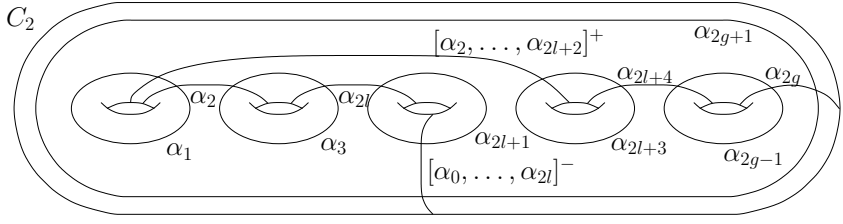} \caption{$\Cf_{2}$ for the genus $5$ surface and the first reordering.}\label{SecondChain}
\end{center}
\end{figure}\\
\indent We now prove, using Lemma \ref{Dehnlemma}, that $\tau_{\Bf^{-}}^{\pm 1}(\Cf) \cup \tau_{\Cf}^{\pm 1}(\Bf^{-}) \subset (\Cf \cup \Bf)^{3}$.\\
\indent In order to facilitate the use of Lemma \ref{Dehnlemma}, we cyclically reorder the elements of $\Cf_{l}$ as follows: $$\gamma_{0} = \alpha_{2g}, \hspace{0.15cm} \gamma_{1} = \alpha_{2g-1}, \hspace{0.15cm} \ldots, \hspace{0.15cm} \gamma_{2g-1} = \alpha_{2l+1}, \hspace{0.15cm} \gamma_{2g} = [\alpha_{0}, \hspace{0.15cm} \ldots, \hspace{0.15cm} \alpha_{2l}]^{-}, \hspace{0.15cm} \mathrm{and} \hspace{0.15cm} \gamma_{2g+1} = \alpha_{2g+1}.$$ Again, see Figure \ref{SecondChain} for an example.\\
\indent By inspection we can verify that $\Cf_{l}$ satisfies the conditions of Lemma \ref{Dehnlemma}, and thus $\tau_{\gamma_{2g}}^{\pm 1}(\gamma_{2g-1}) = \tau_{[\alpha_{0}, \ldots, \alpha_{2l}]^{-}}^{\pm 1}(\alpha_{2l+1}) \in (\Cf \cup \Bf)^{3}$, however for the sake of completeness we give a detailed account of which set of curves uniquely determines the needed curves.\\
For $1 \leq l \leq g-2$ we have:
\begin{center}
\begin{tabular}{rcl}
 $[\gamma_{2g-3},\gamma_{2g-2},\gamma_{2g-1}]^{+}$ & $=$ & $[\alpha_{2l-1},\alpha_{2l},\alpha_{2l+1}]^{+}$\\
 $[\gamma_{2},\ldots,\gamma_{2g-2}]^{+}$ & $=$ & $\langle \alpha_{2g-2},\ldots,\alpha_{2l+3},[\alpha_{2}, \ldots, \alpha_{2l+2}]^{+}, \alpha_{1},\alpha_{2},\ldots,\alpha_{2l},$\\
 & & $[\alpha_{0},\ldots,\alpha_{2l}]^{-}, \alpha_{2g+1}, \alpha_{2g}, [\alpha_{2l+2}, \ldots, \alpha_{2g-2}]^{-}\rangle$\\
 $[\gamma_{2}, \ldots, \gamma_{2g-2}]^{-}$ & $=$ & $[\alpha_{2l+2}, \ldots, \alpha_{2g-2}]^{-}$
\end{tabular}
\end{center}
\indent In the case of $l = 1$ we have:
\begin{center}
\begin{tabular}{rcl}
 $[\gamma_{2g-2},\gamma_{2g-1},\gamma_{2g}]^{+}$ & $=$ & $\alpha_{0}$\\
 $[\gamma_{2g-2},\gamma_{2g-1},\gamma_{2g}]^{-}$ & $=$ & $\langle \alpha_{2},\alpha_{3},[\alpha_{0},\alpha_{1},\alpha_{2}]^{-},\alpha_{0},[\alpha_{2},\alpha_{3},\alpha_{4}]^{+},\alpha_{5},\alpha_{6}, \ldots, \alpha_{2g}\rangle$
\end{tabular}
\end{center}
\indent In the cases with $l > 1$ we have:
\begin{center}
\begin{tabular}{rcl}
 $[\gamma_{2g-2},\gamma_{2g-1},\gamma_{2g}]^{+}$ & $=$ & $[\alpha_{0},\ldots,\alpha_{2l-2}]^{-}$\\
 $[\gamma_{2g-2},\gamma_{2g-1},\gamma_{2g}]^{-}$ & $=$ & $\langle \alpha_{2l},\alpha_{2l+1},[\alpha_{0},\ldots,\alpha_{2l}]^{-},[\alpha_{0},\ldots,\alpha_{2l-2}]^{-}, \alpha_{2l-2}, \alpha_{2l-3}, \ldots, \alpha_{1},$\\
 & & $[\alpha_{2},\ldots,\alpha_{2l+2}]^{+}, \alpha_{2l+3},\alpha_{2l+4},\ldots, \alpha_{2g}\rangle$
\end{tabular}
\end{center}
\indent Letting $l$ vary from $1$ to $g-2$, and applying Lemma \ref{Dehnlemma}, we have that $\tau_{\gamma_{2g}}^{\pm 1}(\gamma_{2g-1}) = \tau_{[\alpha_{0},\ldots,\alpha_{2l}]^{-}}^{\pm 1}(\alpha_{2l+1}) \in (\Cf \cup \Bf)^{3}$ for all $1 \leq l \leq g-2$.\\
\indent Now, using the fact that $\Hcp < \EMod{S}$ acts transitively on $\Cf$, we have as a consequence that it also acts transitively on each of the sets $\{[\alpha_{i}, \ldots, \alpha_{i+2l}]^{-}: 0 \leq i \leq 2g+1\}$ for $1 \leq l \leq g-2$. This implies that given $[\alpha_{i}, \ldots, \alpha_{i+2l}]^{-} \in \Bf^{-}$, there exists $h \in \Hcp$ such that $h([\alpha_{i}, \ldots, \alpha_{i+2l}]^{-}) = [\alpha_{0},\ldots,\alpha_{2l}]^{-}$. Thus, by Lemmas \ref{translema} and \ref{Dehnlemma}, we have then that $\tau_{\Bf^{-}}^{\pm 1}(\Cf) \subset (\Cf \cup \Bf)^{3}$, and by Equation \ref{taupm}, we obtain that $\tau_{\Bf^{-}}^{\pm 1}(\Cf) \cup \tau_{\Cf}^{\pm 1}(\Bf^{-}) \subset (\Cf \cup \Bf)^{3}$.\\[0.3cm]
\textbf{Part 2} To prove the rest of the cases, recall that the hyperelliptic involution $\iota$ is an element of $\stab{\Cf}$ and $\iota([\alpha_{i}, \ldots, \alpha_{i+2k}]^{+}) = [\alpha_{i}, \ldots, \alpha_{i+2k}]^{-}$ for all $\{\alpha_{i}, \ldots, \alpha_{i+2k}\} \subset \Cf$. Given that (as was shown in part 1) for all $1 \leq l \leq g-2$, the families of maximal closed chains $\Hcp(\Cf_{l})$ satisfy the conditions of Lemma \ref{Dehnlemma}, we have that Lemma \ref{translema} yields that the same is true for the maximal closed chains $\iota\Hcp(\Cf_{l})$. Therefore $$\tau_{\iota\Hcp(\Bf^{-})}^{\pm 1}(\iota\Hcp(\Cf)) \cup \tau_{\iota\Hcp(\Cf)}^{\pm 1}(\iota\Hcp(\Bf^{-})) = \tau_{\Bf^{+}}^{\pm 1}(\Cf) \cup \tau_{\Cf}^{\pm 1}(\Bf^{+}) \subset (\Cf \cup \Bf)^{3},$$ as desired.
\subsection{Proof of Claim 3: $\tau_{\zeta}^{\pm 1}(\Bf) \subset (\Cf \cup \Bf)^{4}$}\label{Chap1Sec4}
\indent Recall $\zeta = [\alpha_{0},\alpha_{1},\alpha_{2}]^{-}$ and let $\gamma \in \Bf$. In the cases where $\zeta$ is disjoint from $\gamma$, we have $\tau_{\zeta}^{\pm 1}(\gamma) = \gamma \in \Cf \cup \Bf$. So we assume that $i(\gamma,\zeta) \neq 0$, which by construction implies $$\gamma \in \{[\alpha_{1}, \ldots, \alpha_{2k+1}]^{\pm},[\alpha_{3}, \ldots, \alpha_{2k+1}]^{\pm}, [\alpha_{2}, \ldots, \alpha_{2l}]^{-}: 1 \leq k \leq g-1, 2 \leq l \leq g-1\}.$$
\indent In these cases there exist subsets of $C_{0} \subset \Cf$ and $\{\beta_{0}\} \subset \Bf$, such that $\gamma = \langle C_{0} \cup \{\beta_{0}\} \rangle$ and $\beta_{0}$ is disjoint from $\zeta$. Note that $\tau_{\zeta}^{\pm 1}(C_{0}) \subset (\Cf \cup \Bf)^{3}$ by claim 2, and $\tau_{\zeta}^{\pm 1}(\beta_{0}) = \beta_{0} \in \Cf \cup \Bf$ by construction. Therefore $$\tau_{\zeta}^{\pm 1}(\gamma) = \tau_{\zeta}^{\pm 1}(\langle C_{0} \cup \beta_{0} \rangle) = \langle \tau_{\zeta}^{\pm 1}(C_{0}) \cup \tau_{\zeta}^{\pm 1}(\beta_{0}) \rangle \in (\Cf \cup \Bf)^{4}.$$\\
\indent For a more detailed account on $C_{0}$ and $\beta_{0}$ see \cite{Thesis}. For some examples see Figures \ref{LastClaimChap1Fig1} and \ref{LastClaimChap1Fig3}.
\begin{figure}[h]
 \begin{center}
  \resizebox{8cm}{!}{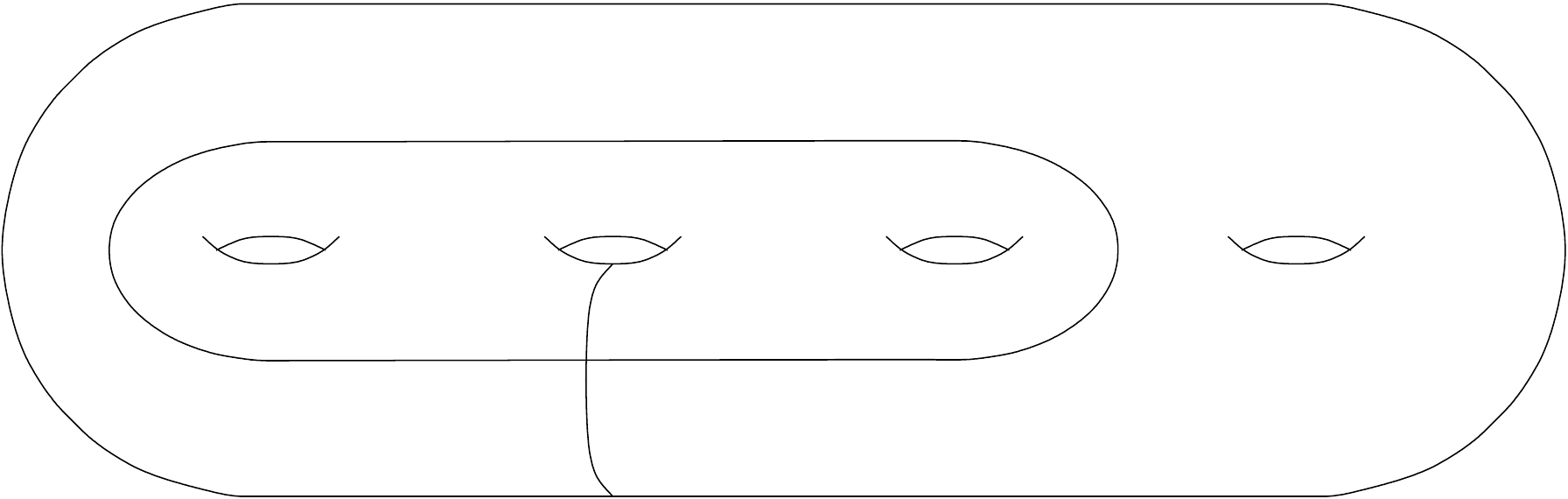}\\[0.3cm]
  \resizebox{8cm}{!}{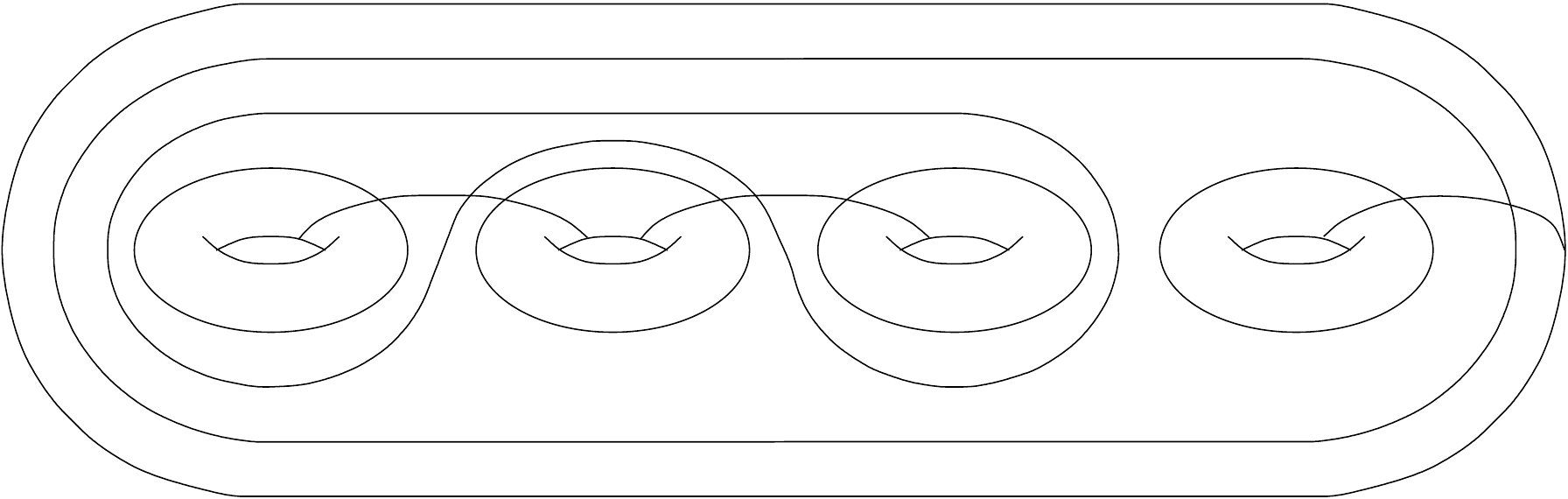} \caption{An example of $\gamma = [\alpha_{1}, \ldots, \alpha_{2k+1}]^{+}$, and the corresponding $C_{0}$ and $\beta_{0}$.}\label{LastClaimChap1Fig1}
 \end{center}
\end{figure}\\
\begin{figure}[h]
 \begin{center}
  \resizebox{8cm}{!}{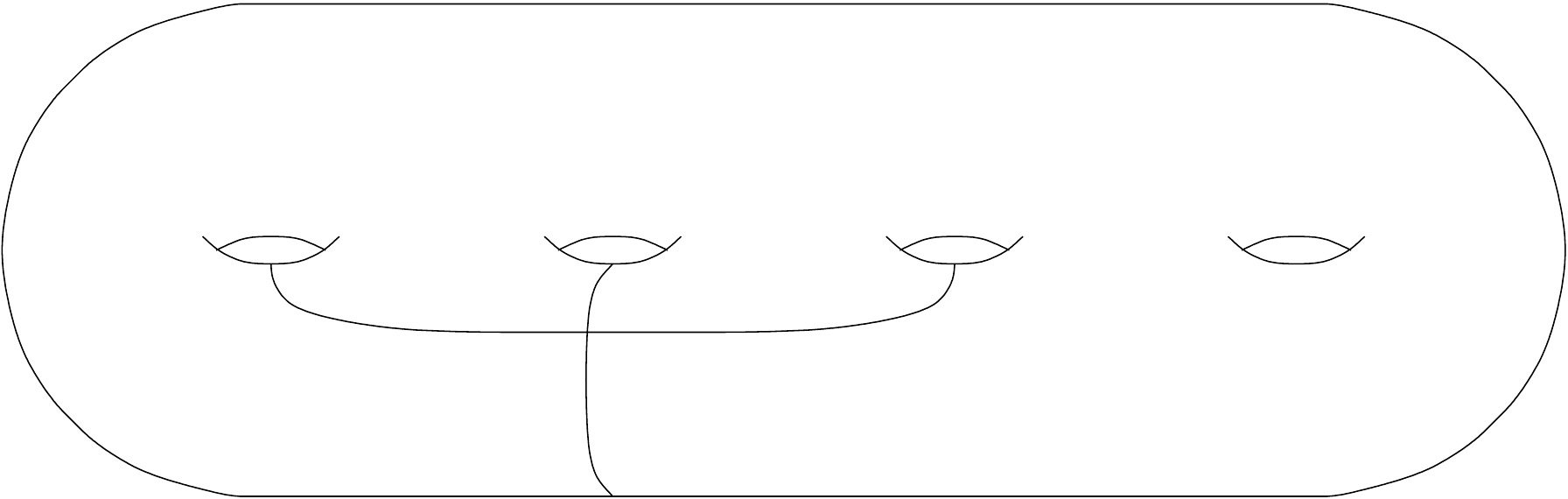}\\[0.3cm]
  \resizebox{8cm}{!}{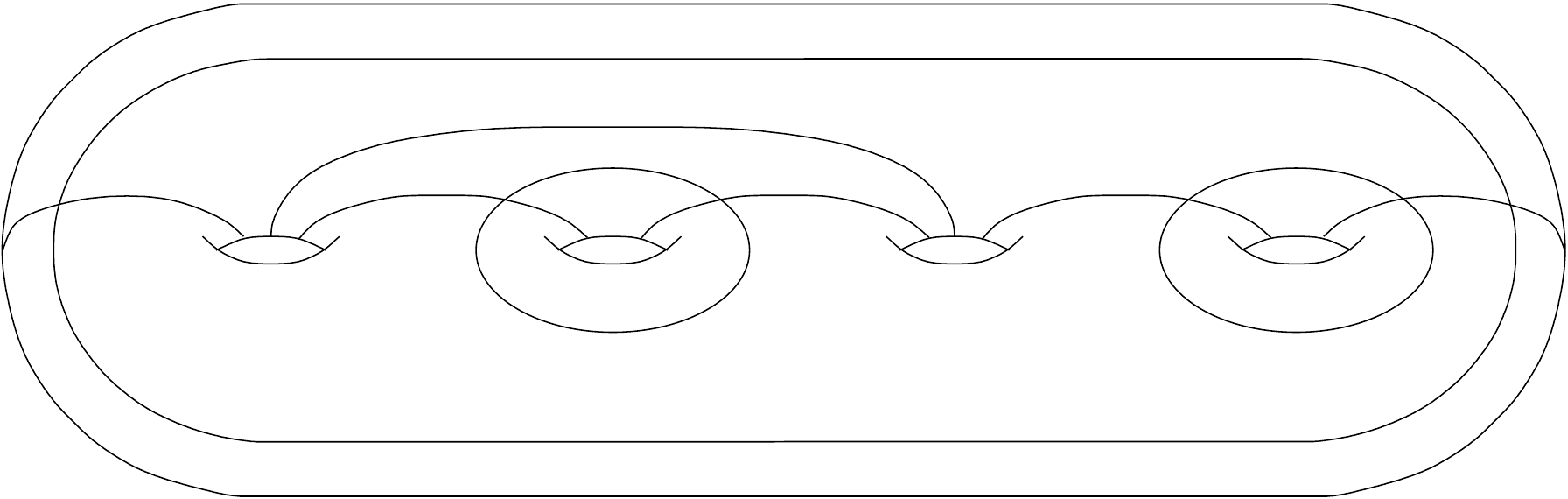} \caption{An example of $\gamma = [\alpha_{2}, \ldots, \alpha_{2k}]^{-}$, and the corresponding $C_{0}$ and $\beta_{0}$.}\label{LastClaimChap1Fig3}
 \end{center}
\end{figure}
\section{Exhaustion of $\ccomp{S}$ for punctured surfaces}\label{chap2}
\indent In this section, we suppose that $S = S_{g,n}$ with genus $g \geq 3$ and $n \geq 1$ punctures. In Subsection \ref{chap2sec1} we fix the principal set and give notation; in addition to a principal set of curves analogous to $\Cf$ and $\Bf$, we introduce some auxiliary curves to aid the exposition in Subsection \ref{subsec5-2}, we also prove they are in specific expansions of the principal set, and state and prove several technical propositions; in Subsection \ref{chap2sec3} we prove the main theorem, pending the proof of a technical lemma; Subsections \ref{chap2sec4}, \ref{chap2sec5}, \ref{chap2sec6}, \ref{chap2sec7}, \ref{chap2sec8} and \ref{chap2sec9} give the proofs of the claims for the technical lemma.
\subsection{Statement of Theorem \ref{Thm3}}\label{chap2sec1}
\indent The idea of the proof of the analogous result to Theorem \ref{Thm2} is the same as in the closed surface case. Using arguments similar to those of Theorem \ref{Thm2} we show that every nonseparating curve is in some expansion and then use that to prove the same for the separating curves.\\
\indent However, the presence of punctures induce several small but important changes, both in the principal set of curves that is used (which while analogous to the closed case, is not as symmetric and thus induces changes in the proofs), and in the manner auxiliary curves are used. For this reason, we first introduce the sets $\Cf$ and $\Bf_{0}$ whose union is the \textbf{principal set} and then we state in detail the theorem to prove.\\
\indent Let $\Cf_{0} = \{\alpha_{1}, \ldots, \alpha_{2g+1}\}$ be the chain depicted in Figure \ref{OriginalChainPunctured}, and $\Cf_{f} = \{\alpha_{0}^{0}, \alpha_{0}^{1}, \ldots, \alpha_{0}^{n}\}$ be the multicurve also depicted in Figure \ref{OriginalChainPunctured}, and $\Cf \ColonEqq \Cfn \cup \Cff$.
\begin{figure}[h]
 \begin{center}
  \resizebox{10cm}{!}{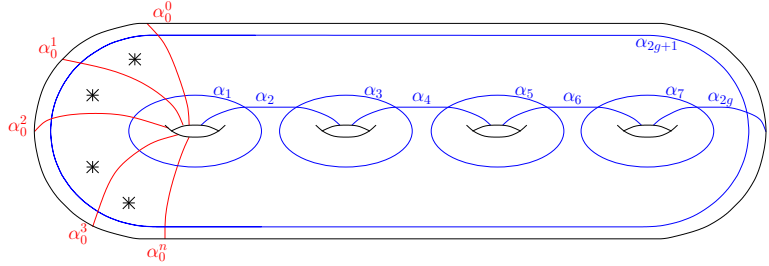} \caption{$\Cf_{0}$ in blue and $\Cf_{f}$ in red for $S_{5,4}$.}\label{OriginalChainPunctured}
 \end{center}
\end{figure}
\begin{Rem}
 Note that for $i \in \{0, \ldots, n-1\}$, we have that $S \backslash \{\alpha_{0}^{i},\alpha_{0}^{i+1}\}$ has a connected component homeomorphic to a thrice punctured sphere. Also note that for each $j \in \{0, \ldots, n\}$, the set $C_{j} = \{\alpha_{0}^{j}, \alpha_{1}, \ldots, \alpha_{2g+1}\}$ is a maximal closed chain.
\end{Rem}
\indent Adapting the notation used in the previous section, for $i \in \{0, \ldots, n\}$, let us consider the closed chain $C_{i} \subset \Cf$. Then we denote the curve $\alpha_{0}^{i}$ by $\alpha_{0}$ to simplify notation when it is understood that $\alpha_{0} \in C_{i}$. As such $C_{i}$ has the subsets: $C_{i(o)} = \{\alpha_{j} \in C_{i} : j$ is odd$\}$ and $C_{i(e)} = \{\alpha_{j} \in C_{i} : j$ is even$\}$. These subsets are such that:
\begin{itemize}
 \item $S \backslash C_{i(e)}$ has two connected components, $S_{i(e)}^{+} = S_{0,i+g+1}$ and $S_{i(e)}^{-} = S_{0, n-i+ g+1}$.
 \item $S \backslash C_{i(o)}$ has two connected components, $S_{i(o)}^{+} = S_{0,n+g+1}$ and $S_{i(o)}^{-} = S_{0,g+1}$.
\end{itemize}
\indent Recalling that the subindices are modulo $2g+2$, we denote by $[\alpha_{j}, \ldots, \alpha_{j + 2k}]^{+}$ for $0 < k < g-1$, the boundary component of a closed regular neighbourhood $N(\{\alpha_{j}, \ldots, \alpha_{j + 2k}\})$, that is contained in either $S_{i(o)}^{+}$ or in $S_{i(e)}^{+}$. Analogously, $[\alpha_{j}, \ldots, \alpha_{j + 2k}]^{-}$ denotes the boundary component of a closed regular neighbourhood $N(\{\alpha_{j}, \ldots, \alpha_{j + 2k}\})$, that is contained in either $S_{i(o)}^{-}$ or in $S_{i(e)}^{-}$. 
\begin{Rem}
 Note that this notation is the same as in the closed surface case for the set $\Bf$; however, when $0 \in \{j, \ldots, j+2k\}$ (modulo $2g+2$), the curves $[\alpha_{j}, \ldots, \alpha_{j+2k}]^{\pm}$ depend on the choice of $i \in \{0, \ldots, n\}$ (recall $\alpha_{0}$ stands for $\alpha_{0}^{i}$).
\end{Rem}
\indent Let $J = \{2l, \ldots, 2(l+k)\}$, for some $1 \leq k \leq g-1$, be a proper interval in the cyclic order modulo $2g+2$. Let also $\beta_{J}^{\pm} = [\alpha_{2l}, \ldots, \alpha_{2(l+k)}]^{\pm}$ (with $\alpha_{0} = \alpha_{0}^{1}$ when necessary). See Figure \ref{ExamplesBf0} for examples. We define $$\Bf_{0} \ColonEqq \{\beta_{J}^{\pm} : J = \{2l, \ldots, 2(l+k)\}, 1 \leq k \leq g-1\}.$$
\begin{figure}[h]
 \begin{center}
  \resizebox{10cm}{!}{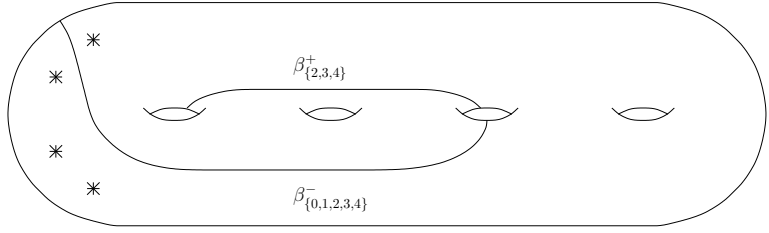}\caption{Examples of curves $\beta_{\{2,3,4\}}^{+}$ and $\beta_{\{0,1,2,3,4\}}^{-}$.}\label{ExamplesBf0}
 \end{center}
\end{figure}\\
\indent Now that we have the principal set of curves, we are able to state the punctured case version of Theorem \ref{TheoA}.
\begin{Teo}\label{Thm3}
 Let $S$ be an orientable surface with genus $g \geq 3$ and $n \geq 1$ punctures; let also $\Cf$ and $\Bf_{0}$ be defined as above. Then $\bigcup_{i \in \nat} (\Cf \cup \Bfn)^{i} = \ccomp{S}$.
\end{Teo}
\subsection{Auxiliary curves}\label{subsec5-2}
\indent We need for the proof some auxiliary curves and some technical results.\\[0.3cm]
\indent For $0 \leq i \leq n -2$, we define $$\eps{i}{i+2} \ColonEqq \langle\Cfn \cup (\Cff \backslash \{\alpha_{0}^{i+1}\})\rangle;$$ note that $\eps{i}{i+2} \in \Cf^{1}$; this can be seen using Figure \ref{OriginalChainPunctured} and removing $\alpha_{0}^{i+1}$ for the chosen $i$.\\
\indent For $0 \leq i < j \leq n$ with $j - i > 2$, we define the curve: $$\eps{i}{j} \ColonEqq \langle\Cfn \cup (\Cff \backslash \{\alpha_{0}^{k} : i < k < j\}) \cup \{\eps{k}{k+2} : i \leq k \leq j-2\}\rangle;$$ note that $\eps{i}{j} \in \Cf^{2}$.\\
\indent Then, we define the set: $$\Df \ColonEqq \{\eps{i}{j} : j-i > 1\} \subset \Cf^{2}.$$
\indent We now expand this set. Let $0 \leq i \leq j \leq n$. We define $$C^{i,j} \ColonEqq \{\alpha_{0}^{i}, \ldots, \alpha_{0}^{j}, \alpha_{1}, \ldots, \alpha_{2g+1}\}, \hspace{0.2cm} \mathrm{and} \hspace{0.2cm} E^{i,j} \ColonEqq \{\eps{k}{l} \in \Df : 0 \leq k < l \leq i, \hspace{0.3cm} \mathrm{or} \hspace{0.3cm} j \leq k < l \leq n\}.$$
\indent Note that $C^{i,i} = C_{i}$, and $E^{i,j} = \Df \backslash \{\eps{k}{l} : (\exists \gamma \in C^{i,j}) \hspace{0.3cm} i(\eps{k}{l}, \gamma) \neq 0\}$.\\
\indent Let $k \in \{1, \ldots, g\}$, and define (see Figure \ref{DefEpspfig1}) $$\epsp{k}{i}{j} \ColonEqq \langle (C^{i,j} \backslash \{\alpha_{2k}\}) \cup E^{i,j}\rangle;$$ for $k = 0$ we take $\epsp{0}{i}{j} \ColonEqq \eps{i}{j}$, and for $k = -1$ we take $\epsp{-1}{i}{j} = \epsp{g}{i}{j}$.
\begin{figure}[h]
 \begin{center}
  \resizebox{10cm}{!}{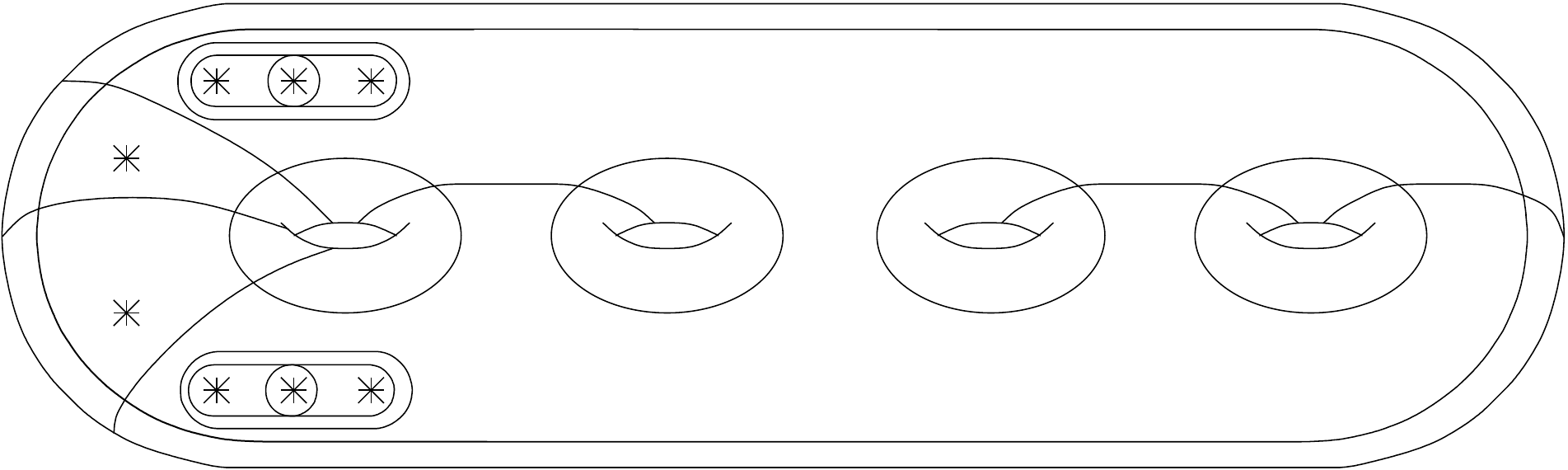}\\[0.3cm]
  \resizebox{10cm}{!}{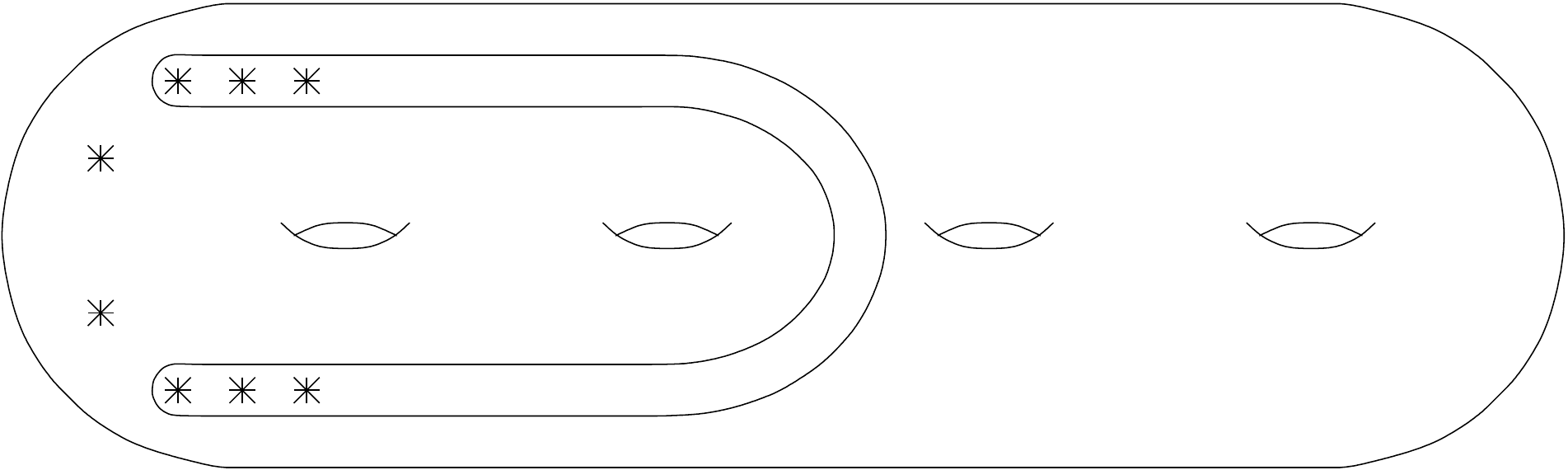} \caption{The curve $\epsp{k}{i}{j} \ColonEqq \langle (C^{i,j} \backslash \{\alpha_{2k}\}) \cup E^{i,j}\rangle$.} \label{DefEpspfig1}
 \end{center}
\end{figure}\\
\indent Finally we define $$\Ef \ColonEqq \Df \cup \{\epsp{k}{i}{j} : 0 \leq i \leq j \leq n, \hspace{0.15cm} k \in \{1, \ldots, g\}\}$$
\indent Note that $\Ef \subset \Cf^{3}$ by construction. See Figure \ref{MathfrakD} for examples of curves in $\Ef$.
\begin{figure}[h]
 \begin{center}
  \resizebox{10cm}{!}{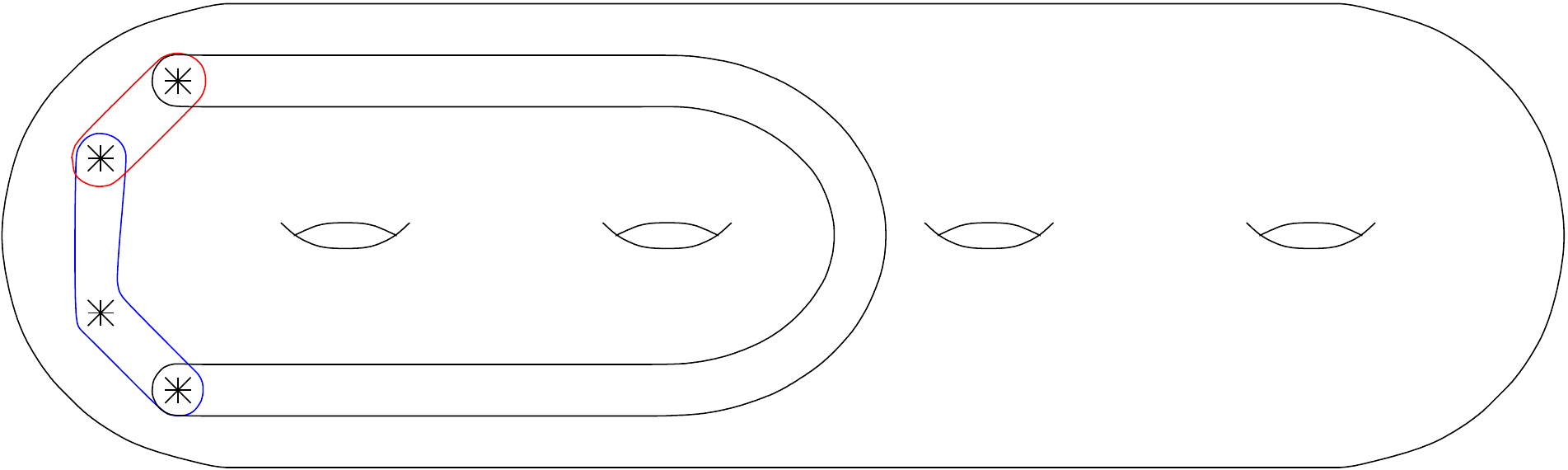} \caption{Examples of curves in $\Ef$.}\label{MathfrakD}
 \end{center}
\end{figure}
\begin{Rem}
 Note that the sets $\Df$ and $\Ef$ are only defined when $n \geq 2$. For this reason, if $n = 1$ we define $\Df = \Ef = \varnothing$.
\end{Rem}
\indent Now, in the case where $1 \leq j \leq n-1$. We define the following curves (see Figure \ref{ExamplesBetaT1} for examples): $$\beta_{\{0,1,2\}}^{j,+} \ColonEqq \langle (\Cff \backslash \{\alpha_{0}^{k} : k < j\}) \cup (\Cfn \backslash \{\alpha_{3}, \alpha_{2g+1}\}) \cup E^{j,n} \cup \{\epsp{k}{1}{n-1} : 2 \leq k \leq g\}\rangle;$$
$$\beta_{\{0,1,2\}}^{j,-} \ColonEqq \langle (\Cff \backslash \{\alpha_{0}^{k} : k > j\}) \cup (\Cfn \backslash \{\alpha_{3}, \alpha_{2g+1}\}) \cup E^{0,j} \cup \{\epsp{k}{1}{n-1} : 2 \leq k \leq g\}\rangle;$$
in the case where $j = n$, we define (see Figure \ref{DefBeta012nplusfig1}): $$\beta_{\{0,1,2\}}^{n,+} \ColonEqq \langle(\Cfn \backslash \{\alpha_{3}, \alpha_{2g+1}\}) \cup \{\alpha_{0}^{n}\} \cup \Df \cup \{\beta_{\{4, \ldots, 2g\}}^{\pm}\}\rangle;$$ $$\beta_{\{0,1,2\}}^{n,-} \ColonEqq \beta_{\{4, \ldots, 2g\}}^{-} \hspace{0.5cm}(\in \Bf_{0}).$$
\begin{figure}[h]
 \begin{center}
  \resizebox{10cm}{!}{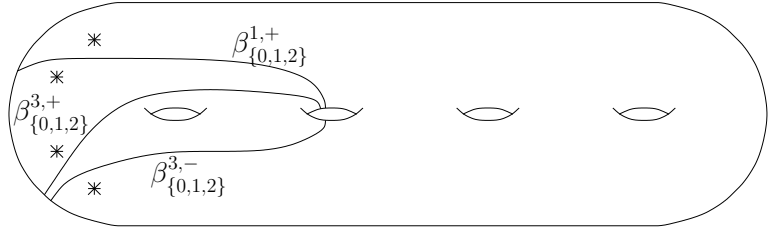} \caption{Examples of curves $\beta_{\{0,1,2\}}^{1,+}$, $\beta_{\{0,1,2\}}^{3,+}$ and $\beta_{\{0,1,2\}}^{3,-}$ in $S_{4,4}$.}\label{ExamplesBetaT1}
 \end{center}
\end{figure}
\begin{figure}[h]
 \begin{center}
  \resizebox{10cm}{!}{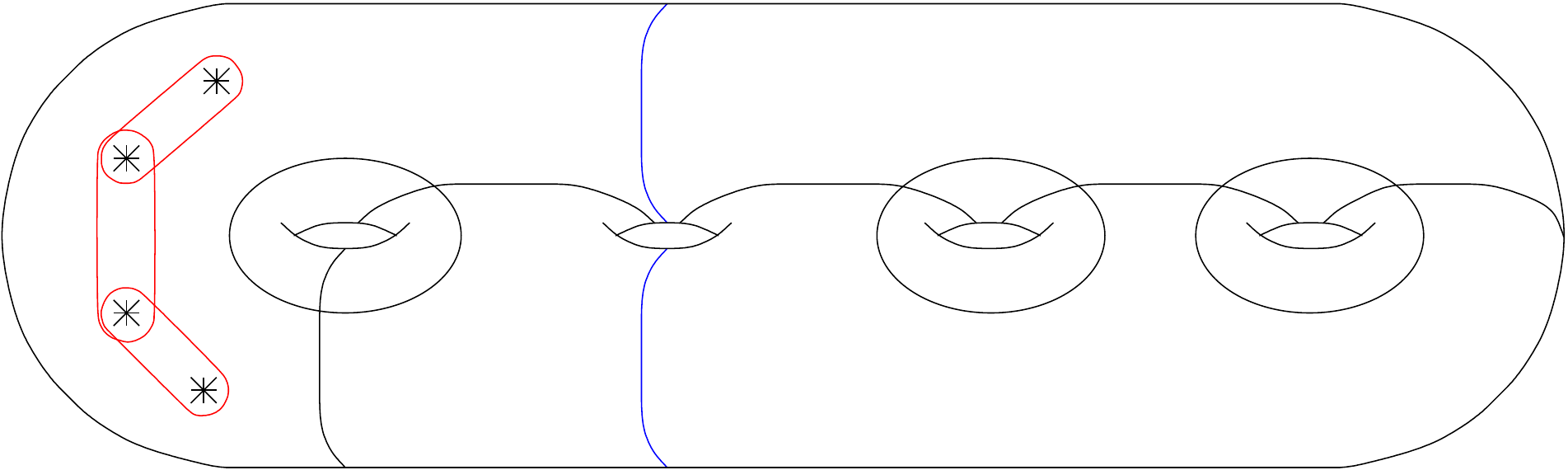}\\[0.3cm]
  \resizebox{10cm}{!}{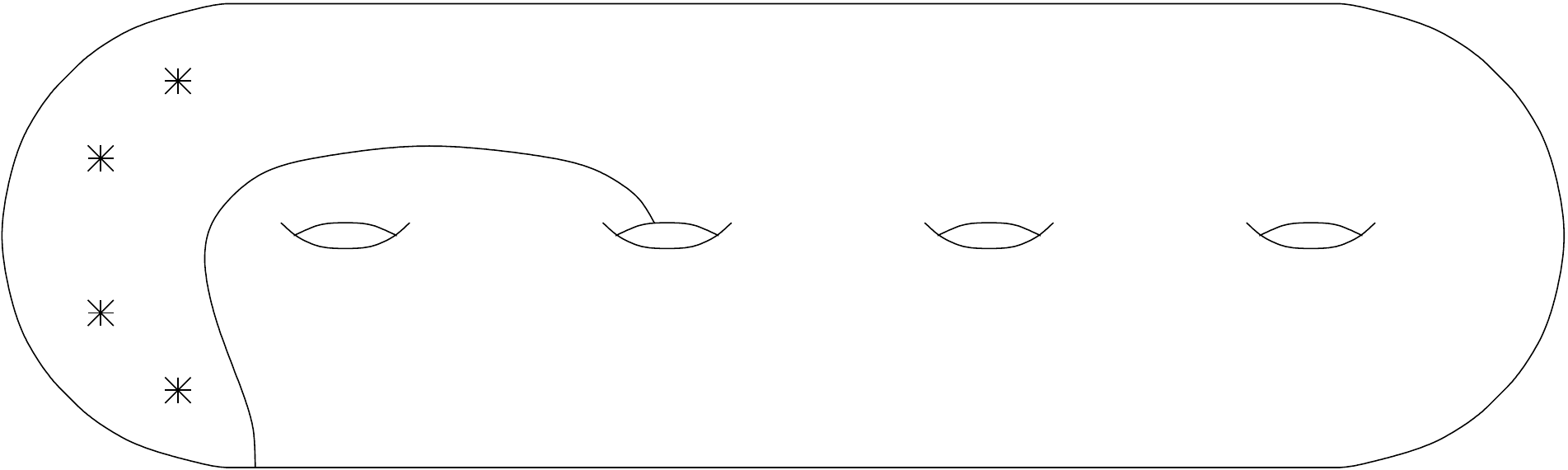} \caption{The curve $\beta_{\{0,1,2\}}^{4,+} \ColonEqq \langle(\Cfn \backslash \{\alpha_{3}, \alpha_{2g+1}\}) \cup \{\alpha_{0}^{n}\} \cup \Df \cup \{\beta_{\{4, \ldots, 2g\}}^{\pm}\}\rangle$.} \label{DefBeta012nplusfig1}
 \end{center}
\end{figure}\\
\indent Similarly, we define the following curves: in the case where $1 \leq j \leq n-1$, we define (see Figure \ref{ExamplesBetaT2} for examples): $$\beta_{\{2g,2g+1,0\}}^{j,+} \ColonEqq \langle(\Cfn \backslash \{\alpha_{1}, \alpha_{2g-1}\}) \cup (\Cff \backslash \{\alpha_{0}^{k} : k < j\}) \cup E^{j,n} \cup \{\epsp{k}{1}{n-1} : 1 \leq k \leq g-1\}\rangle;$$
$$\beta_{\{2g,2g+1,0\}}^{j,-} \ColonEqq \langle(\Cfn \backslash \{\alpha_{1},\alpha_{2g-1}\}) \cup (\Cff \backslash \{\alpha_{0}^{k} : k > j\}) \cup E^{0,j} \cup \{\epsp{k}{1}{n-1} : 1 \leq k \leq g-1\}\rangle;$$
in the case where $j = n$, we define (see Figure \ref{DefBeta2g2gplus10} for examples): $$\beta_{\{2g,2g+1,0\}}^{n,+} \ColonEqq \langle (\Cfn \backslash \{\alpha_{1}, \alpha_{2g-1}\}) \cup \{\alpha_{0}^{n}\} \cup \Df \cup \{\beta_{\{2, \ldots, 2g-2\}}^{\pm}\}\rangle$$
$$\beta_{\{2g,2g+1,0\}}^{n,-} \ColonEqq \beta_{\{2, \ldots, 2g-2\}}^{-}.$$
\begin{figure}[h]
 \begin{center}
  \resizebox{10cm}{!}{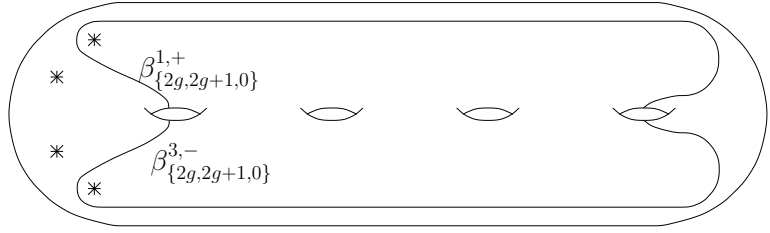} \caption{Examples of curves $\beta_{\{2g,2g+1,0\}}^{1,+}$ and $\beta_{\{2g,2g+1,0\}}^{3,-}$ in $S_{4,4}$.}\label{ExamplesBetaT2}
 \end{center}
\end{figure}
\begin{figure}[h]
 \begin{center}
  \resizebox{10cm}{!}{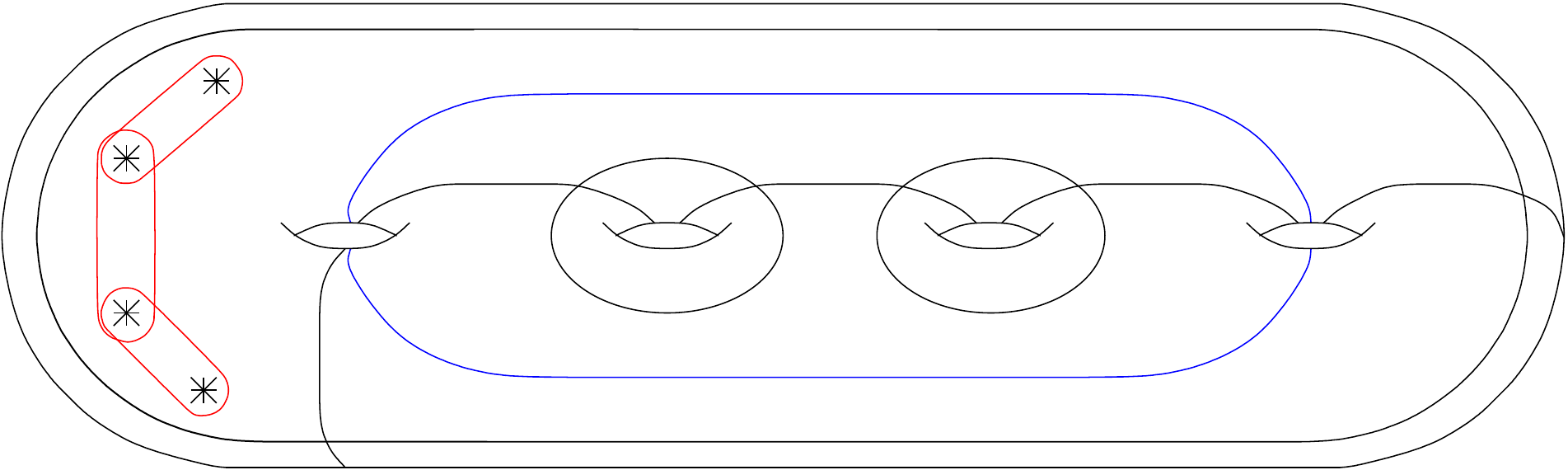}\\[0.3cm]
  \resizebox{10cm}{!}{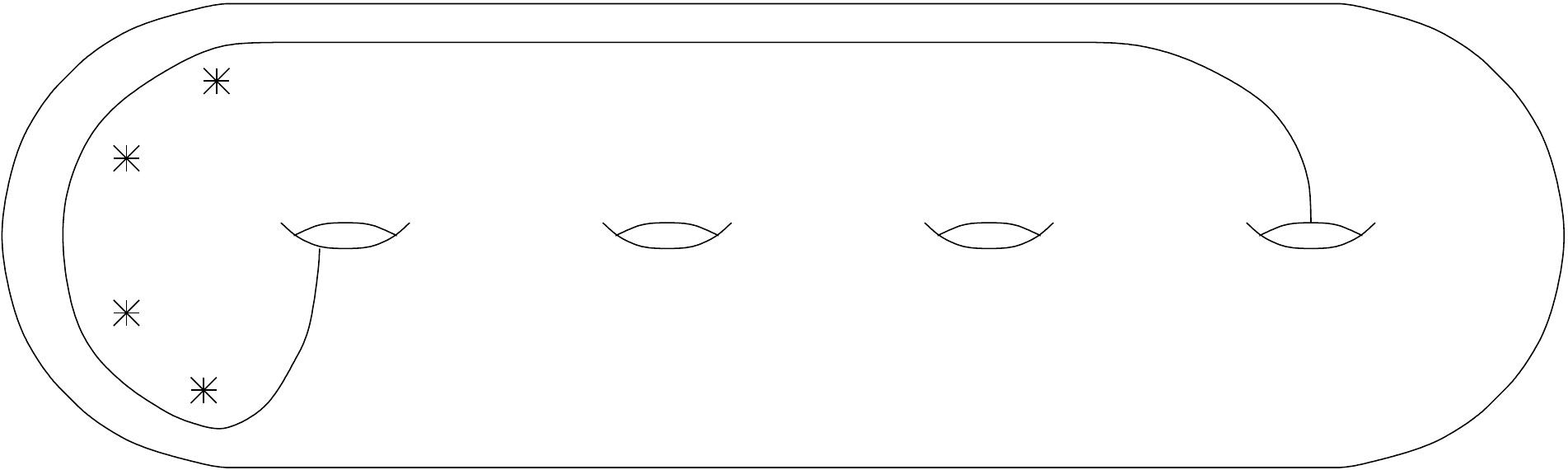} \caption{The curve $\beta_{\{2g,2g+1,0\}}^{4,+} \ColonEqq \langle (\Cfn \backslash \{\alpha_{1}, \alpha_{2g-1}\}) \cup \{\alpha_{0}^{n}\} \cup \Df \cup \{\beta_{\{2, \ldots, 2g-2\}}^{\pm}\}\rangle$.} \label{DefBeta2g2gplus10}
 \end{center}
\end{figure}\\
\indent Note that $\beta_{\{0,1,2\}}^{j,\pm}, \beta_{\{2g,2g+1,0\}}^{j,\pm} \in \Cf^{4}$ for all $1 \leq j \leq n-1$, and $\beta_{\{0,1,2\}}^{n,\pm}, \beta_{\{2g,2g+1,0\}}^{n, \pm} \in (\Cf \cup \Bfn)^{4}$. Then, we define the set $$\Bf_{T} \ColonEqq \{\beta_{J}^{i, +}, \beta_{J}^{i,-} : J \in \{ \{0,1,2\},\{2g,2g+1,0\}\}, 1 \leq i \leq n\} \subset (\Cf \cup \Bfn)^{4}.$$
\indent The set $\Bf_{0} \cup \Bf_{T}$ and the set $\Bf$ of the previous section are quite similar. However, $\Bf_{0} \cup \Bf_{T}$ is not as symmetric as $\Bf$ and it has a sense of incompleteness, for example that we are not including the boundary components of regular neighbourhoods of chains of odd length whose first curve has odd index. While some of these missing curves are not needed for the proof of Theorem \ref{Thm3}, some particular curves are needed for an idea of ``translation'' as in the previous section.
\begin{Prop}\label{translemaprop2}
 Let $k \in \mathbb{Z}$. Then $[\alpha_{k+1}, \ldots, \alpha_{k+(2g-1)}]^{+} \in (\Cf \cup \Bf_{0})^{4}$ for any choice of $i \in \{0, \ldots, n\}$ (with $\alpha_{0} = \alpha_{0}^{i}$ when necessary).
\end{Prop}
\begin{proof} Suppose $n \geq 2$. If $k = 0$, $[\alpha_{1}, \ldots, \alpha_{2g-1}]^{+} = \langle (\Cf_{0} \backslash \{\alpha_{2g}\}) \cup \Df\rangle$. We split the rest of the proof into two cases according to the parity of $k$.\\
\indent If $k \neq 0$ is even, then $k = 2(l+2)$ for $l \in \{-1, \ldots, g-2\}$, and so $[\alpha_{k+1}, \ldots, \alpha_{k+(2g-1)}]^{+} = [\alpha_{2l+5}, \ldots, \alpha_{2l+1}]^{+}$ (recall the subindices are modulo $2g+2$). Thus, if $i \in \{1, \ldots, n-1\}$ we get, $$[\alpha_{2l+5}, \ldots, \alpha_{2l+1}]^{+} = \langle \{\alpha_{2l+5}, \ldots, \alpha_{2l+1}\} \cup \{\alpha_{2l+3}\} \cup E^{i,i} \cup \{\epsp{l+1}{1}{n-1}, \epsp{l+2}{1}{n-1}\}\rangle \in (\Cf \cup \Bf_{0})^{4},$$ see Figure \ref{Translemaprop2fig1}; if $i \in \{0,n\}$ we get, $$[\alpha_{2l+5}, \ldots, \alpha_{2l+1}]^{+} = \langle \{\alpha_{2l+5}, \ldots, \alpha_{2l+1}\} \cup \{\alpha_{2l+3}\} \cup \Df\rangle \in (\Cf \cup \Bf_{0})^{3}.$$
\begin{figure}[h]
 \begin{center}
  \resizebox{10cm}{!}{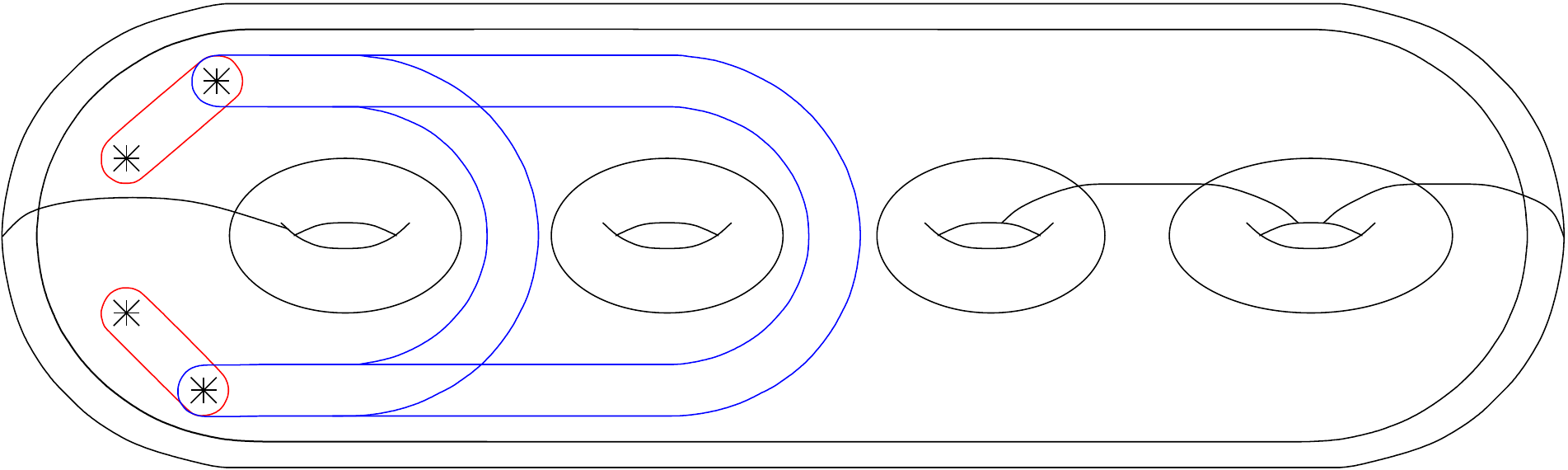} \\[0.3cm]
  \resizebox{10cm}{!}{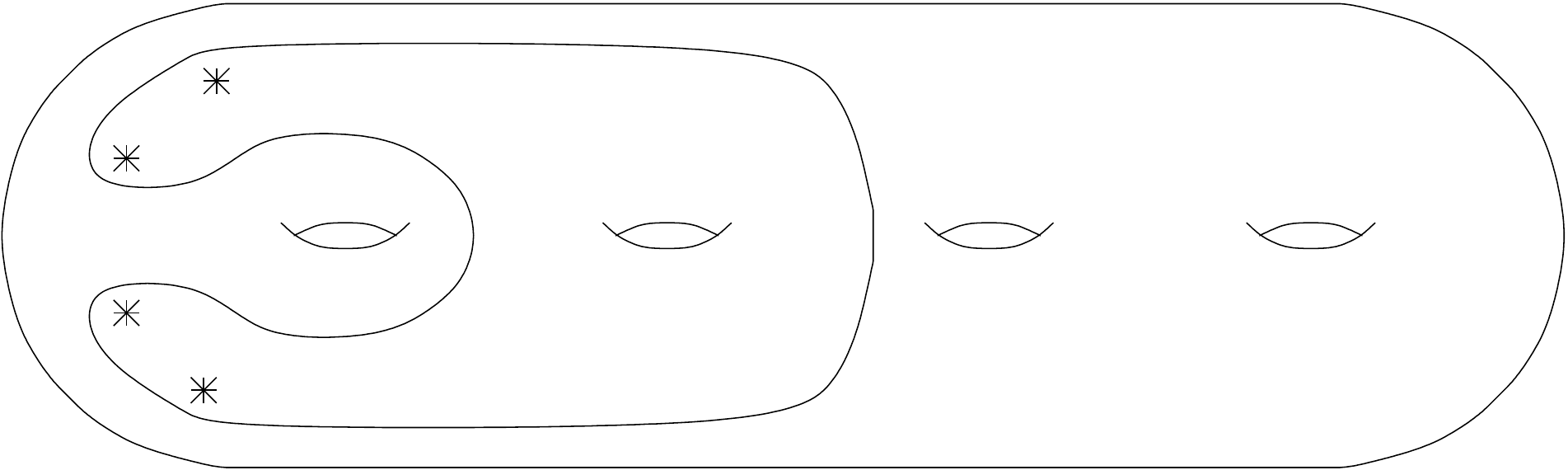} \caption{$[\alpha_{2l+5}, \ldots, \alpha_{2l+1}]^{+} = \left\langle \{\alpha_{2l+5}, \ldots, \alpha_{2l+1}\} \cup \{\alpha_{2l+3}\} \cup \right.$ \textcolor{red}{$E^{i,i}$} $\cup$ \textcolor{blue}{$\{\epsp{l+1}{1}{n-1}, \epsp{l+2}{1}{n-1}\}$}$\left. \right\rangle$ for $l = 0$ and $i = 2$, in $S_{4,4}$.} \label{Translemaprop2fig1}
 \end{center}
\end{figure}\\
\indent If $k$ is odd, then $k = 2l- 1$ for some $0 \leq l \leq g$, and so $[\alpha_{k+1}, \ldots, \alpha_{k+(2g-1)}]^{+} = [\alpha_{2l}, \ldots, \alpha_{2l-4}]^{+}$ (recall the subindices are modulo $2g+2$). Thus, if $i = 0$ we have that $[\alpha_{2l}, \ldots, \alpha_{2l-4}]^{+} = \alpha_{2l-2};$ if $i = 1$ we have that $[\alpha_{2l}, \ldots, \alpha_{2l-4}]^{+} \in \Bf_{0}$; if $i \in \{2, \ldots, n-1\}$ we have (see Figure \ref{Translemaprop2fig3}): $$[\alpha_{2l}, \ldots, \alpha_{2l-4}]^{+} = \langle \{\alpha_{2l}, \ldots, \alpha_{2l-4}\} \cup \{\alpha_{2l-2}\} \cup \{\alpha_{0}^{j}: j > i\} \cup E^{i,i} \cup \{\epsp{l-1}{1}{n-1}\}\rangle \in (\Cf \cup \Bf_{0})^{4};$$ and finally, if $i = n$ we have $$[\alpha_{2l}, \ldots, \alpha_{2l-4}]^{+} =\langle \{\alpha_{2l}, \ldots, \alpha_{2l-4}\} \cup \{\alpha_{2l-2}\} \cup \Df\rangle \in (\Cf \cup \Bf_{0})^{3}.$$
\begin{figure}[h]
 \begin{center}
  \resizebox{10cm}{!}{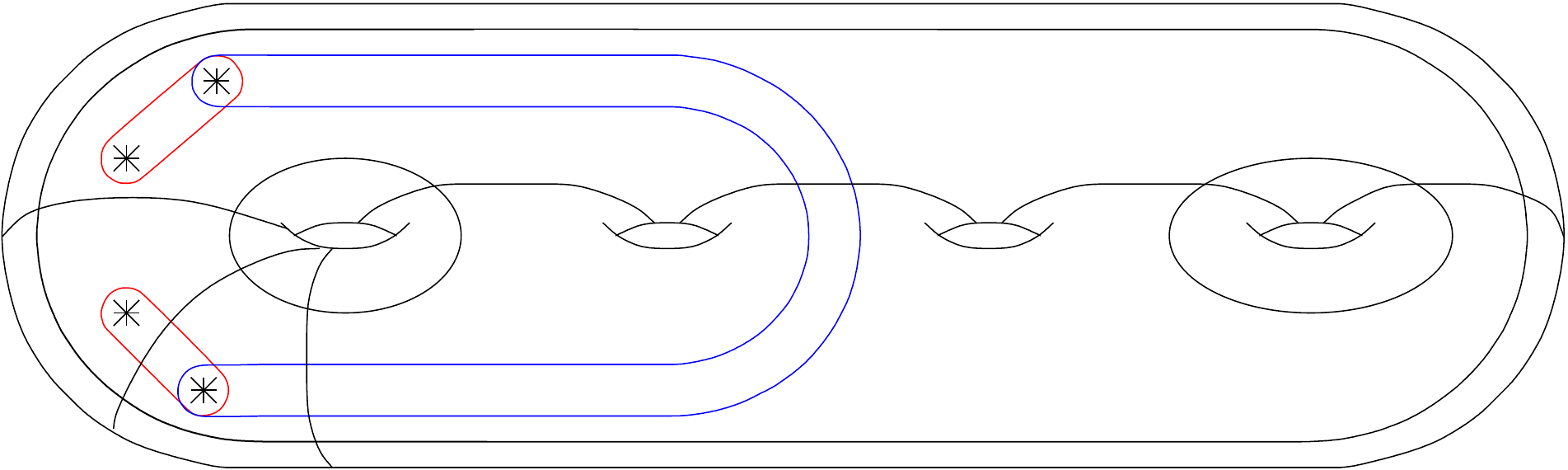} \\[0.3cm]
  \resizebox{10cm}{!}{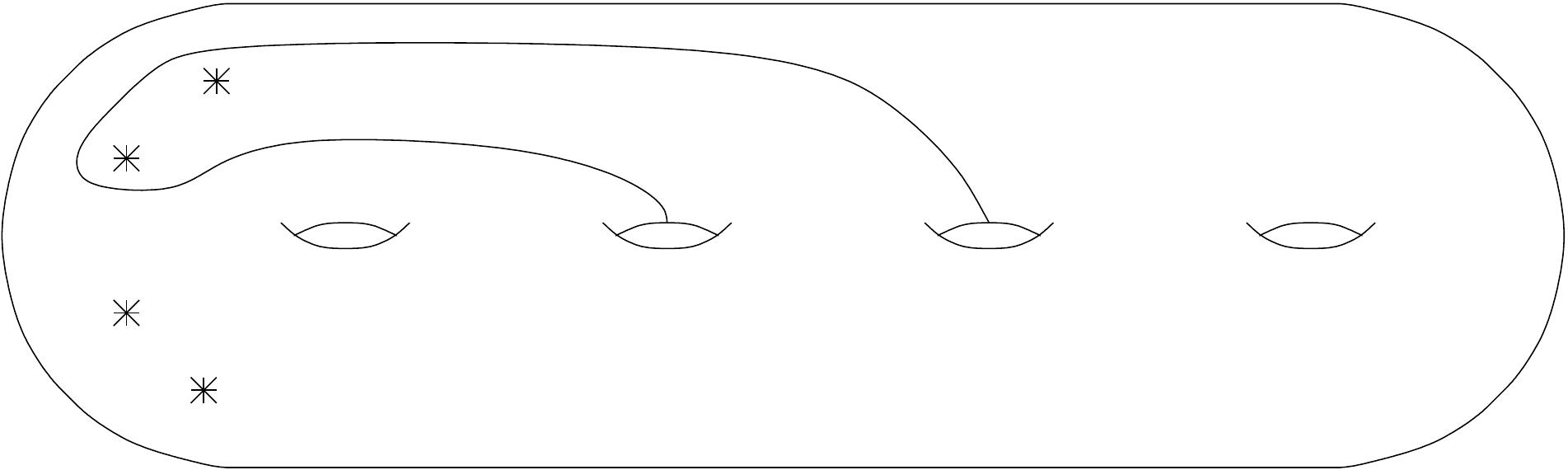} \caption{$[\alpha_{2l}, \ldots, \alpha_{2l-4}]^{+} = \langle \{\alpha_{2l}, \ldots, \alpha_{2l-4}\} \cup \{\alpha_{2l-2}\} \cup \{\alpha_{0}^{j}: j > i\} \cup$ \textcolor{red}{$E^{i,i}$} $\cup$ \textcolor{blue}{$\{\epsp{l-1}{1}{n-1}\}$}$\rangle$ for $l = 3$ and $i = 2$ in $S_{4,4}$.} \label{Translemaprop2fig3}
 \end{center}
\end{figure}
\indent Now, if $n = 1$, we can uniquely determine the curves $[\alpha_{k+1}, \ldots, \alpha_{k+(2g-1)}]^{+}$ in the same way as above, recalling that in this instance $\Df = \Ef = \varnothing$ and taking the cases $i = 0$ and $i = n$.\\
\indent Therefore, for $i \in \{0, \ldots, n\}$, $k \in \mathbb{Z}$, $[\alpha_{k+1}, \ldots, \alpha_{k+(2g-1)}]^{+} \in (\Cf \cup \Bf_{0})^{4}$.
\end{proof}
\indent Note that for $n = 1$, the curves $[\alpha_{k + 1}, \ldots, \alpha_{k + (2g+1)}]^{+}$ are elements of $(\Cf \cup \Bf_{0})^{1}$ for any $k \in \mathbb{Z}$ and any choice of $i \in \{0,1\}$.\\
\indent Now we have the following proposition.
\begin{Prop}\label{translemaprop1}
 Let $k \in \mathbb{Z}$. Then $[\alpha_{k}, \alpha_{k+1}, \alpha_{k+2}]^{\pm} \in (\Cf \cup \Bf_{0})^{6}$ for any choice of $i \in \{0, \ldots, n\}$ (with $\alpha_{0} = \alpha_{0}^{i}$ when necessary).
\end{Prop}
\begin{proof} 
 We start by proving that for $k \in \mathbb{Z}$ even, $[\alpha_{k}, \alpha_{k+1}, \alpha_{k+2}]^{\pm} \in (\Cf \cup \Bf_{0})^{4}$ (part 1); afterwards, we prove that for $k$ odd, $[\alpha_{k}, \alpha_{k+1}, \alpha_{k+2}]^{-} \in (\Cf \cup \Bf_{0})^{4}$ (part 2); finally we prove that for $k$ odd, $[\alpha_{k}, \alpha_{k+1}, \alpha_{k+2}]^{+} \in (\Cf \cup \Bf_{0})^{6}$ (part 3).\\
\textbf{Part 1:} If $k$ is even, $[\alpha_{k}, \alpha_{k+1}, \alpha_{k+2}]^{\pm} \in (\Cf \cup \Bf_{0})^{4}$ with the exception of $[\alpha_{2g}, \alpha_{2g+1}, \alpha_{0}^{0}]^{-}$, and $[\alpha_{0}^{0},\alpha_{1},\alpha_{2}]^{-}$ (we can verify that $[\alpha_{2g}, \alpha_{2g+1}, \alpha_{0}^{0}]^{+} = \beta_{\{2, \ldots, 2g-2\}}^{+}$ and $[\alpha_{0}^{0},\alpha_{1},\alpha_{2}]^{+} = \beta_{\{4, \ldots, 2g\}}^{+}$). This happens since $[\alpha_{k}, \alpha_{k+1}, \alpha_{k+2}]^{\pm}$ is an element of $\Bf_{0}$ or $\Bf_{T}$, for $k$ even with the aforementioned exceptions. So, for the first exception we have the following (see Figure \ref{Translemaprop1fig1}):$$[\alpha_{2g}, \alpha_{2g+1}, \alpha_{0}^{0}]^{-} = \langle (C_{0} \backslash \{\alpha_{1}, \alpha_{2g-1}\}) \cup \Df \cup \{\beta_{\{2, \ldots, 2g-2\}}^{\pm}\}\rangle.$$
\begin{figure}[h]
 \begin{center}
  \resizebox{9cm}{!}{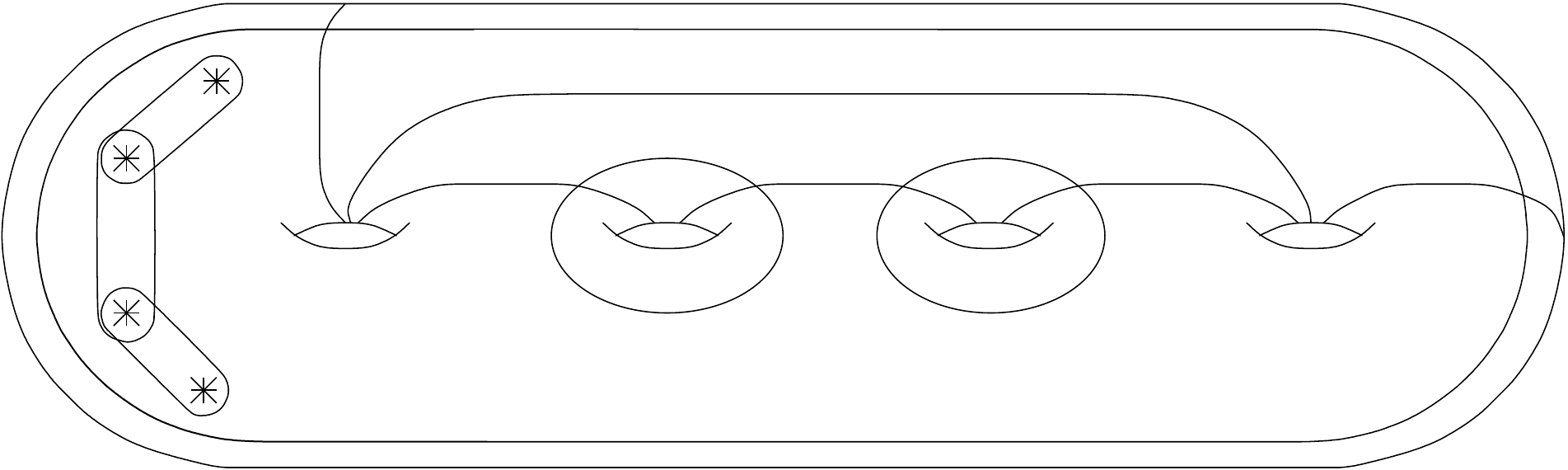} \\[0.2cm]
  \resizebox{9cm}{!}{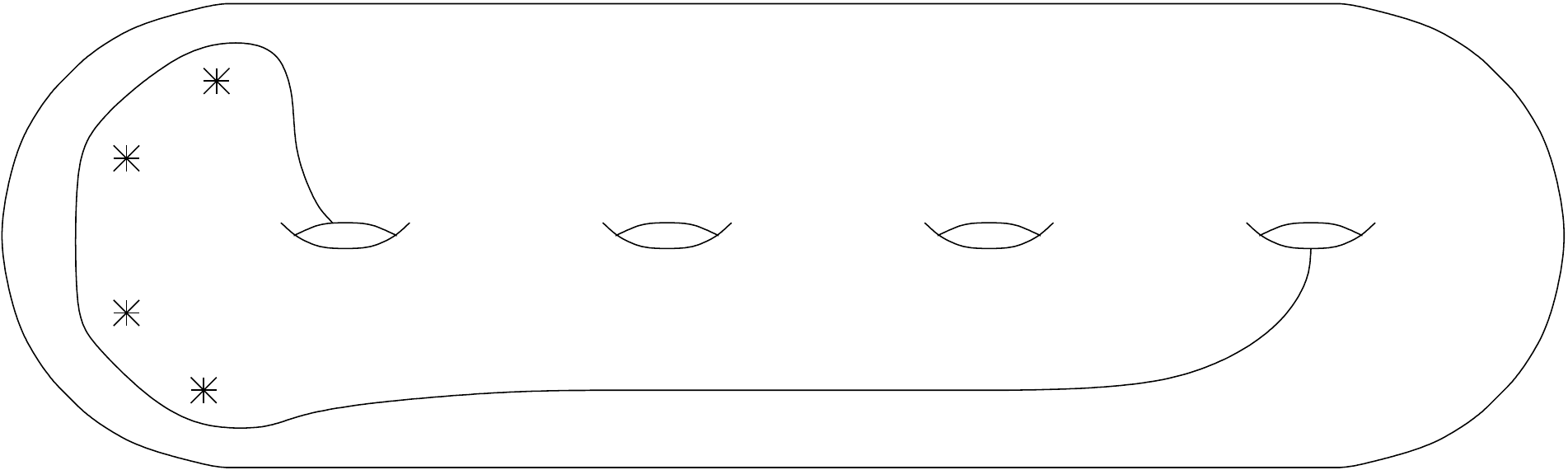} \caption{An illustration of $[\alpha_{2g}, \alpha_{2g+1}, \alpha_{0}^{0}]^{-} = \langle (C_{0} \backslash \{\alpha_{1}, \alpha_{2g-1}\}) \cup \Df \cup \{\beta_{\{2, \ldots, 2g-2\}}^{\pm}\}\rangle$.} \label{Translemaprop1fig1}
 \end{center}
\end{figure}\\
And for the second exception we have (see Figure \ref{Translemaprop1fig3}), $$[\alpha_{0}^{0},\alpha_{1},\alpha_{2}]^{-} = \langle (C_{0} \backslash \{\alpha_{3},\alpha_{2g+1}\}) \cup \Df \cup \{\beta_{\{4,\ldots, 2g\}}^{\pm}\}\rangle.$$
\begin{figure}[h]
 \begin{center}
  \resizebox{9cm}{!}{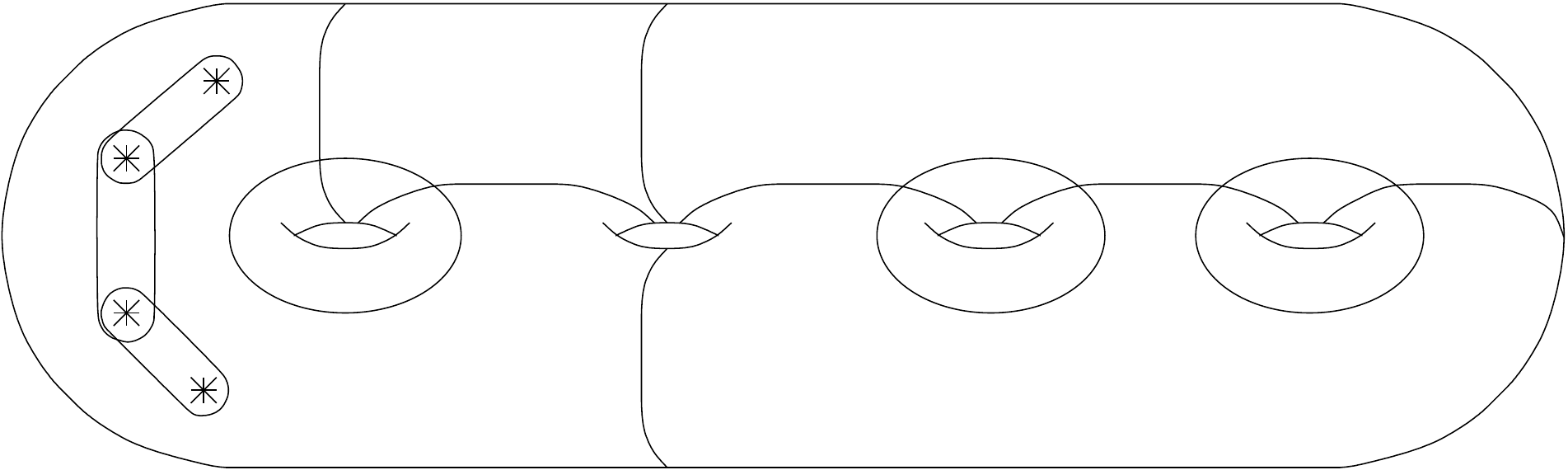} \\[0.2cm]
  \resizebox{9cm}{!}{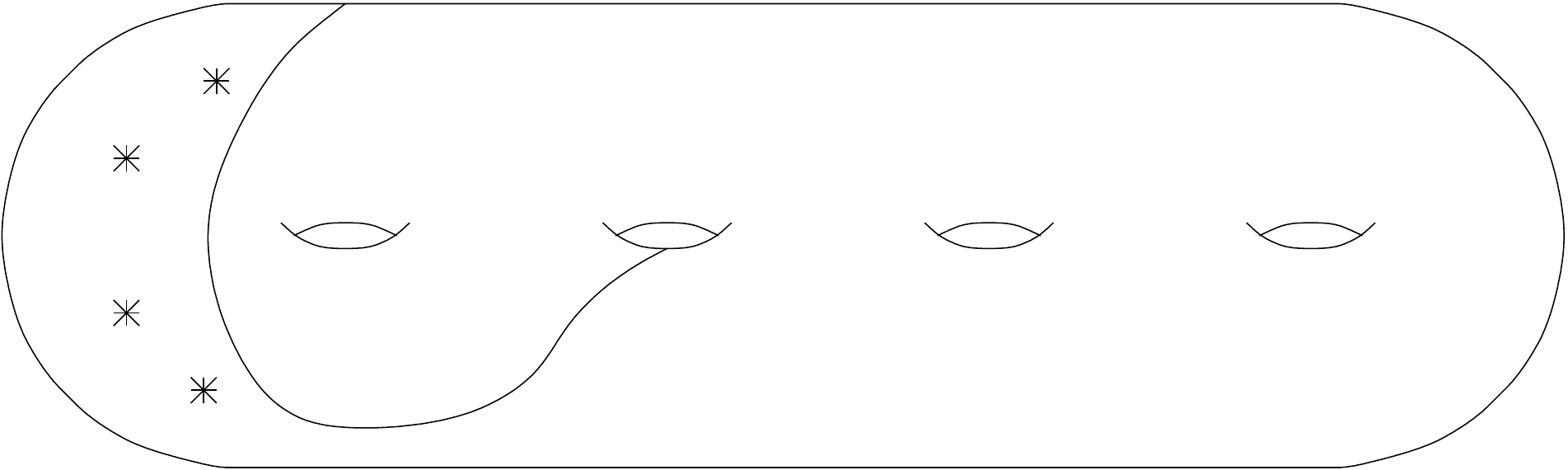} \caption{An illustration of $[\alpha_{0}^{0},\alpha_{1},\alpha_{2}]^{-} = \langle (C_{0} \backslash \{\alpha_{3},\alpha_{2g+1}\}) \cup \Df \cup \{\beta_{\{4,\ldots, 2g\}}^{\pm}\}\rangle$.} \label{Translemaprop1fig3}
 \end{center}
\end{figure}\\
\indent Therefore, for $k$ even (with $\alpha_{0} = \alpha_{0}^{i}$ when necessary), $[\alpha_{k}, \alpha_{k+1}, \alpha_{k+2}]^{\pm} \in (\Cf \cup \Bf_{0})^{4}$.\\
\textbf{Part 2:} Here we prove the case $n \geq 2$, leaving the analogous details of the case $n = 1$ to the reader (see \cite{Thesis}).\\
\indent Let $i \in \{0, \ldots, n\}$. For $k$ odd, we have to prove that each of the curves $[\alpha_{k}, \alpha_{k+1}, \alpha_{k+2}]^{-}$ is in $(\Cf \cup \Bf_{0})^{4}$. Let $k \in \{3, 5, \ldots, 2g+1\} \backslash \{2g-1\}$, then (see Figure \ref{Translemaprop1fig5}): $$[\alpha_{k},\alpha_{k+1},\alpha_{k+2}]^{-} = \langle (\Cf \backslash \{\alpha_{k-1}, \alpha_{k+3}\}) \cup \Ef\rangle.$$
\indent We also have $$[\alpha_{1}, \alpha_{2}, \alpha_{3}]^{-} = \langle (\Cf_{0} \backslash \{\alpha_{4}\}) \cup \Ef\rangle,$$ $$[\alpha_{2g-1}, \alpha_{2g}, \alpha_{2g+1}]^{-} = \langle (\Cf_{0} \backslash \{\alpha_{2g-2}\}) \cup \Ef\rangle.$$
\begin{figure}[h]
 \begin{center}
  \resizebox{10cm}{!}{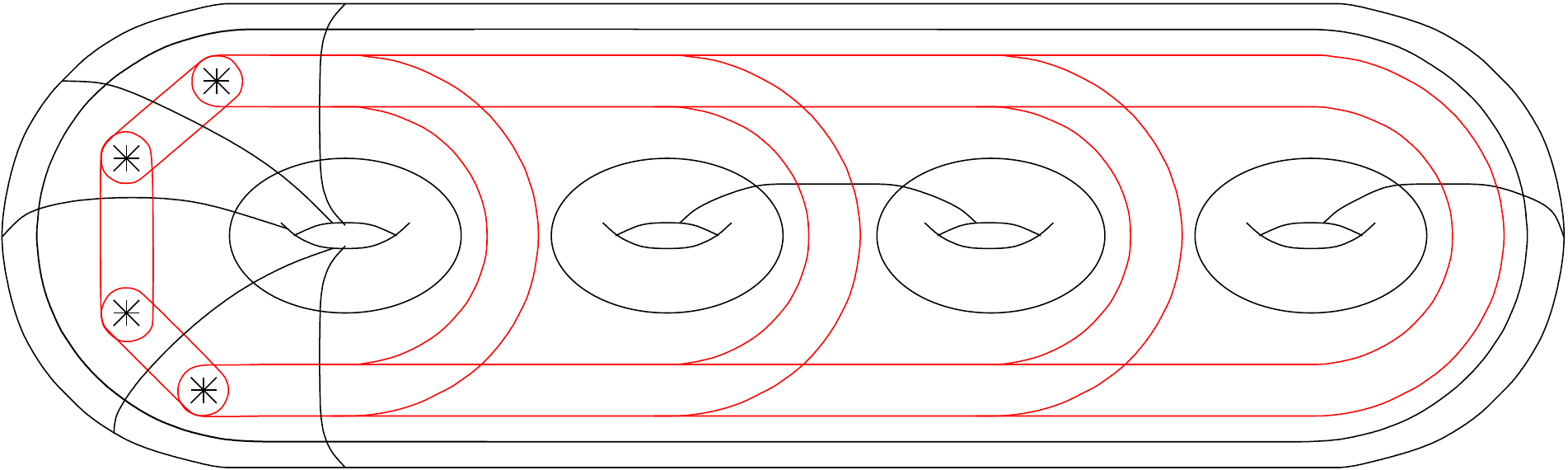} \\[0.3cm]
  \resizebox{10cm}{!}{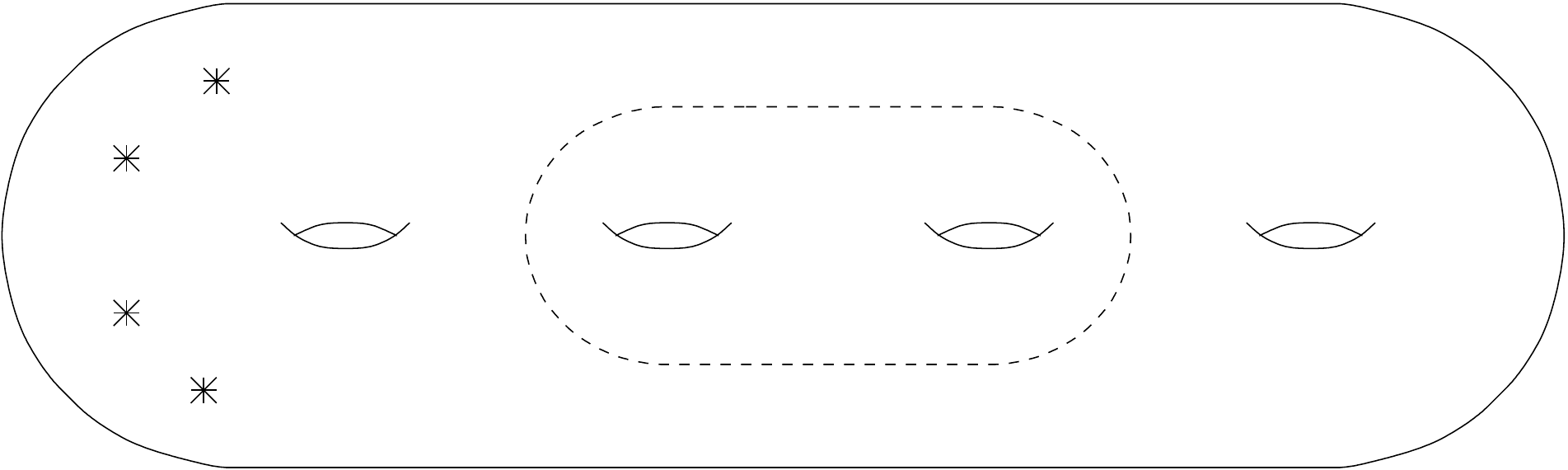} \caption{$[\alpha_{k},\alpha_{k+1},\alpha_{k+2}]^{-} = \langle (\Cf \backslash \{\alpha_{k-1}, \alpha_{k+3}\}) \cup$ \textcolor{red}{$\Ef$} $\rangle$ for $k = 3$ in $S_{4,4}$.}\label{Translemaprop1fig5}
 \end{center}
\end{figure}\\
\indent Therefore, for $k$ odd (with $\alpha_{0} = \alpha_{0}^{i}$ when necessary), $[\alpha_{k}, \alpha_{k+1}, \alpha_{k+2}]^{-}) \in (\Cf \cup \Bf_{0})^{4}$.\\
\textbf{Part 3:} As above, we only prove the case $n \geq 2$; for the details of the case $n = 1$ see \cite{Thesis}.\\
\indent Let $i \in \{0, \ldots, n\}$, $k$ be odd. Similarly to the previous part, we have to prove that $[\alpha_{k}, \alpha_{k+1}, \alpha_{k+2}]^{+} \in (\Cf \cup \Bf_{0})^{6}$. Let $k \in \{3, 5, \ldots, 2g-3\}$. Thus (see Figure \ref{Translemaprop1fig7}), $$[\alpha_{k},\alpha_{k+1},\alpha_{k+2}]^{+} = \langle (\Cf \backslash \{\alpha_{k-1},\alpha_{k+3}\}) \cup \{[\alpha_{k},\alpha_{k+1},\alpha_{k+2}]^{-}\}\rangle.$$
\begin{figure}[h]
 \begin{center}
  \resizebox{8cm}{!}{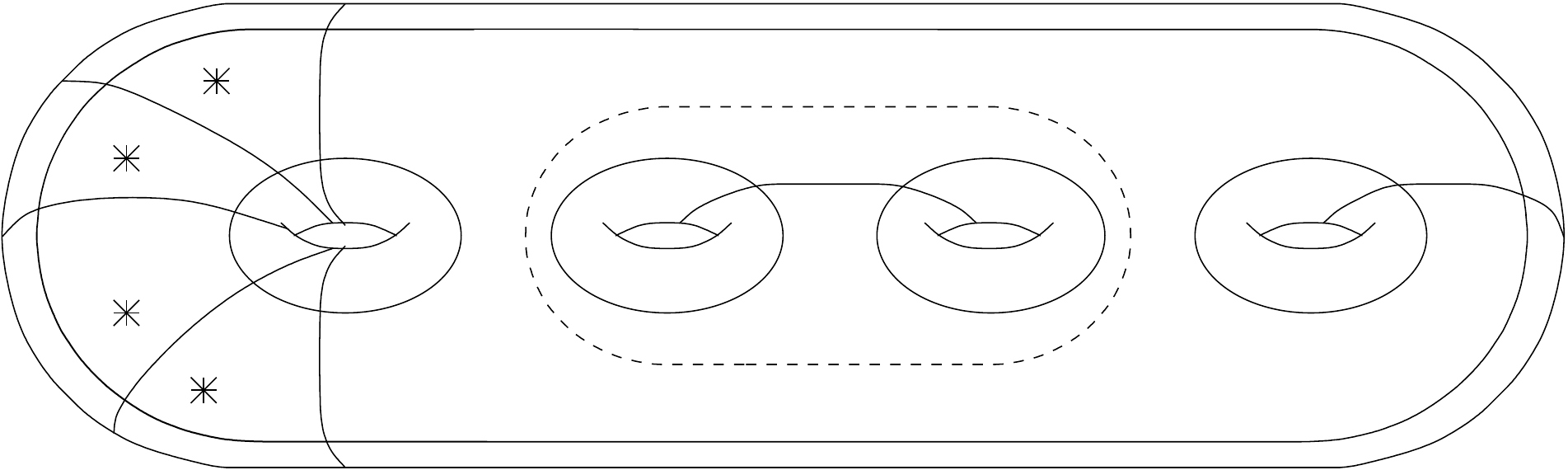} \\[0.2cm]
  \resizebox{8cm}{!}{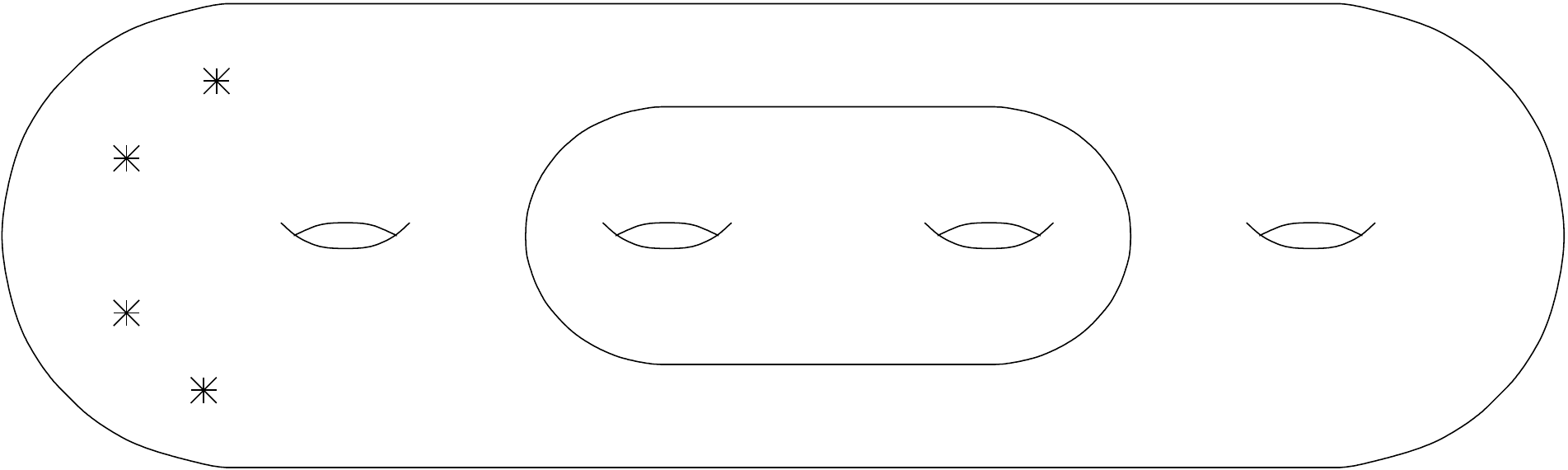} \caption{$[\alpha_{k},\alpha_{k+1},\alpha_{k+2}]^{+} = \langle (\Cf \backslash \{\alpha_{k-1},\alpha_{k+3}\}) \cup \{[\alpha_{k},\alpha_{k+1},\alpha_{k+2}]^{-}\}\rangle$ for $k = 3$ in $S_{4,4}$.}\label{Translemaprop1fig7}
 \end{center}
\end{figure}\\
 We then have (see Figure \ref{Translemaprop1fig9}), $$[\alpha_{2g-1}, \alpha_{2g}, \alpha_{2g+1}]^{+} = \left\langle (\Cf_{0} \backslash \{\alpha_{2g-2}\}) \cup \Df \cup \left(\bigcup_{l \in \{1, \ldots, g-1\}} \epsp{l}{1}{n-1}\right) \cup \{[\alpha_{2g-1},\alpha_{2g},\alpha_{2g+1}]^{-}\}\right\rangle,$$ $$[\alpha_{1}, \alpha_{2}, \alpha_{3}]^{+} = \left\langle (\Cf_{0} \backslash \{\alpha_{4}\}) \cup \Df \cup \left( \bigcup_{l \in \{2, \ldots, g\}} \{\epsp{l}{1}{n-1}\} \right) \cup \{[\alpha_{1}, \alpha_{2}, \alpha_{3}]^{-}\}\right\rangle.$$
\begin{figure}[h]
 \begin{center}
  \resizebox{8cm}{!}{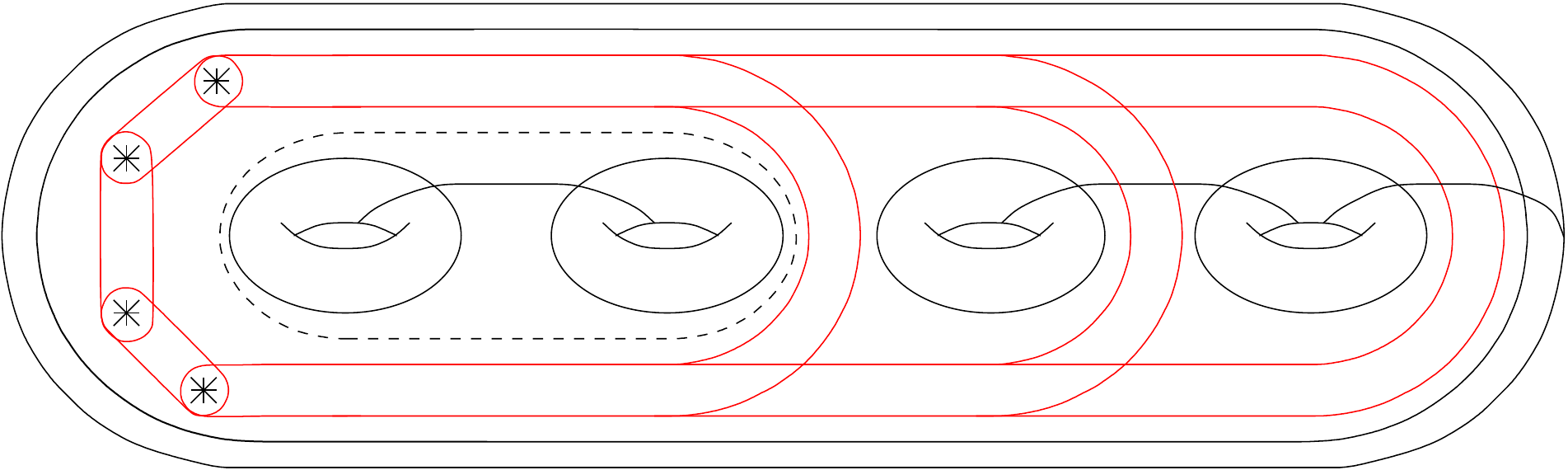} \\[0.2cm]
  \resizebox{8cm}{!}{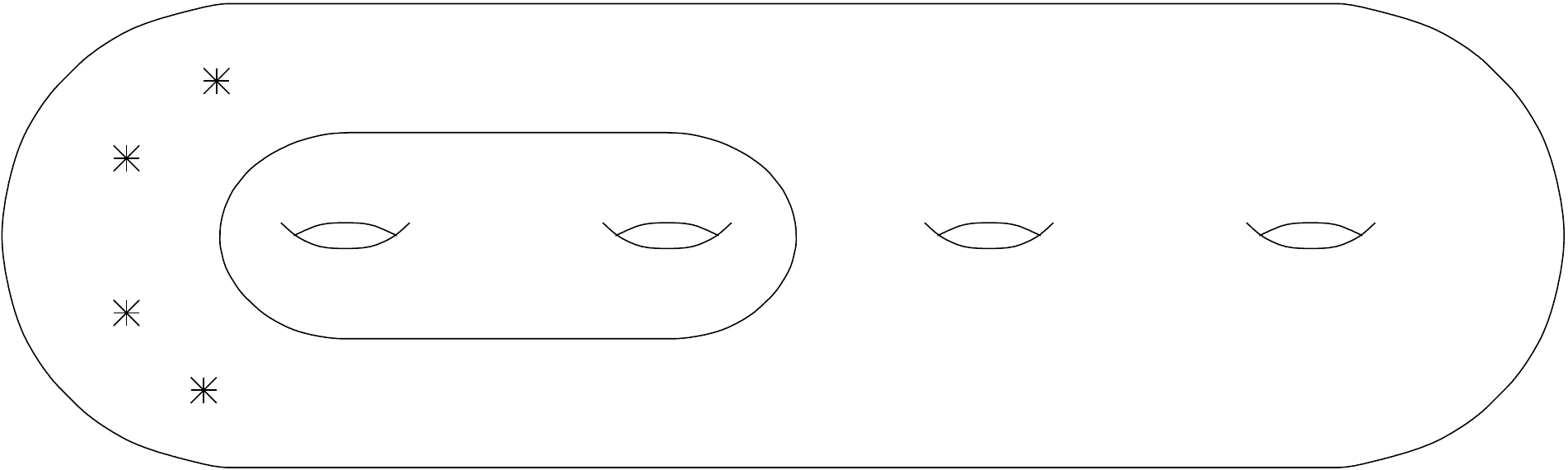} \caption{$[\alpha_{1}, \alpha_{2}, \alpha_{3}]^{+} = \left\langle (\Cf_{0} \backslash \{\alpha_{4}\}) \cup \Df \cup \left( \bigcup_{l \in \{2, \ldots, g\}} \{\epsp{l}{1}{n-1}\} \right) \cup \{[\alpha_{1}, \alpha_{2}, \alpha_{3}]^{-}\}\right\rangle$ in $S_{4,4}$.} \label{Translemaprop1fig9}
 \end{center}
\end{figure}\\
 For $i \in \{1, \ldots, n-1\}$, we get (see Figure \ref{DefAlpha2gplus1Alpha0iAlpha1plusfig1}), $$[\alpha_{2g+1}, \alpha_{0}^{i}, \alpha_{1}]^{+} = \left\langle (C_{i} \backslash \{\alpha_{2}, \alpha_{2g}\}) \cup E^{i,i} \cup \left(\bigcup_{l \in \{1, \ldots, g\}} \{\epsp{l}{1}{n-1}\} \right) \cup \{[\alpha_{2g+1}, \alpha_{0}, \alpha_{1}]^{-}\}\right\rangle;$$ 
\begin{figure}[h]
 \begin{center}
  \resizebox{8cm}{!}{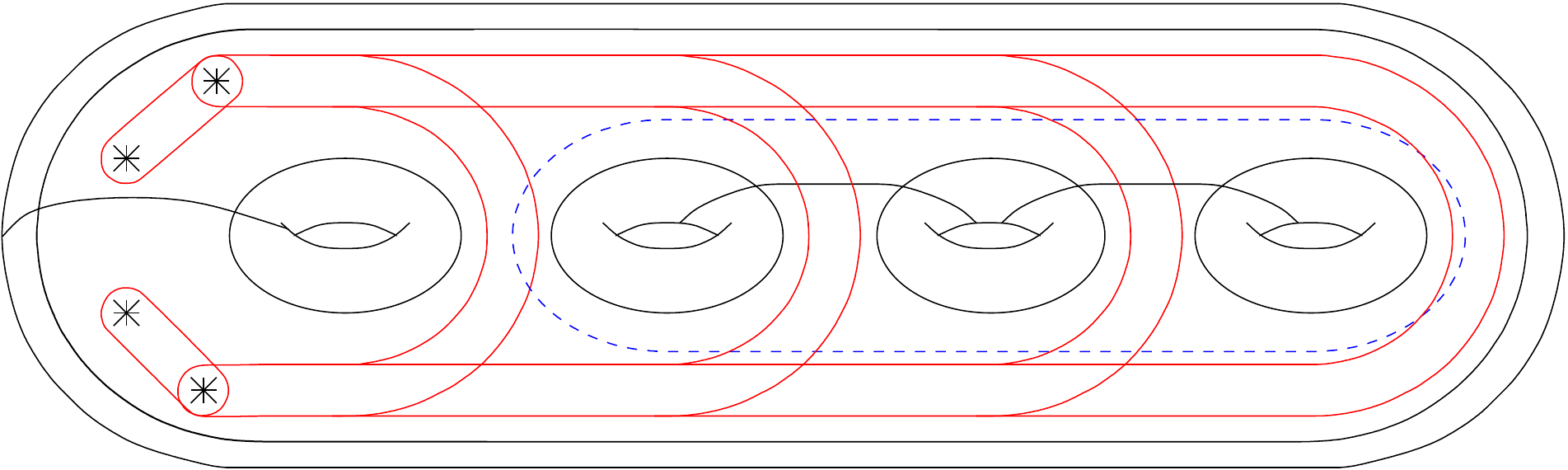} \\[0.2cm]
  \resizebox{8cm}{!}{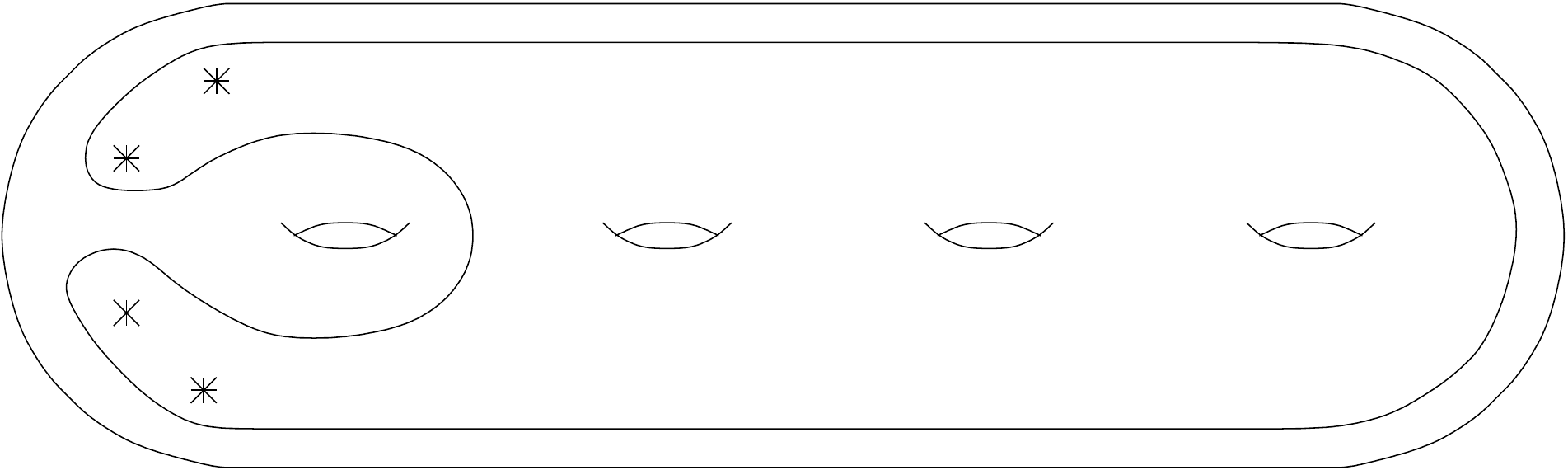} \caption{The curve $[\alpha_{2g+1}, \alpha_{0}^{2}, \alpha_{1}]^{+}$ $= \left\langle (C_{i} \backslash \{\alpha_{2}, \alpha_{2g}\}) \cup E^{i,i} \cup \left(\bigcup_{l \in \{1, \ldots, g\}} \{\epsp{l}{1}{n-1}\} \right) \cup \{[\alpha_{2g+1}, \alpha_{0}, \alpha_{1}]^{-}\}\right\rangle$.} \label{DefAlpha2gplus1Alpha0iAlpha1plusfig1}
 \end{center}
\end{figure}\\
 for $i \in \{0,n\}$, to prove the result for $[\alpha_{2g+1},\alpha_{0}^{i},\alpha_{1}]^{+}$, we need the auxiliary curve (see Figure \ref{Translemaprop1fig11}): $$[\alpha_{3}, \ldots, \alpha_{2g-1}]^{+} = \langle (\Cf \backslash \{\alpha_{2}, \alpha_{2g}\}) \cup \{[\alpha_{2g+1}, \alpha_{0}^{i}, \alpha_{1}]^{-}\}\rangle \in (\Cf \cup \Bf_{0})^{5};$$
\begin{figure}[h]
 \begin{center}
  \resizebox{8cm}{!}{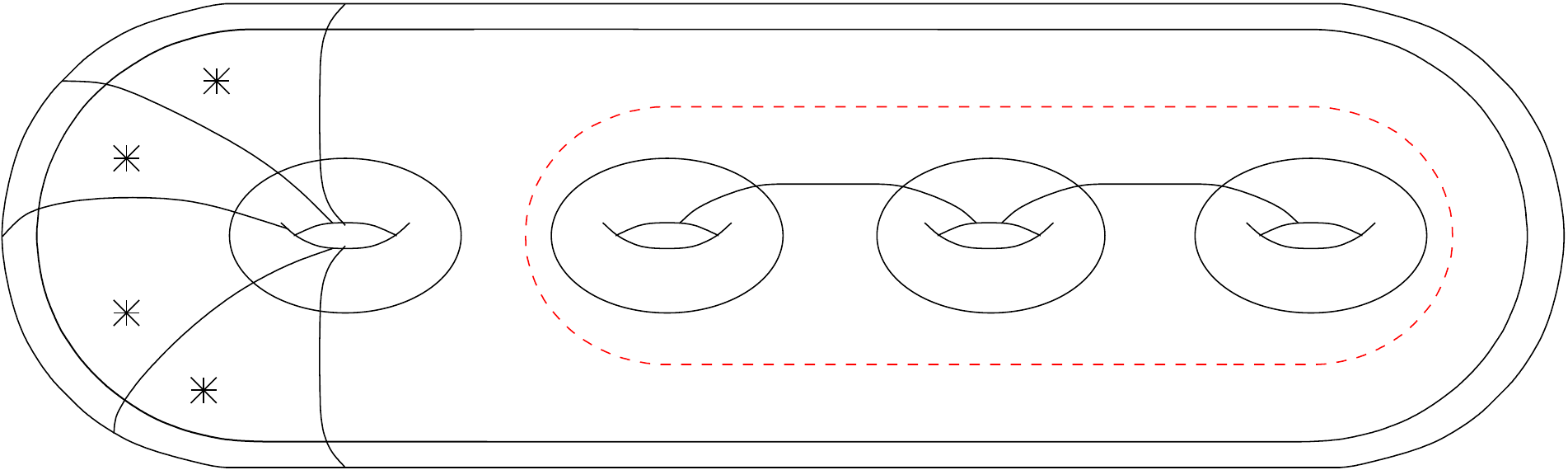} \\[0.2cm]
  \resizebox{8cm}{!}{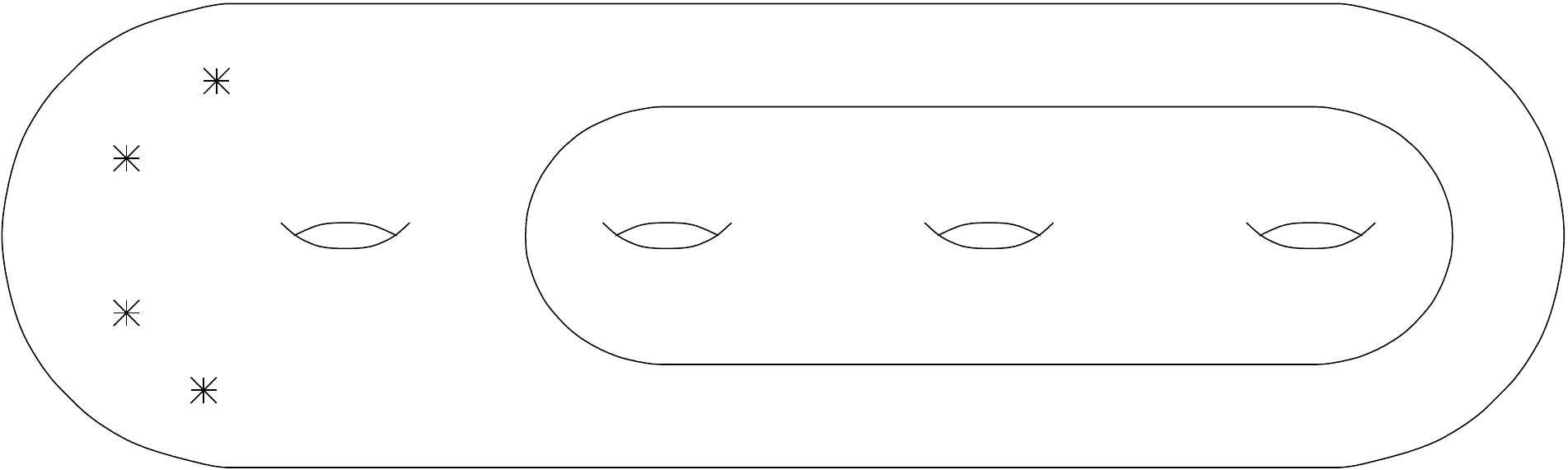} \caption{$[\alpha_{3}, \ldots, \alpha_{2g-1}]^{+} = \langle (\Cf \backslash \{\alpha_{2}, \alpha_{2g}\}) \cup$ \textcolor{red}{$\{[\alpha_{2g+1}, \alpha_{0}^{i}, \alpha_{1}]^{-}\}$} $\rangle$ for $j = 3$ in $S_{4,4}$.}\label{Translemaprop1fig11}
 \end{center}
\end{figure}\\
 and so (see Figure \ref{Translemaprop1fig13}), $$[\alpha_{2g+1}, \alpha_{0}^{i}, \alpha_{1}]^{+} = \langle (C_{i} \backslash \{\alpha_{2},\alpha_{2g}\}) \cup \Df \cup \{[\alpha_{2g+1}, \alpha_{0}^{i}, \alpha_{1}]^{-}, [\alpha_{3}, \ldots, \alpha_{2g-1}]^{+}\}\rangle.$$
\begin{figure}[h]
 \begin{center}
  \resizebox{8cm}{!}{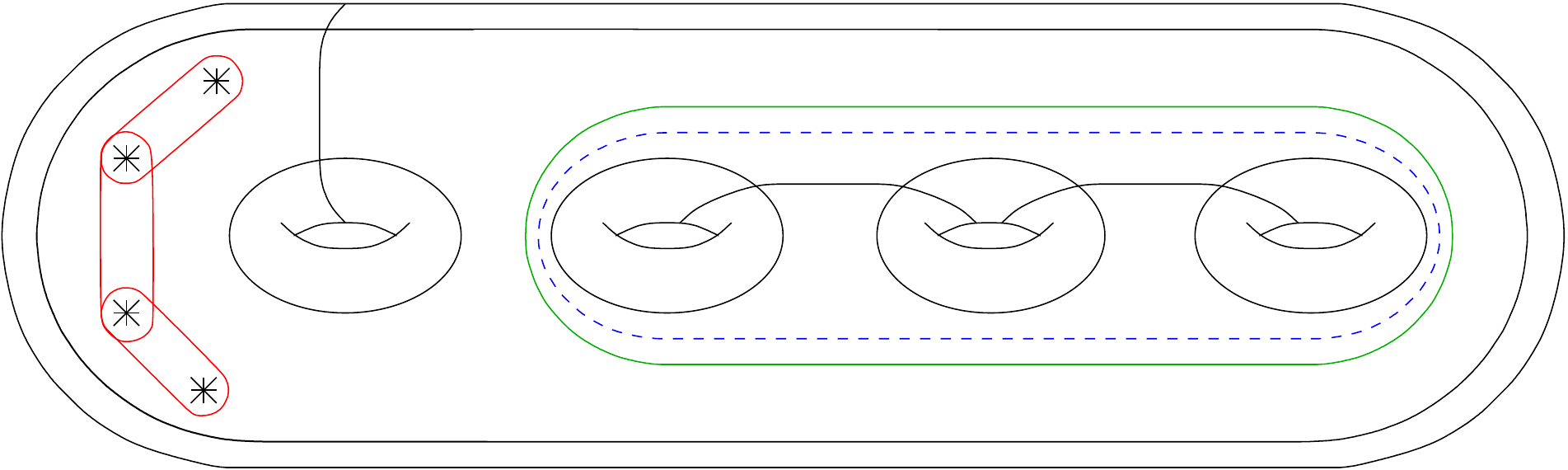} \\[0.2cm]
  \resizebox{8cm}{!}{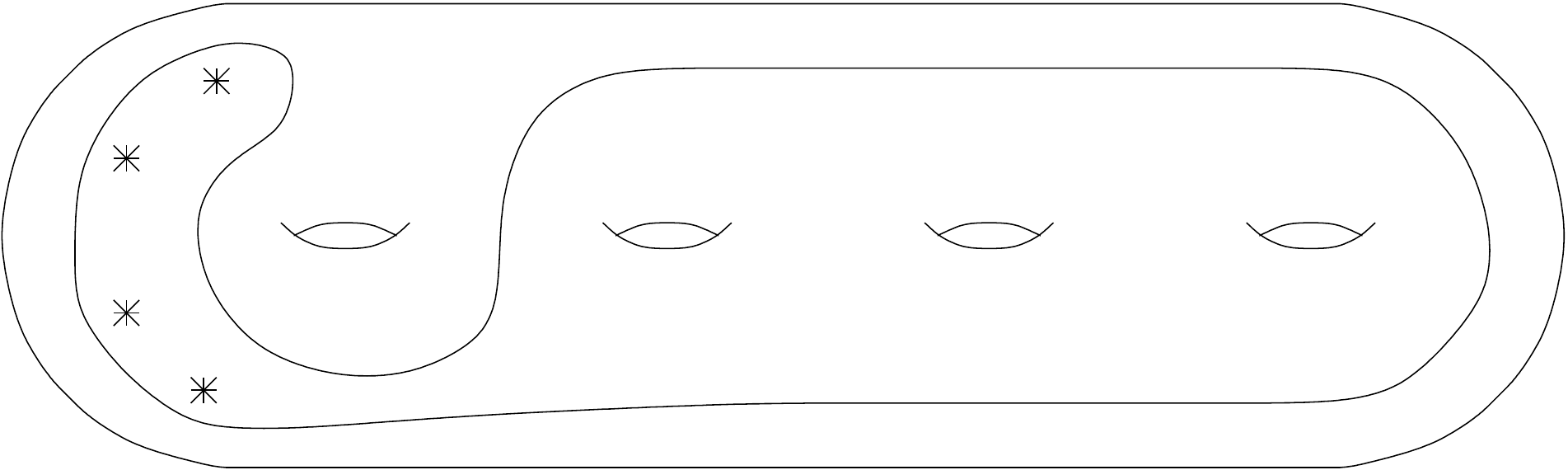} \caption{$[\alpha_{2g+1}, \alpha_{0}^{i}, \alpha_{1}]^{+} = \langle (C_{i} \backslash \{\alpha_{2},\alpha_{2g}\}) \cup$ \textcolor{red}{$\Df$} $\cup \{$ \textcolor{blue}{$[\alpha_{2g+1}, \alpha_{0}^{i}, \alpha_{1}]^{-}$}$, $ \textcolor{green}{$[\alpha_{3}, \ldots, \alpha_{2g-1}]^{+}$} $\}\rangle$ for $j = 3$ in $S_{4,4}$.}\label{Translemaprop1fig13}
 \end{center}
\end{figure}\\
\indent Therefore, for $k \in \mathbb{Z}$ (with $\alpha_{0} = \alpha_{0}^{i}$ when necessary), $[\alpha_{k}, \alpha_{k+1}, \alpha_{k+2}]^{\pm} \in (\Cf \cup \Bf_{0})^{6}$.
\end{proof}
\indent Finally, we define the set of auxiliary curves $\Bf$. Note that by construction, $\Bf \subset (\Cf \cup \Bf_{0})^{6}$. $$\Bf \ColonEqq \Bf_{0} \cup \Bf_{T} \cup \left(\bigcup_{\stackrel{i \in \{0, \ldots, n\}}{\scriptscriptstyle k \in \mathbb{Z}}} \{[\alpha_{k+1}, \ldots, \alpha_{k+(2g-1)}]^{+}, [\alpha_{k},\alpha_{k+1},\alpha_{k+2}]^{\pm}\} \right).$$
\subsection{Proof of Theorem \ref{Thm3}}\label{chap2sec3}
\indent Having defined the principal set $\Cf \cup \Bf_{0}$ and constructed the auxiliary sets $\Df \subset \Ef$ and $\Bf_{T} \subset \Bf$, we state some results to ease the proofs of the following section, as well as give necessary notation and the proof of Theorem \ref{Thm3}.
\begin{Prop}\label{OnlyCUB0}
 Let $h \in \EMod{S}$ and $Y \subset \ccomp{S}$. If $h(Y) \subset (\Cf \cup \Bf_{0})^{k}$ for some $k \in \mathbb{Z}$, then $h(Y^{m}) \subset (\Cf \cup \Bf_{0})^{k+m}$.
\end{Prop}
\begin{proof}
 Let $\gamma \in Y^{1}$. If $\gamma \in Y$, then $h(\gamma) \in (\Cf \cup \Bf_{0})^{k} \subset (\Cf \cup \Bf_{0})^{k+1}$ by hypothesis. Otherwise, there exists a set $A \subset Y$ such that $\gamma = \langle A \rangle$. Given that $h$ is a mapping class we have that $h(\gamma) = \langle h(A) \rangle$, and since by hypothesis $h(A) \subset h(Y) \subset (\Cf \cup \Bf_{0})^{k}$ we get that $h(\gamma) \in (\Cf \cup \Bf_{0})^{k+1}$. This implies that $h(Y^{1}) \subset (\Cf \cup \Bf_{0})^{k+1}$. By induction, we obtain that $h(Y^{m}) \subset (\Cf \cup \Bf_{0})^{k+m}$.\\
\end{proof}
\indent As a consequence of this proposition, since $\Cf \cup \Ef \cup \Bf \subset (\Cf \cup \Bf_{0})^{6}$, we have the following corollary.
\begin{Cor}\label{CorOnlyCUB0}
 Let $h \in \EMod{S}$. If $h(\Cf \cup \Bf_{0}) \subset (\Cf \cup \Bf_{0})^{k}$ for some $k \in \mathbb{Z}$, then $h(\Cf \cup \Ef \cup \Bf) \subset (\Cf \cup \Bf_{0})^{k+6}$.
\end{Cor}
\indent An \textit{outer curve} $\alpha$ is a separating curve such that cutting along $\alpha$ one of the resulting connected components is homeomorphic to a thrice-punctured sphere. Let $\alpha, \beta \in \ccomp{S}$ and $A, B \subset \ccomp{S}$. We denote by $\eta_{\alpha}(\beta)$ the half-twist of $\beta$ along $\alpha$ and $\eta_{A}(B) = \bigcup_{\gamma \in A} \eta_{\gamma}(B)$.\\
\indent We must recall that the half-twist $\eta_{\alpha}$ is defined if and only if $\alpha$ is an outer curve, and there is exactly one half-twist along $\alpha$ if $S \ncong S_{0,4}$. Let $\Hff \ColonEqq \{\epsilon^{i-2,i} \in \Df: 2 \leq i \leq n\}$, then we can state the following lemma.
\begin{Lema}\label{CUBen12}
 Let $\zeta = \beta_{\{2g-2,2g-1,2g\}}^{+}$, and $\Hff$ as above. Then $\tau_{\Cf \cup \{\zeta\}}^{\pm 1}(\Cf \cup \Ef \cup \Bf) \cup \eta_{\Hff}^{\pm 1}(\Cf \cup \Ef \cup \Bf) \subset (\Cf \cup \Bf_{0})^{18}$.
\end{Lema}
Assuming this lemma (for which we give a proof in the following subsections) we can proceed to prove Theorem \ref{Thm3} as follows.
\begin{proof}[\textbf{Proof of Theorem \ref{Thm3}}]
Let $\Gf = (\Cf \backslash \{\alpha_{2g+1}\}) \cup \{\zeta\}$. Recalling Lickorish-Humphries Theorem (see \cite{Lickorish}, \cite{Hump} and Section 4 of \cite{FarbMar}), we have that $\Mod{S}$ is generated by the Dehn twists along $\Gf$ if $n \leq 1$ and by the Dehn twists along $\Gf$ and the half twists along $\Hff$ if $2 \leq n$. By Lemma \ref{CUBen12} we know that $\tau_{\Gf}^{\pm 1}(\Cf \cup \Ef \cup \Bf) \cup \eta_{\Hff}^{\pm 1}(\Cf \cup \Ef \cup \Bf) \subset (\Cf \cup \Bf_{0})^{18}$.\\
\indent Let us denote by $f_{\gamma}$ the Dehn twist along $\gamma$ if $\gamma \in \Gf$ or the half-twist along $\gamma$ if $\gamma \in \Hff$.\\
\indent Now let $\gamma$ be either a nonseparating curve or an outer curve, and $\alpha$ an element in $\Gf$ or $\Hff$ respectively. There exists $h \in \Mod{S}$ such that $\gamma = h(\alpha)$. Thus, by an iterated used of Lemma \ref{CUBen12}, there are some curves $\gamma_{1}, \ldots, \gamma_{l} \in \Gf \cup \Hff$ and some exponents $n_{1}, \ldots, n_{l} \in \mathbb{Z}$, such that: $$\gamma = f_{\gamma_{1}}^{n_{1}} \circ \cdots \circ f_{\gamma_{l}}^{n_{l}} (\alpha) \in (\Cf \cup \Bf_{0})^{18(|n_{1}| + \ldots + |n_{l}|)}.$$ So, every nonseparating curve and every outer curve is an element of $\bigcup_{i \in \nat} (\Cf \cup \Bf_{0})^{i}$.
\begin{figure}[h]
\begin{center}
 \includegraphics[width=10cm]{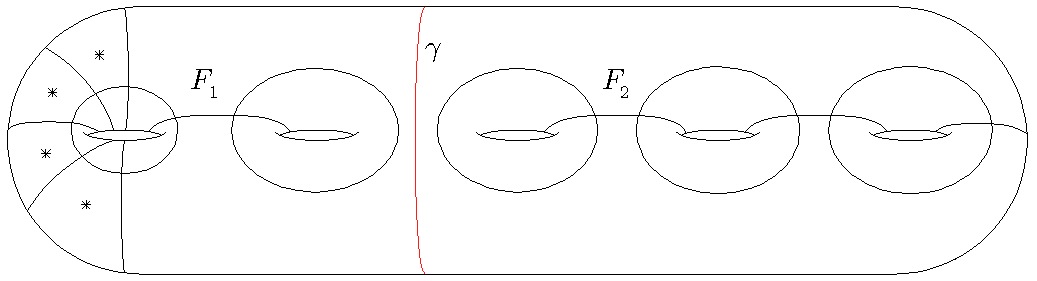}
 \includegraphics[width=10cm]{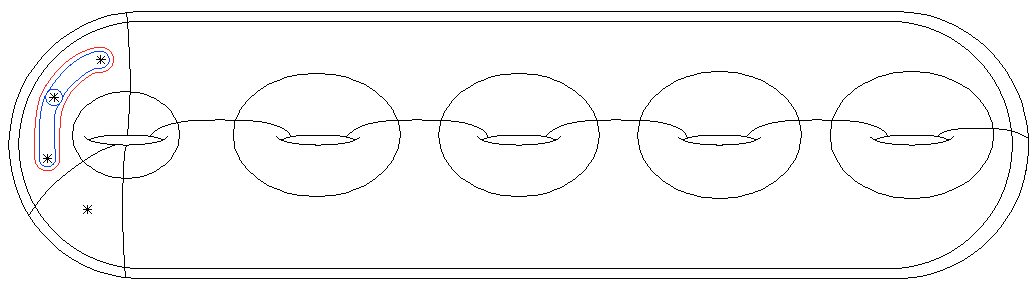} \caption{Above, a separating curve $\gamma$ with every connected component of $S \backslash \{\gamma\}$ of positive genus; below, a separating curve $\gamma$ with a connected component of $S \backslash \{\gamma\}$ of genus zero, the set $F_{1}$ in blue, the set $F_{2}$ in black, and $\gamma$ in red.} \label{PseudoChains}
\end{center}
\end{figure}\\
\indent Let $\gamma$ be a nonouter separating curve. We can always find sets $F_{1}$ and $F_{2}$ containing only nonseparating and outer curves, such that $\gamma = \langle F_{1} \cup F_{2} \rangle$, see Figure \ref{PseudoChains}. By the previous case, $F_{1} \cup F_{2} \subset (\Cf \cup \Bf_{0})^{k}$ for some $k \in \nat$; thus $\gamma \in (\Cf \cup \Bf_{0})^{k+1}$. Therefore $\ccomp{S} = \bigcup_{i \in \nat} (\Cf \cup \Bf_{0})^{i}$.
\end{proof}
Now, to prove Lemma \ref{CUBen12}, due to Corollary \ref{CorOnlyCUB0} we only need to prove that $\tau_{\Cf \cup \{\zeta\}}^{\pm 1}(\Cf \cup \Bf_{0}) \cup \eta_{\Hff}^{\pm 1}(\Cf \cup \Bf_{0}) \subset (\Cf \cup \Bf_{0})^{12}$. For this, we divide the proof into the following claims:\\
\textbf{Claim 1:} $\tau_{\Cf}^{\pm 1}(\Cf) \subset (\Cf \cup \Bf_{0})^{8}$\\
\textbf{Claim 2:} $\tau_{\Cf}^{\pm 1}(\Cf \cup \Bf_{0}) \subset (\Cf \cup \Bf_{0})^{12}$\\
\textbf{Claim 3:} $\tau_{\zeta}^{\pm 1}(\Cf) \subset (\Cf \cup \Bf_{0})^{8}$\\
\textbf{Claim 4:} $\tau_{\zeta}^{\pm 1}(\Cf \cup \Bf_{0}) \subset (\Cf \cup \Bf_{0})^{10}$\\
\textbf{Claim 5:} $\eta_{\Hff}^{\pm 1}(\Cf) \subset (\Cf \cup \Bf_{0})^{7}$\\
\textbf{Claim 6:} $\eta_{\Hff}^{\pm 1}(\Cf \cup \Bf_{0}) \subset (\Cf \cup \Bf_{0})^{11}$
\subsection{Proof of Claim 1: $\tau_{\Cf}^{\pm 1}(\Cf) \subset (\Cf \cup \Bf_{0})^{8}$}\label{chap2sec4}
\indent Let $\alpha_{j}, \alpha_{k} \in \Cf$ (taking $\alpha_{0} = \alpha_{0}^{i}$ when necessary). If $|k -j| > 1$, $i(\alpha_{j},\alpha_{k}) = 0$ and so we have that $\tau_{\alpha_{k}}^{\pm 1}(\alpha_{j}) = \alpha_{j} \in \Cf$. We then only need to prove for the case when $|k-j| = 1$.\\
\indent In contrast to the closed surface case $\EMod{S}$ \emph{does not act transitively} on $\Cf$ (it has two orbits), so here we first prove that $\tau_{\alpha_{2g}}^{\pm 1}(\alpha_{2g-1})$ and $\tau_{\alpha_{2g-1}}^{\pm 1}(\alpha_{2g-2})$ are elements of $(\Cf \cup \Bf_{0})^{8}$, then we use the action of a subgroup of $\EMod{S}$ to prove Claim 1.\\
\indent Let $A \subset \ccomp{S}$. We define $E(A) \ColonEqq \{\epsilon \in \Ef : i(\epsilon,\delta) = 0$ for all $\delta \in A\}$.\\
\indent Following \cite{Ara2}, as in Subsection \ref{subsec4-2}, we prove first that $\tau^{\pm 1}_{\alpha_{2g}} (\alpha_{2g-1}) \in (\Cf \cup \Bf_{0})^{8}$.
\begin{Lema}\label{taupm1v1CUB0}
 $\tau^{\pm 1}_{\alpha_{2g}} (\alpha_{2g-1}) \in (\Cf \cup \Bf_{0})^{8}$.
\end{Lema}
\begin{proof}
 \indent Taking the set $$C_{1^{+}} = \{\alpha_{2g+1}, \alpha_{1}, \alpha_{2}, \ldots, \alpha_{2g-4}, \alpha_{2g-2}, [\alpha_{2g-3}, \alpha_{2g-2}, \alpha_{2g-1}]^{+}, [\alpha_{2g-4},\alpha_{2g-3},\alpha_{2g-2}]^{+}, [\alpha_{1}, \ldots, \alpha_{2g-1}]^{+}\},$$
 then, by Proposition \ref{translemaprop1}, we get that $\gamma_{+} \ColonEqq \langle C_{1^{+}} \cup E(C_{1^{+}})\rangle \in (\Cf \cup \Bf_{0})^{7}$. See Figure \ref{DehnTwistPuncturedFig1}.
 \begin{figure}[h]
  \begin{center}
   \resizebox{10cm}{!}{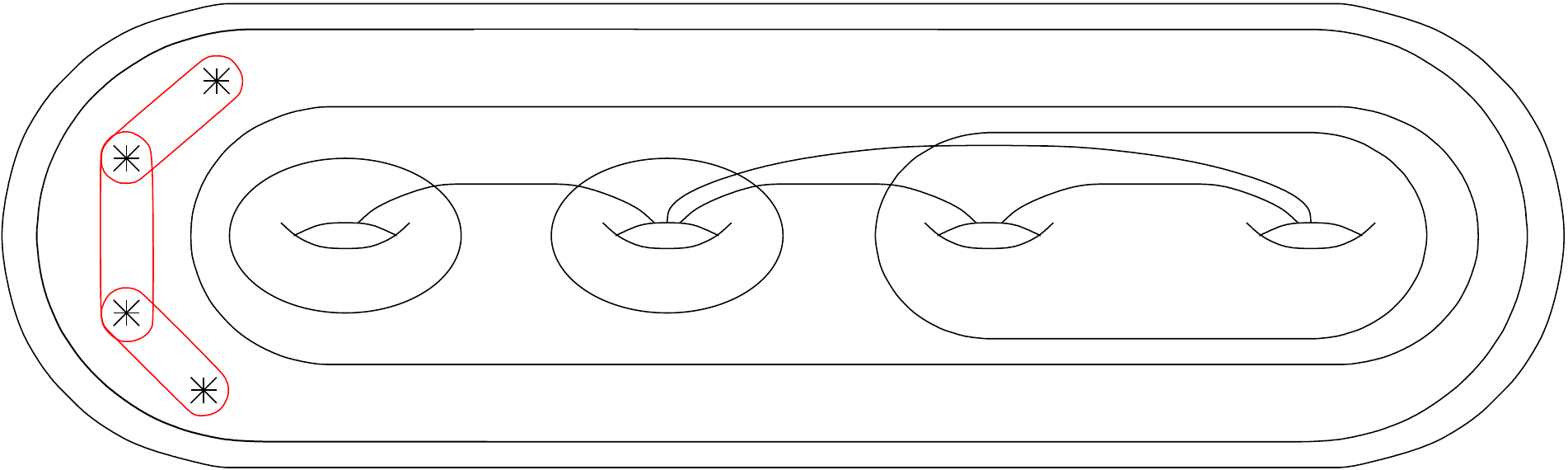} \\[0.2cm]
   \resizebox{10cm}{!}{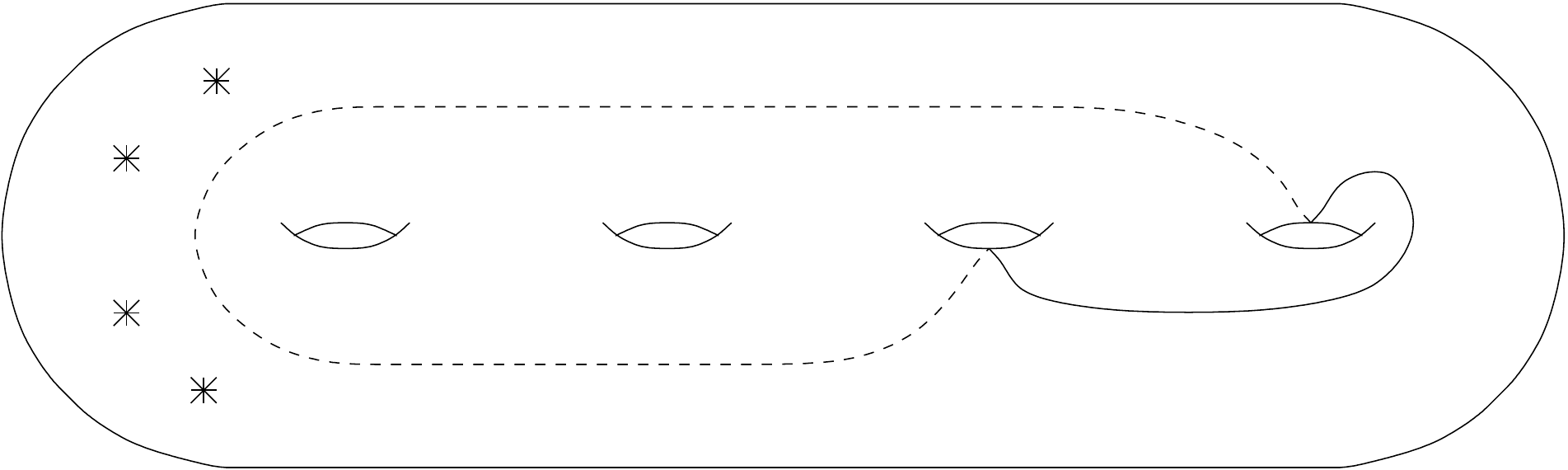} \caption{Examples of $C_{1^{+}}$ and \textcolor{red}{$E(C_{1^{+}})$} above, and the corresponding curve $\gamma_{+}$ below.}\label{DehnTwistPuncturedFig1}
  \end{center}
 \end{figure}\\
 \indent Letting $$C_{1^{+}}^{\prime} = \{\alpha_{1}, \ldots, \alpha_{2g-3}, [\alpha_{2g-2},\alpha_{2g-1},\alpha_{2g}]^{+}, [\alpha_{2g-2},\alpha_{2g-1},\alpha_{2g}]^{-}, \gamma_{+}\},$$ we then have that $\tau_{\alpha_{2g}}(\alpha_{2g-1}) = \langle C_{1^{+}}^{\prime} \cup E(C_{1^{+}}^{\prime})\rangle \in (\Cf \cup \Bf_{0})^{8}$ (see Figure \ref{DehnTwistPuncturedFig3}).
 \begin{figure}[h]
  \begin{center}
   \resizebox{10cm}{!}{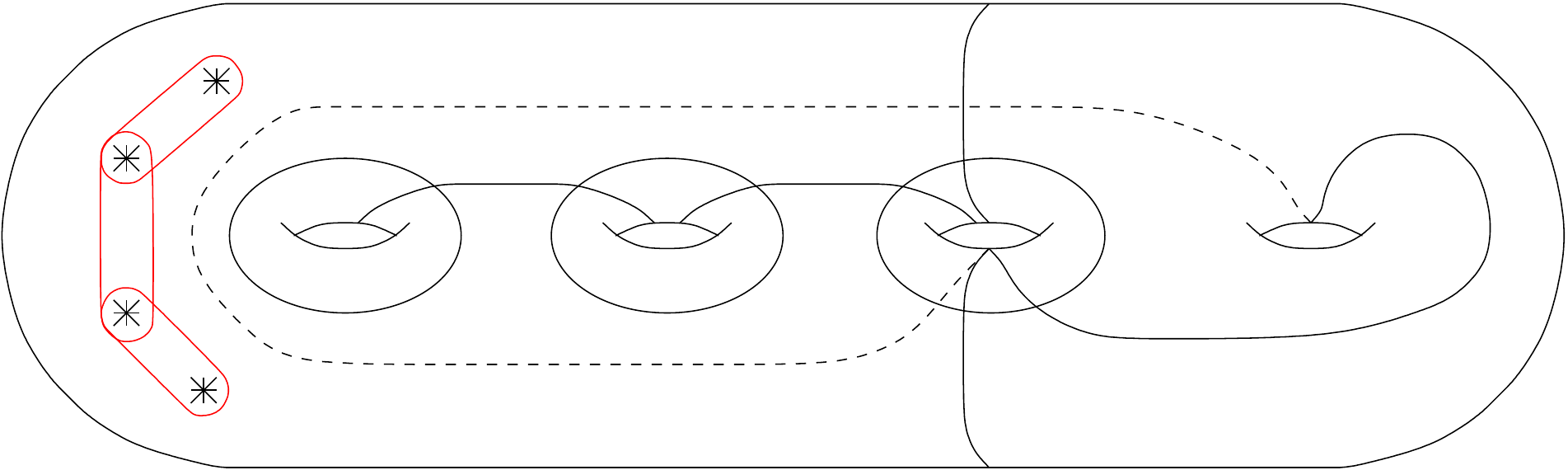} \\[0.2cm]
   \resizebox{10cm}{!}{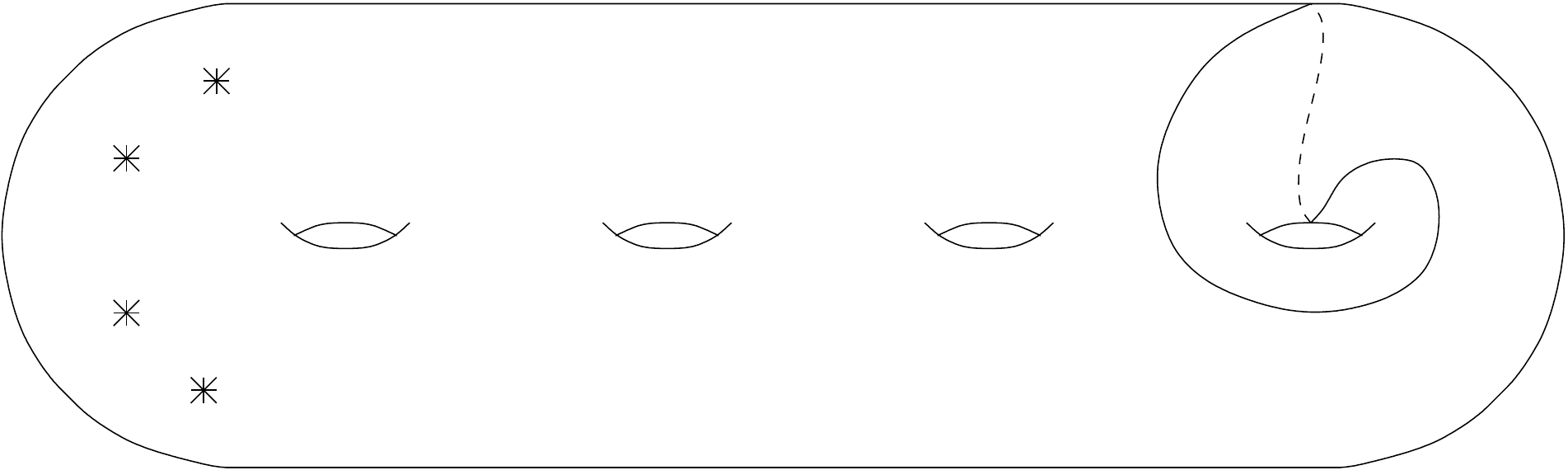} \caption{Examples of $C^{\prime}_{1^{+}}$ and \textcolor{red}{$E(C^{\prime}_{1^{+}})$} above, and $\tau_{\alpha_{2g}}(\alpha_{2g-1})$ below.} \label{DehnTwistPuncturedFig3}
  \end{center}
 \end{figure}\\
 \indent As in Subsection \ref{subsec4-2}, let $C_{-}$ be the set obtained by substituting $[\alpha_{2g-4}, \alpha_{2g-3}, \alpha_{2g-2}]^{+} \in C_{+}$ for $[\alpha_{2g-4}, \alpha_{2g-3}, \alpha_{2g-2}]^{-}$, $\gamma_{-} = \langle C_{-} \cup E(C_{-})\rangle$, and $C_{-}^{\prime}$ be the set obtained by substituting $\gamma_{+}$ in $C_{+}^{\prime}$ for $\gamma_{-}$. As before, we have that $\tau_{\alpha_{2g}}^{-1}(\alpha_{2g-1}) = \langle C_{-}^{\prime} \cup E(C_{-}^{\prime})\rangle \in (\Cf \cup \Bf_{0})^{8}$.
\end{proof}
\indent To prove that $\tau_{\alpha_{2g-1}}^{\pm 1}(\alpha_{2g-2}) \in (\Cf \cup \Bf_{0})^{8}$, we cannot proceed as in the closed case. This is due to the fact that for any choice of $i$, there are nonhomeomorphic connected components of $S \backslash C_{i}$. This implies it is possible that there is no homeomorphism that leaves $C_{i}$ invariant but sends $\alpha_{2g}$ into $\alpha_{2g-1}$. So, we first prove a proposition for an auxiliary curve used only in this proof, and then follow a method similar to Lemma \ref{taupm1v1CUB0}.
\begin{Prop}\label{translemaprop3}
 Let $k \in \mathbb{Z}$. Then we have that $[\alpha_{2k}, \ldots, \alpha_{2k + (2g-2)}]^{-} \in (\Cf \cup \Bf_{0})^{2}$ for any choice of $i \in \{0, \ldots, n\}$ (with $\alpha_{0} = \alpha_{0}^{i}$ when necessary).
\end{Prop}
\begin{proof}
 If $2k \neq 2$ (modulo $2g+2$) and $i \neq n$, then we have (see Figure \ref{DefAlpha2kUntilAlpha2kplus2gminus2minusfig1}) $$[\alpha_{2k}, \ldots, \alpha_{2k + (2g-2)}]^{-} = \langle\{\alpha_{0}^{j} : 0 \leq j \leq i\} \cup \{\alpha_{2k}, \ldots, \alpha_{2k+(2g-2)}\} \cup \{\alpha_{2k+2g}\} \cup \{\eps{j}{k} : i \leq j < k \leq n\} \rangle.$$
\begin{figure}[h]
\begin{center}
 \resizebox{10cm}{!}{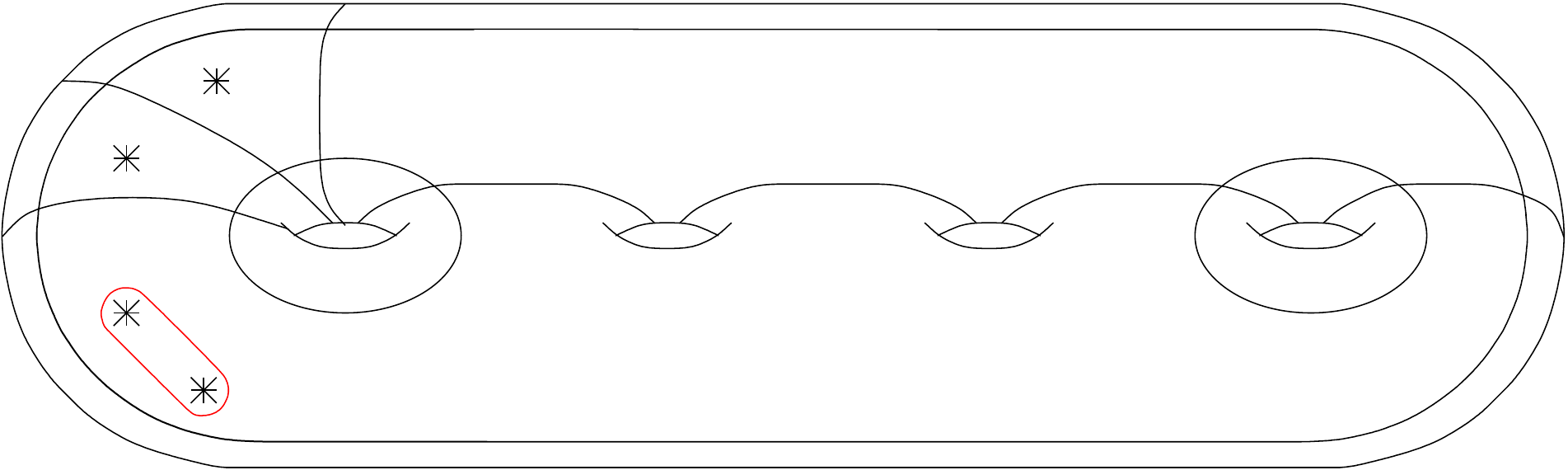} \\[0.3cm]
 \resizebox{10cm}{!}{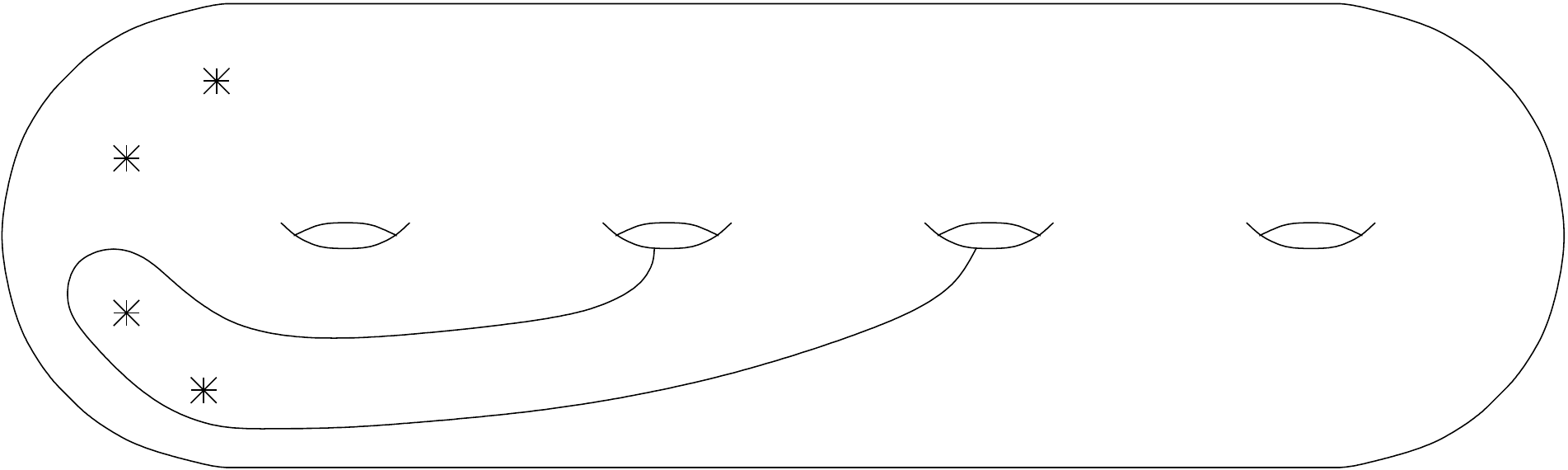} \caption{The curve $[\alpha_{2k}, \ldots, \alpha_{2k + (2g-2)}]^{-} = \langle\{\alpha_{0}^{j} : 0 \leq j \leq i\} \cup \{\alpha_{2k}, \ldots, \alpha_{2k+(2g-2)}\} \cup \{\alpha_{2k+2g}\} \cup \{\eps{j}{k} : i \leq j < k \leq n\} \rangle$.} \label{DefAlpha2kUntilAlpha2kplus2gminus2minusfig1}
\end{center}
\end{figure}\\
 \indent If $2k \neq 2$ (modulo $2g+2$) and $i = n$, then we have $[\alpha_{2k}, \ldots, \alpha_{2k + (2g-2)}]^{-} = \alpha_{2k +2g}$.\\
 \indent If $2k = 2$ (modulo $2g+2$) then $[\alpha_{2k}, \ldots, \alpha_{2k + (2g-2)}]^{-} = \beta_{\{2, \ldots, 2g-2\}}^{-} = \alpha_{0}^{n}$.
\end{proof}
\vspace{0.5cm}
\indent Now, we prove:
\begin{Lema}\label{taupm1v2CUB0}
 $\tau_{\alpha_{2g-1}}^{\pm 1}(\alpha_{2g-2}) \in (\Cf \cup \Bf_{0})^{8}$.
\end{Lema}
\begin{proof}
 \indent Let $i \in \{0, \ldots, n\}$. Taking the set $$C_{2^{+}} = \{\alpha_{2g}, \alpha_{0}^{i}, \alpha_{1}, \ldots, \alpha_{2g-5}, \alpha_{2g-3}, [\alpha_{2g-4}, \alpha_{2g-3}, \alpha_{2g-2}]^{+}, [\alpha_{2g-5},\alpha_{2g-4},\alpha_{2g-3}]^{+}, [\alpha_{0}^{i}, \ldots, \alpha_{2g-2}]^{\pm}\},$$
 then, by Lemma \ref{translemaprop1}, we get that $\gamma_{+} \ColonEqq \langle C_{2^{+}} \cup E(C_{2^{+}})\rangle \in (\Cf \cup \Bf_{0})^{7}$ (see Figure \ref{DefGammaplusBeforeRem2-11fig1}).
 \begin{figure}[h]
 \begin{center}
  \resizebox{10cm}{!}{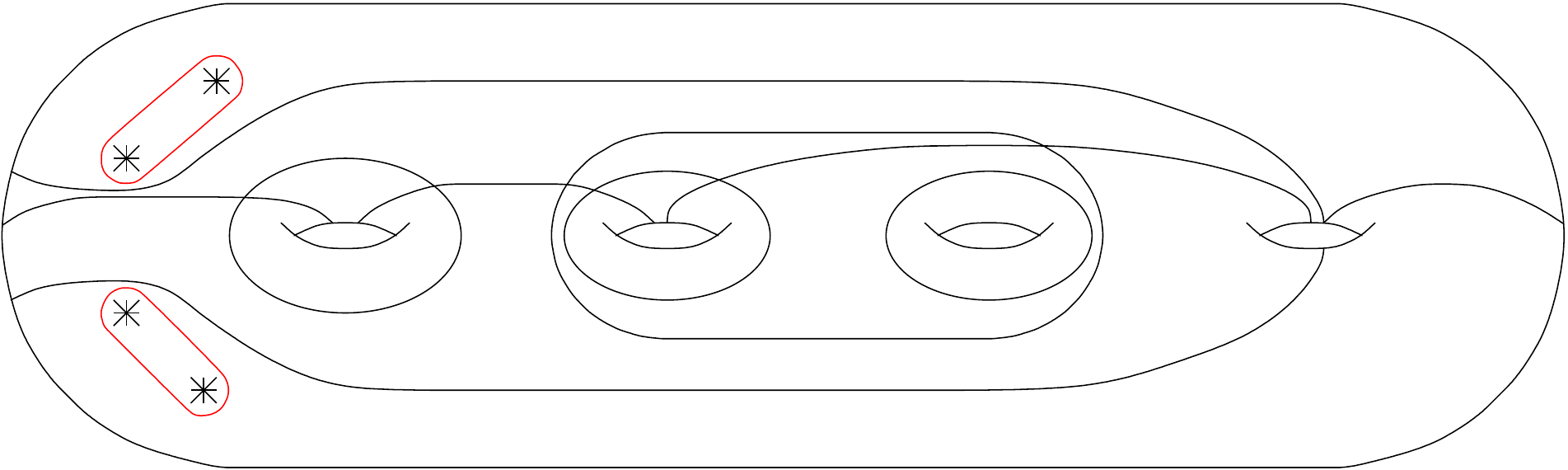}\\[0.3cm]
  \resizebox{10cm}{!}{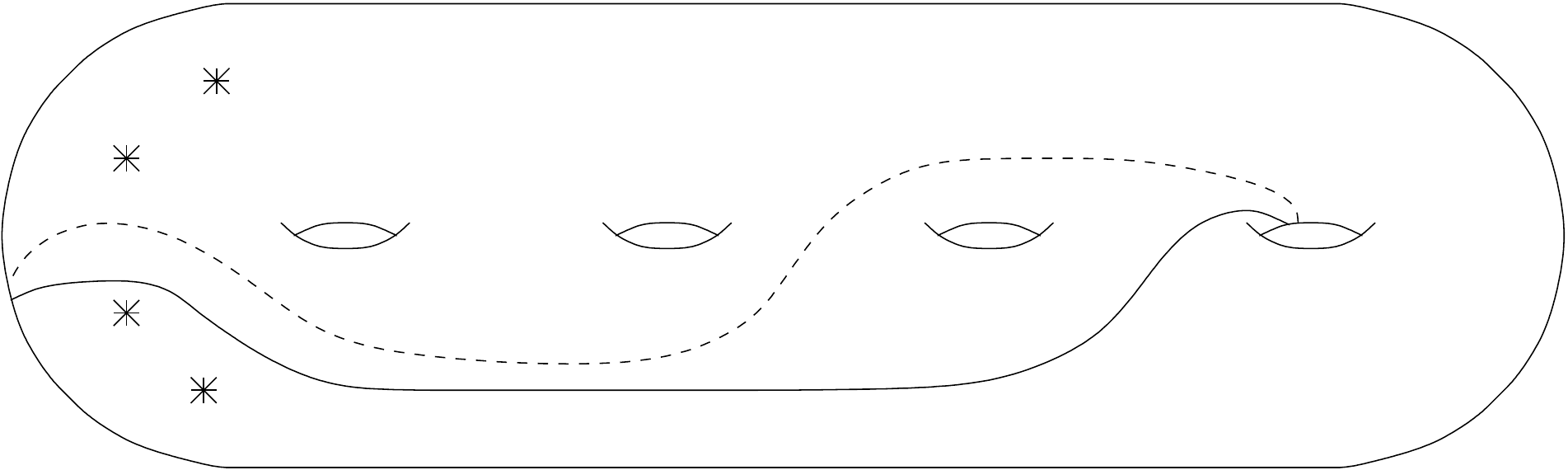} \caption{The curve $\gamma_{+} \ColonEqq \langle C_{2^{+}} \cup E(C_{2^{+}})\rangle $.} \label{DefGammaplusBeforeRem2-11fig1}
 \end{center}
 \end{figure}\\
 Letting $$C_{2^{+}}^{\prime} = \{\alpha_{0}^{i}, \ldots, \alpha_{2g-4}, [\alpha_{2g-3},\alpha_{2g-2},\alpha_{2g-1}]^{+}, [\alpha_{2g-3},\alpha_{2g-2},\alpha_{2g-1}]^{-}, \gamma_{+}\},$$ we then have that $\tau_{\alpha_{2g-1}}(\alpha_{2g-2}) = \langle C_{2^{+}}^{\prime} \cup E(C_{2^{+}}^{\prime})\rangle \in (\Cf \cup \Bf_{0})^{8}$ (see Figure \ref{DefTauofAlpha2gminus1onAlpha2gminus2}).
 \begin{figure}[h]
  \begin{center}
   \resizebox{10cm}{!}{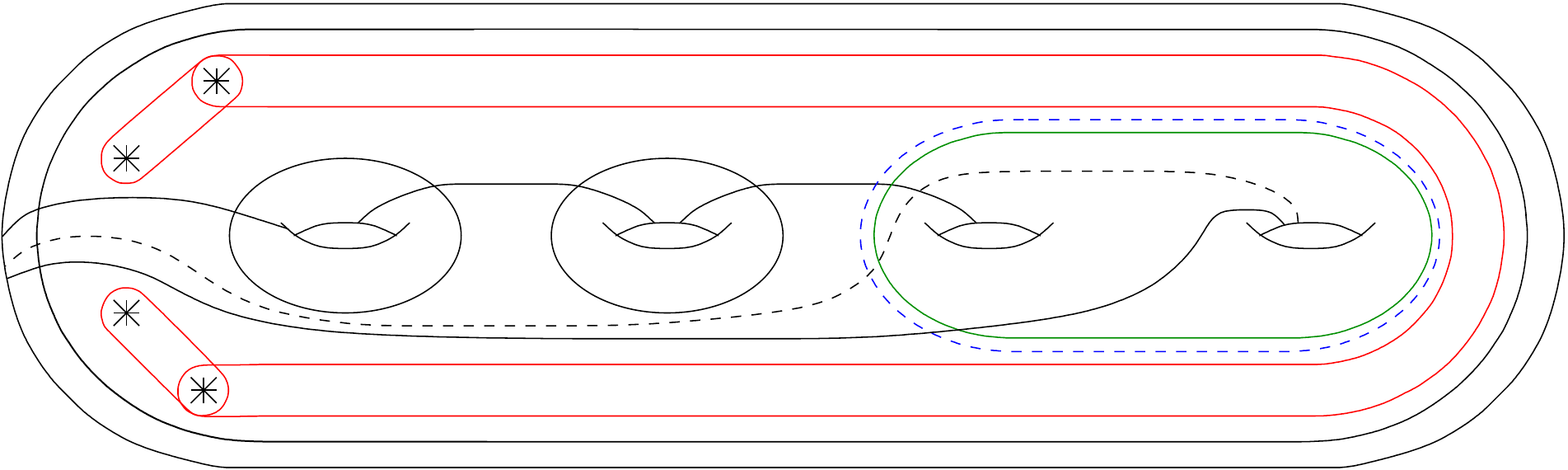} \\[0.3cm]
   \resizebox{10cm}{!}{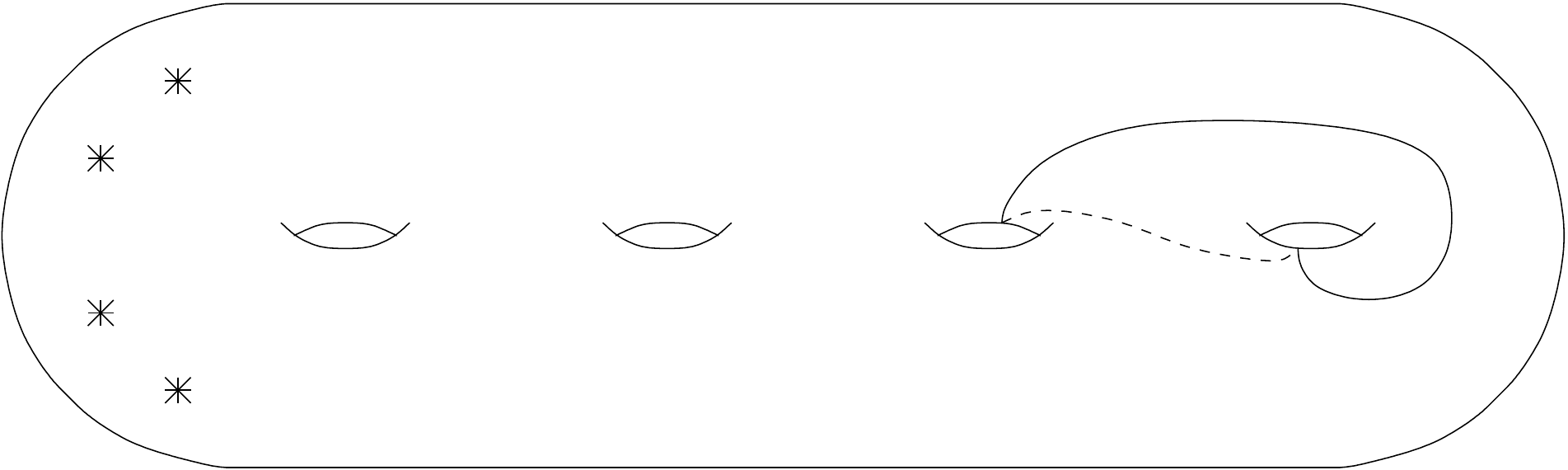} \caption{The curve $\tau_{\alpha_{2g-1}}(\alpha_{2g-2}) = \langle C_{2^{+}}^{\prime} \cup E(C_{2^{+}}^{\prime})\rangle \in (\Cf \cup \Bf_{0})^{8}$.} \label{DefTauofAlpha2gminus1onAlpha2gminus2}
  \end{center}
 \end{figure}\\
 \indent As in the case for $\tau_{\alpha_{2g}}^{-1}(\alpha_{2g-1})$ of Lemma \ref{taupm1v1CUB0}, substituting the analogous curves, we also have that $\tau_{\alpha_{2g-1}}^{-1}(\alpha_{2g-2}) \in (\Cf \cup \Bf_{0})^{8}$.
\end{proof}
\indent Now, let $h_{i}$ be the mapping class obtained by cutting $S$ along $C_{i}$ and rotating the resulting (sometimes punctured) discs so that $h_{i}(\alpha_{i}) = \alpha_{i+2}$ (with $\alpha_{0} = \alpha_{0}^{i}$). See Figure \ref{SpheresWithHandlesFig}.
\begin{figure}
 \begin{center}
  \resizebox{10cm}{!}{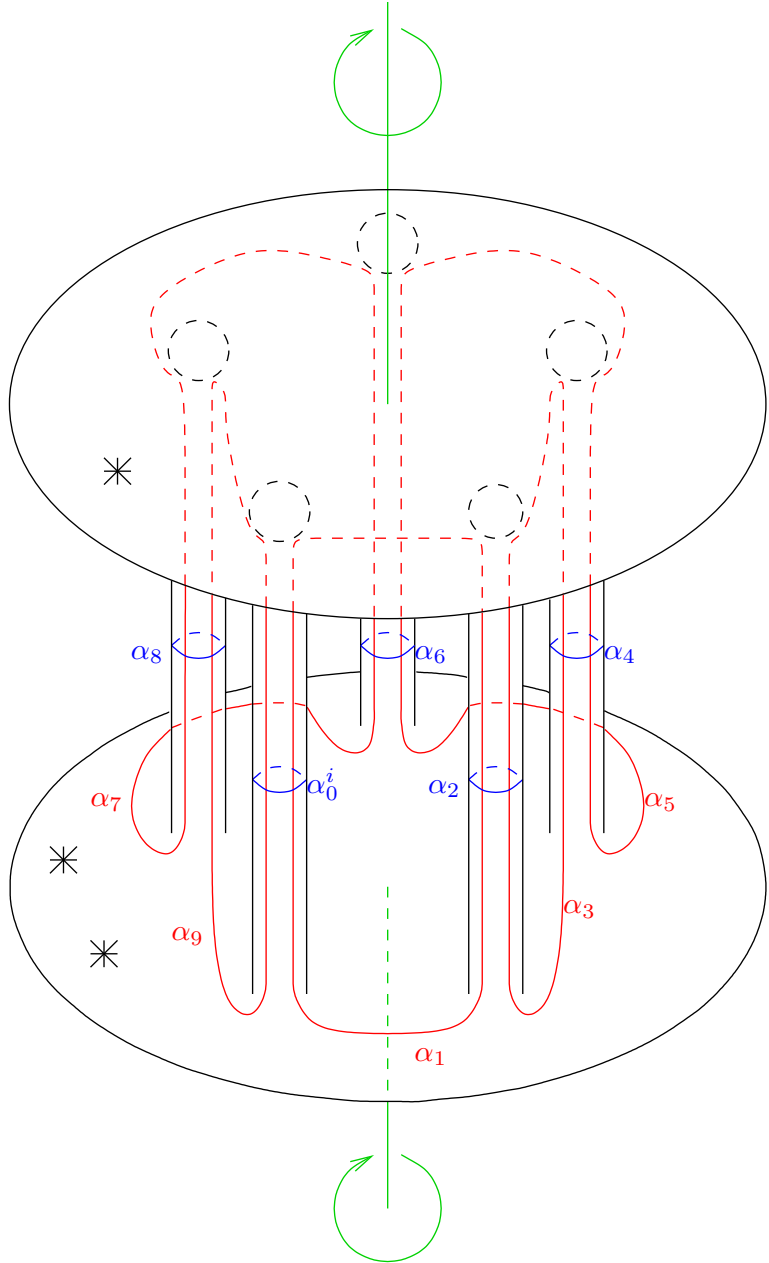}
 \end{center}\caption{An example of $h_{i}$ in $S_{4,3}$, for which $h_{i}(S_{i(o)}^{+}) = S_{i(o)}^{+}$, $h_{i}(S_{i(o)}^{-}) = S_{i(o)}^{-}$, $h_{i}(S_{i(e)}^{+}) = S_{i(e)}^{+}$ and $h_{i}(S_{i(e)}^{-}) = S_{i(e)}^{-}$. \textbf{In particular} $h_{i}(\alpha_{j}) = \alpha_{j+2}$. Also, $h_{i} \in \stab{E^{i,i}}$ and if $k \in \{1, \ldots, g-1\}$, $h_{i}(\epsp{k}{1}{n-1}) = \epsp{k+1}{1}{n-1}$.}\label{SpheresWithHandlesFig}
\end{figure}
\begin{Rem}
 Note that $h_{i}(S_{i(o)}^{+}) = S_{i(o)}^{+}$, $h_{i}(S_{i(o)}^{-}) = S_{i(o)}^{-}$, $h_{i}(S_{i(e)}^{+}) = S_{i(e)}^{+}$ and $h_{i}(S_{i(e)}^{-}) = S_{i(e)}^{-}$. Also, $h_{i} \in \stab{E^{i,i}}$ and if $k \in \{1, \ldots, g-1\}$, $h_{i}(\epsp{k}{1}{n-1}) = \epsp{k+1}{1}{n-1}$.
\end{Rem}
\begin{proof}[\textbf{Proof of Claim 1:}] Using the same arguments as in Subsection \ref{subsec4-2} and Lemmas \ref{translemaprop1}, \ref{translemaprop2}, \ref{translemaprop3}, we can precompose by appropriate elements of the group $\langle h_{i} \rangle$ on $C_{i}$ to translate of the elements of $C_{j^{+}}$, $C_{j^{+}}^{\prime}$ and the corresponding sets for the negative exponents of the Dehn twists, and following the procedure for $\tau_{\alpha_{2g}}^{\pm 1}(\alpha_{2g-1}), \tau_{\alpha_{2g-1}}^{\pm 1}(\alpha_{2g-2}) \in (\Cf \cup \Bf_{0})^{8}$ we have that $\tau_{\alpha_{j}}^{\pm 1}(\alpha_{j-1}) \in (\Cf \cup \Bf_{0})^{8}$. Therefore $\tau_{\Cf}^{\pm 1}(\Cf) \subset (\Cf \cup \Bf_{0})^{8}$.
\end{proof}
\subsection{Proof of Claim 2: $\tau_{\Cf}^{\pm 1}(\Cf \cup \Bf_{0}) \subset (\Cf \cup \Bf_{0})^{12}$}\label{chap2sec5}
\indent Let $\alpha \in \Cf$ and $\beta \in \Bf_{0}$; if $i(\alpha,\beta) = 0$, we have that $\tau_{\alpha}^{\pm 1}(\beta) = \beta \in \Cf \cup \Bf_{0}$. So, we assume this is not the case. Now, to prove the claim we first see that every curve in $\Bf_{0}$ can be taken to be a curve uniquely determined by a set $C \cup E \cup B$ such that $C \subset \Cf$, $E \subset \Ef$, $B \subset \Bf$, and with every element in $B$ disjoint from $\alpha$.
Since $\tau_{\alpha}^{\pm 1}$ are mapping classes, we get that $\tau_{\alpha}^{\pm 1}(\beta) = \langle \tau_{\alpha}^{\pm 1}(C \cup E \cup B) \rangle = \langle \tau_{\alpha}^{\pm 1}(C) \cup \tau_{\alpha}^{\pm 1}(E) \cup B) \rangle$. Using the result in Claim 1 we get that $\tau_{\alpha}^{\pm 1}(C) \subset (\Cf \cup \Bf_{0})^{8}$, and by Proposition \ref{OnlyCUB0} $\tau_{\alpha}^{\pm 1}(E) \subset (\Cf \cup \Bf_{0})^{11}$. Therefore $\tau_{\alpha}^{\pm 1}(\beta) \in (\Cf \cup \Bf_{0})^{12}$.\\
\indent For a detailed account of the sets $C$, $E$ and $B$, see \cite{Thesis}.\\
\subsection{Proof of Claim 3: $\tau_{\zeta}^{\pm 1}(\Cf) \subset (\Cf \cup \Bf_{0})^{8}$}\label{chap2sec6}
\indent Recall $\zeta = \beta_{\{2g-2,2g-1,2g\}}^{+}$, which is disjoint from every element in $\Cf \backslash \{\alpha_{2g-3},\alpha_{2g+1}\}$. This implies we only need to prove that $\tau_{\zeta}^{\pm 1}(\alpha_{2g-3}), \tau_{\zeta}^{\pm 1}(\alpha_{2g+1}) \in (\Cf \cup \Bf_{0})^{8}$.\\
\indent We proceed as in Claim 1, using the following ordered maximal closed chain $$\gamma_{0} = \alpha_{0}^{1},\gamma_{1} = \alpha_{1}, \ldots, \gamma_{2g-5} = \alpha_{2g-5}, \gamma_{2g-4} = \beta_{\{2g-4,2g-3,2g-2\}}^{-},$$ $$\gamma_{2g-3} = \alpha_{2g-1}, \gamma_{2g-2} = \alpha_{2g-2}, \gamma_{2g-1} = \alpha_{2g-3}, \gamma_{2g} = \zeta, \gamma_{2g+1} = \alpha_{2g+1},$$ (see Figure \ref{ThirdChain1}) and proving first that $\tau_{\zeta}(\alpha_{2g-3}) \in (\Cf \cup \Bf_{0})^{8}$.
\begin{figure}[h]
\begin{center}
 \includegraphics[width=10cm]{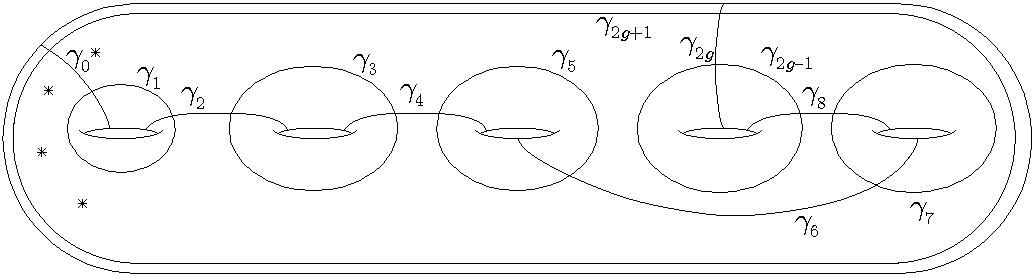} \caption{The ordered maximal closed chain for $S = S_{5,4}$, which is used for the case of $\tau_{\zeta}^{\pm 1}(\alpha_{2g-3})$.} \label{ThirdChain1}
\end{center}
\end{figure}\\
\indent Using the set 
\begin{center}
 \begin{tabular}{rcl}
  $C_{1^{+}}$ & $=$ & $\{\alpha_{2g+1}, \alpha_{1}, \alpha_{2}, \ldots, \alpha_{2g-5}, \beta_{\{2g-4,2g-3,2g-2\}}^{-}, \alpha_{2g-2},$\\
   & & $[\alpha_{2g-3},\alpha_{2g-2},\alpha_{2g-1}]^{+}, \beta_{\{2,\ldots, 2g-4\}}^{+}, [\alpha_{1}, \ldots, \alpha_{2g-1}]^{+}\}$\\
   & & for genus $g \geq 4$,\\
   & $=$ & $\{\alpha_{7}, \alpha_{1}, \alpha_{2}, \beta_{\{2,3,4\}}^{-}, \alpha_{4}, [\alpha_{3},\alpha_{4},\alpha_{5}]^{+}, [\alpha_{1}, \ldots, \alpha_{5}]^{+}\}$\\
   & & for genus $g=3$,
 \end{tabular}
\end{center}
and Propositions \ref{translemaprop1} and \ref{translemaprop2} we obtain the curve $\gamma_{+} = \langle C_{1^{+}} \cup \Df \rangle \in (\Cf \cup \Bf_{0})^{7}$. Then, using the set $$C_{1^{+}}^{\prime} = \{\alpha_{1}, \ldots, \alpha_{2g-5}, \beta_{\{2g-4,2g-3,2g-2\}}^{-}, \alpha_{2g-2}, \alpha_{2g-3}, \zeta, \alpha_{2g}\},$$ we obtain the curve $\gamma_{+}^{\prime} = \langle \Cff \cup C_{1^{+}}^{\prime} \rangle \in (\Cf \cup \Bf_{0})^{1}$. Finally, using the set $$C_{1^{+}}^{\biprime} = \{\alpha_{1}, \ldots, \alpha_{2g-5}, \beta_{\{2g-4,2g-3,2g-2\}}^{-}, \alpha_{2g-1}, \gamma_{+}^{\prime},\alpha_{2g}, \gamma_{+}\},$$ we obtain that $\tau_{\zeta}(\alpha_{2g-3}) = \langle C_{1^{+}}^{\biprime} \cup \Cff \rangle \in (\Cf \cup \Bf_{0})^{8}$.\\
\indent For $\tau_{\zeta}^{-1}(\alpha_{2g-3})$ we proceed analogously, substituting the appropriate curves, getting that $\tau_{\zeta}^{-1}(\alpha_{2g-3}) \in (\Cf \cup \Bf_{0})^{8}$.\\
\indent To prove that $\tau_{\zeta}^{\pm 1}(\alpha_{2g+1}) \in (\Cf \cup \Bf_{0})^{8}$ we proceed analogously, using the ordered maximal closed chain $$\gamma_{0} = \alpha_{2g-2}, \gamma_{1} = \alpha_{2g-1}, \gamma_{2} = \beta_{\{2g-4,2g-3,2g-2\}}^{-}, \gamma_{3} = \alpha_{2g-5},$$ $$\gamma_{4} = \alpha_{2g-6}, \ldots, \gamma_{2g-4} = \alpha_{2}, \gamma_{2g-3} = \alpha_{1}, \gamma_{2g-2} = \alpha_{0}^{1}, \gamma_{2g-1} = \alpha_{2g+1}, \gamma_{2g} = \zeta, \gamma_{2g+1} = \alpha_{2g-3}$$ (see Figure \ref{ThirdChain2}). Note that this closed chain as a set, is the same closed chain as the previous case but with the order reversed.
\begin{figure}[h]
\begin{center}
 \includegraphics[width=10cm]{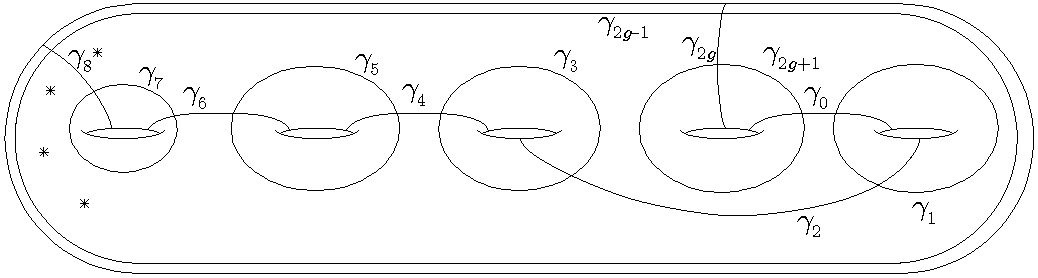} \caption{The ordered maximal closed chain for $S = S_{5,4}$, which is used for the case of $\tau_{\zeta}^{\pm 1}(\alpha_{2g+1})$.} \label{ThirdChain2}
\end{center}
\end{figure}\\
\indent We have then $\tau_{\zeta}^{\pm 1}(\alpha_{2g+1}) \in (\Cf \cup \Bf_{0})^{8}$.
\subsection{Proof of Claim 4: $\tau_{\zeta}^{\pm 1}(\Cf \cup \Bf_{0}) \subset (\Cf \cup \Bf_{0})^{10}$}\label{chap2sec7}
\indent Given that $\zeta$ is disjoint from every curve $\beta \in \Bf_{0}$ of the form $\beta_{\{2l, \ldots, 2(l+k)\}}^{-}$ for some $l \in \nat$ and $k \in \mathbb{Z}^{+}$, it follows that $\tau_{\zeta}^{\pm 1}(\beta_{\{2l, \ldots, 2(l+k)\}}^{-}) = \beta_{\{2l, \ldots, 2(l+k)\}}^{-} \in \Bf_{0}$. So, we need to prove the result for $\beta = \beta_{\{2l, \ldots, 2(l+k)\}}^{+}$; we do so dividing into several cases in the following way (See Figure \ref{ExamplesChap2Claim4Fig1} for examples):
\begin{enumerate}
 \item $\beta$ is of the form $\beta_{\{2l, \ldots, 2(l+k)\}}^{+}$ for $l < l+k$.
 \item $\beta$ is of the form $\beta_{\{2l, \ldots, 2(l+k)\}}^{+}$ for $l > l+k$.
\end{enumerate}
\begin{figure}[h]
 \begin{center}
  \resizebox{9cm}{!}{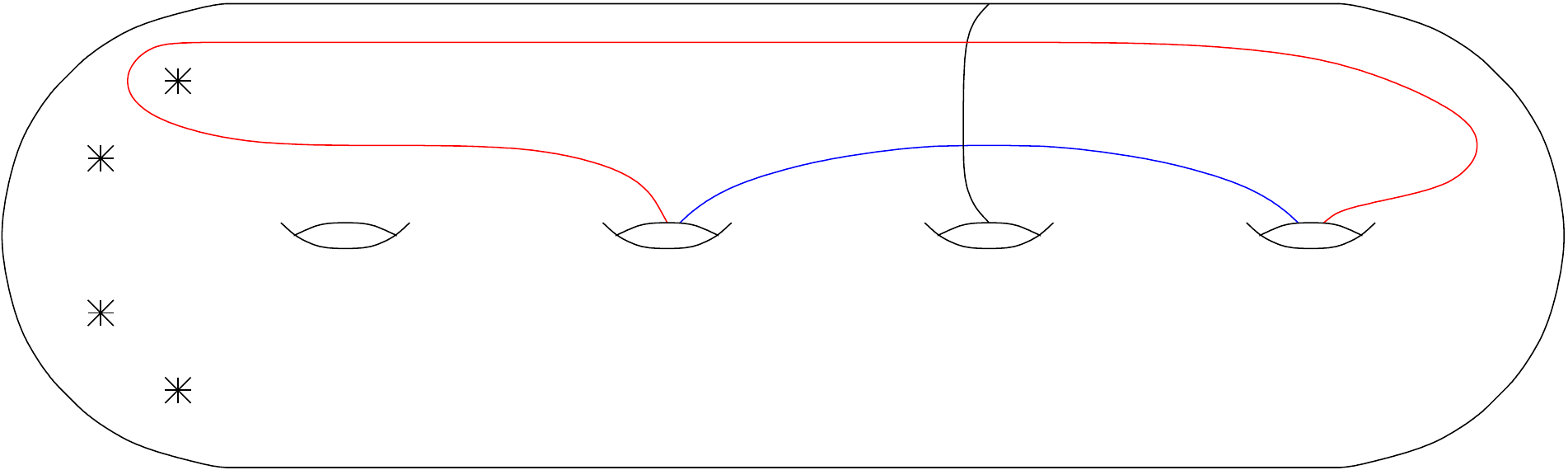} \caption{Examples of $\beta$ for the first case (in blue), the second case (in red), and how they intersect $\zeta$.} \label{ExamplesChap2Claim4Fig1}
 \end{center}
\end{figure}
\indent To prove the claim, as in Claim 2 in Subsection \ref{chap2sec5}, we just have to remember that every curve in $\Bf_{0}$ can be taken to be a curve uniquely determined by a set $C \cup E \cup B$ such that $C \subset \Cf$, $E \subset \Ef$, $B \subset \Bf$, and with every element in $B$ disjoint from $\zeta$. For a detailed account of these sets see \cite{Thesis}.
\subsection{Proof of Claim 5: $\eta_{\Hff}^{\pm 1}(\Cf) \subset (\Cf \cup \Bf_{0})^{7}$}\label{chap2sec8}
\indent Recall $\Hff \ColonEqq \{\epsilon^{i-2,i} \in \Df: 2 \leq i \leq n\}$. Given that $i(\alpha,\eps{i-2}{i}) = 0$ for all $\alpha \in \Cf \backslash \{\alpha_{0}^{i-1}\}$, to prove the claim we just need to prove that $\eta_{\eps{i-2}{i}}^{\pm 1}(\alpha_{0}^{i-1}) \in (\Cf \cup \Bf_{0})^{7}$. We do so following \cite{Ara2}.\\
\indent Let $i \in \{1,\ldots,n\}$; we define (see Figure \ref{HalfTwist1}) $$\beta^{i} = \langle \{\alpha_{1}\} \cup \{\alpha_{3}, \ldots, \alpha_{2g+1}\} \cup \{\beta_{\{4, \ldots, 2g\}}^{\pm}\} \cup \{\beta_{0,1,2}^{j,+}: j < i\} \cup \{\beta_{0,1,2}^{k,-}: i \leq k\}\rangle \in (\Cf \cup \Bf_{0})^{5}.$$
\begin{figure}[h]
\begin{center}
 \includegraphics[width=10cm]{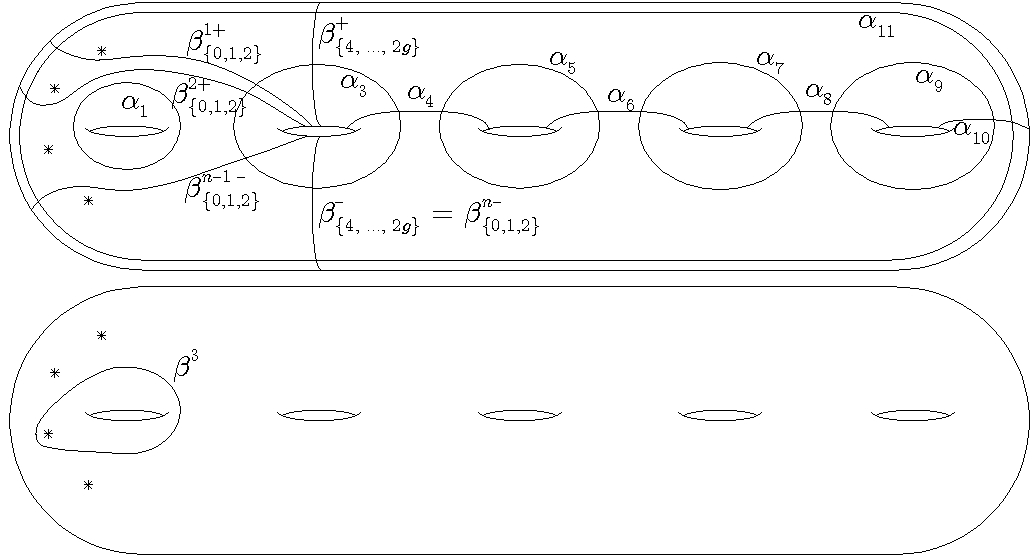} \caption{Above, the curves that uniquely determine $\beta^{3}$; below, the curve $\beta^{3}$.} \label{HalfTwist1}
\end{center}
\end{figure}\\
\indent Then, for $i \in \{1, \ldots, n\}$, we define (see Figure \ref{HalfTwist2}) $$\gamma_{+}^{i} = \langle \{\alpha_{0}^{0},\alpha_{1},\alpha_{2}\} \cup \{\alpha_{4}, \ldots, \alpha_{2g}\} \cup \{\beta_{\{4,\ldots,2g\}}^{\pm}\} \cup \{\beta^{j}: i \neq j\}\rangle \in (\Cf \cup \Bf_{0})^{6},$$ $$\gamma_{-}^{i} = \langle \{\alpha_{0}^{n}, \alpha_{1},\alpha_{2}\} \cup \{\alpha_{4}, \ldots, \alpha_{2g}\} \cup \{\beta_{\{4,\ldots, 2g\}}^{\pm}\} \cup \{\beta^{j}:i \neq j\}\rangle \in (\Cf \cup \Bf_{0})^{6}.$$
\begin{figure}[h]
\begin{center}
 \resizebox{10cm}{!}{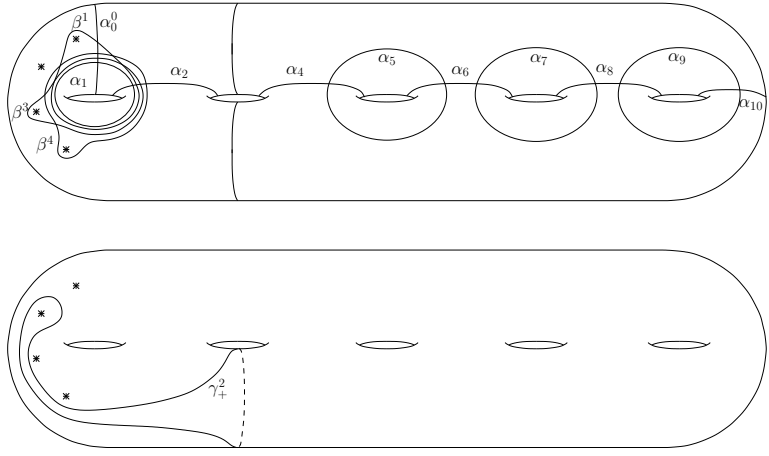} \caption{Above, the curves that uniquely determine $\gamma_{+}^{2}$; below, the curve $\gamma_{+}^{2}$.} \label{HalfTwist2}
\end{center}
\end{figure}\\
\indent Finally, for $i \in \{2, \ldots, n\}$, we get that $\eta_{\eps{i-2}{i}}(\alpha_{0}^{i-1}) = \langle (\Cf \backslash \{\alpha_{2g+1},\alpha_{0}^{i-1},\alpha_{1}\}) \cup \{\gamma_{+}^{i-1}\}\rangle$ and $\eta_{\eps{i-2}{i}}^{-1}(\alpha_{0}^{i-1}) = \langle (\Cf \backslash \{\alpha_{2g+1}, \alpha_{0}^{i-1}, \alpha_{1}\}) \cup \{\gamma_{-}^{i}\}\rangle$ (see Figure \ref{HalfTwist3}). Therefore $\eta_{\Hff}^{\pm 1}(\Cf) \subset (\Cf \cup \Bf_{0})^{7}$.
\begin{figure}[h]
\begin{center}
 \includegraphics[width=10cm]{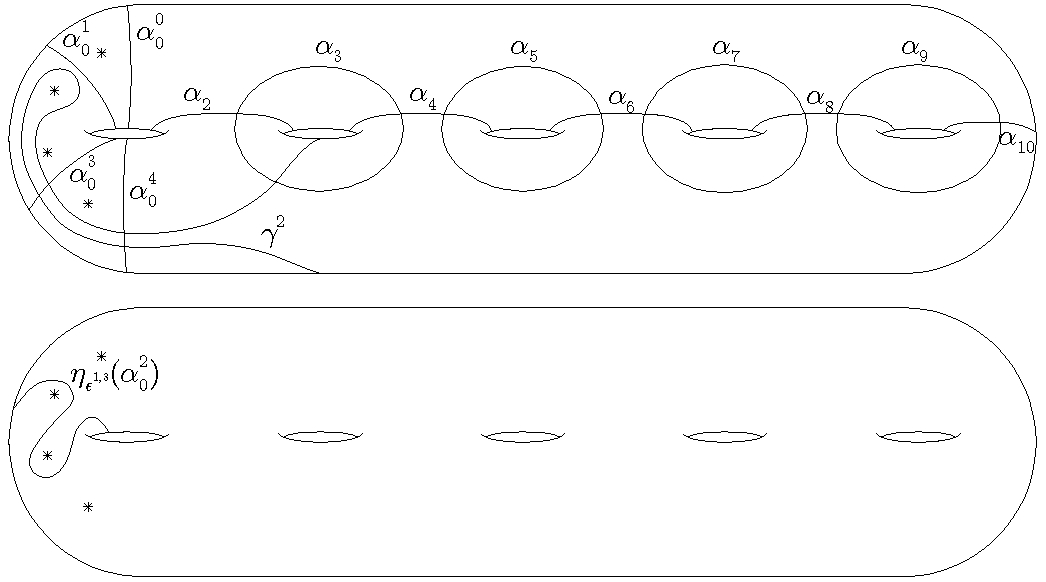} \caption{Above, the curves that uniquely determine $\eta_{\eps{1}{3}}(\alpha_{0}^{2})$; below the curve $\eta_{\eps{1}{3}}(\alpha_{0}^{2})$.} \label{HalfTwist3}
\end{center}
\end{figure}
\subsection{Proof of Claim 6: $\eta_{\Hff}^{\pm 1}(\Cf \cup \Bf_{0}) \subset (\Cf \cup \Bf_{0})^{11}$}\label{chap2sec9}
\indent Let $\beta \in \Bf_{0}$, as such it is of the form $\beta_{\{2l, \ldots, 2(l+k)\}}^{\pm}$ for some $l \in \nat$ and some $k \in \mathbb{Z}^{+}$.\\
\indent If $0 < l < l+k$, then $\beta$ and $\eps{i}{i+2}$ are disjoint for $i \in \{0, \ldots, n-2\}$. This implies that $\eta_{\eps{i}{i+2}}^{\pm 1}(\beta) = \beta$.\\
\indent If $l = 0$ and $k \in \{2, \ldots, g-1\}$, either we have that (see Figure \ref{ExamplesChap2Claim6Fig1}) $$\beta = \left\langle (\Cf \backslash \{\alpha_{0}^{0}, \alpha_{2g+1}, \alpha_{2k+1}\}) \cup \left(\bigcup_{l \in \{k+1, \ldots, g\}} \epsp{l}{1}{n-1}\right)\right\rangle,$$
\begin{figure}[h]
 \begin{center}
  \resizebox{9cm}{!}{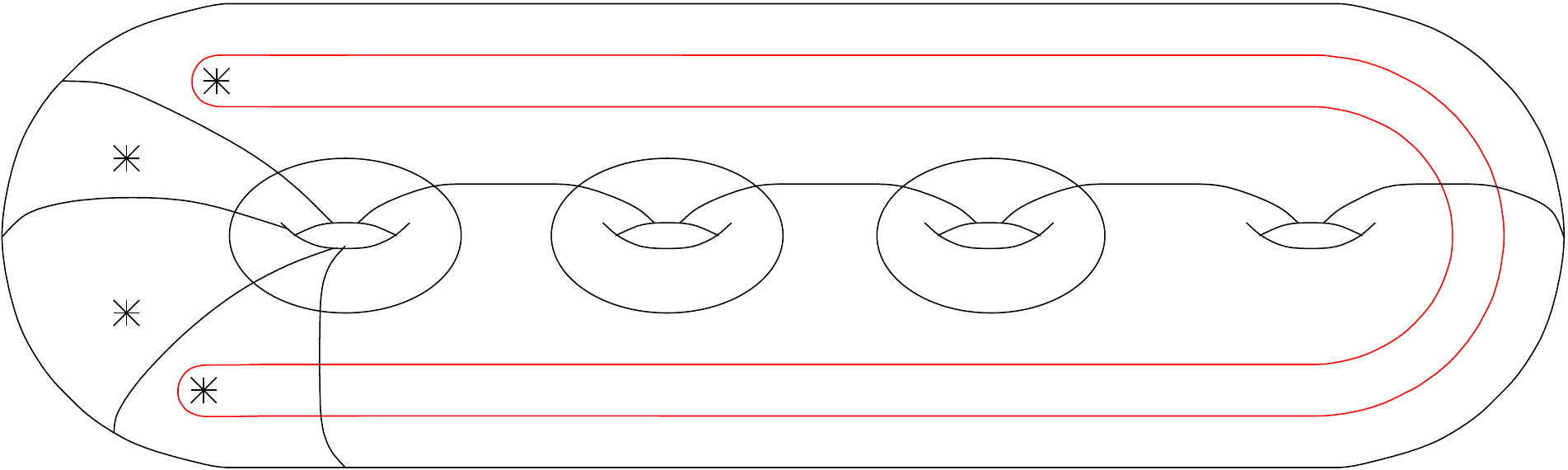} \\[0.2cm]
  \resizebox{9cm}{!}{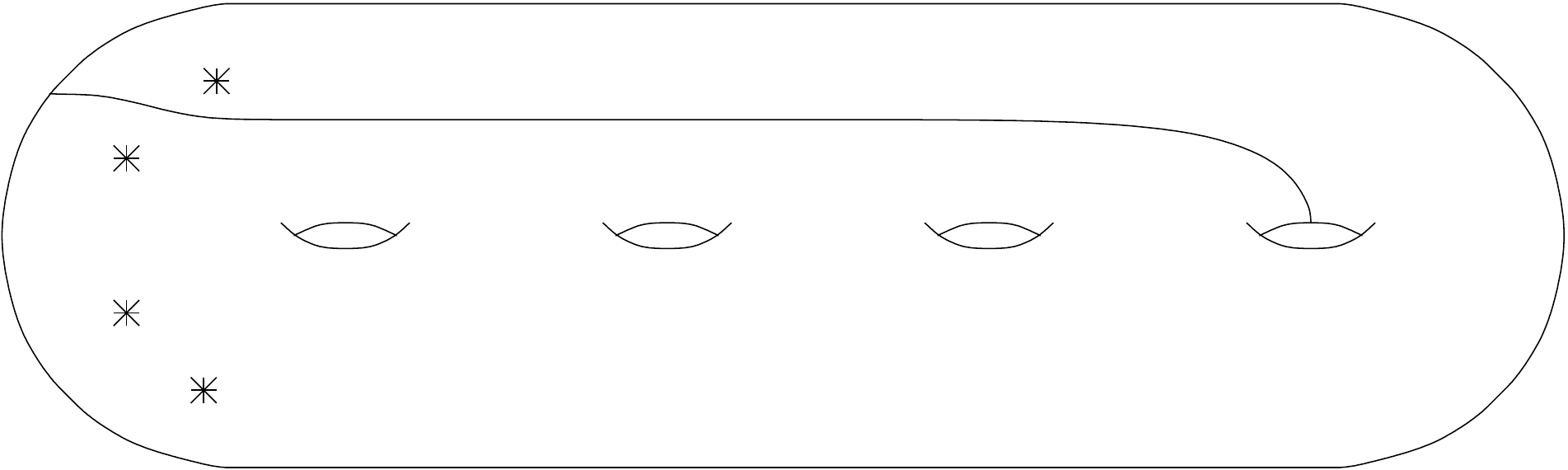} \caption{Examples of $\beta = \beta_{\{2l, \ldots, 2(l+k)\}}^{+}$ with $l = 0$ and $k = 3$. The elements in $\Ef$ are coloured red.} \label{ExamplesChap2Claim6Fig1}
 \end{center}
\end{figure}\\
or we have that (see Figure \ref{ExamplesChap2Claim6Fig3}) $$\beta = \left\langle (C_{1} \backslash \{\alpha_{2k+1}, \alpha_{2g+1}\}) \cup \{\alpha_{0}^{0}\} \cup E^{1,1} \cup \left(\bigcup_{l \in \{k+1, \ldots, g\}} \epsp{l}{1}{n-1}\right)\right\rangle.$$ 
\begin{figure}[h]
 \begin{center}
  \resizebox{9cm}{!}{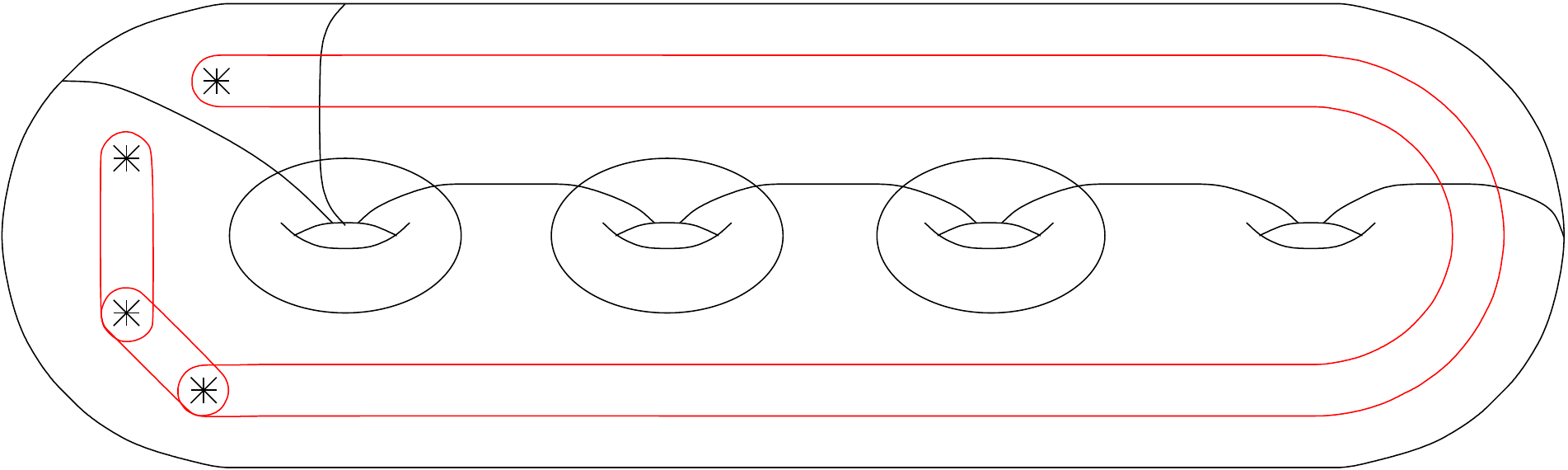} \\[0.2cm]
  \resizebox{9cm}{!}{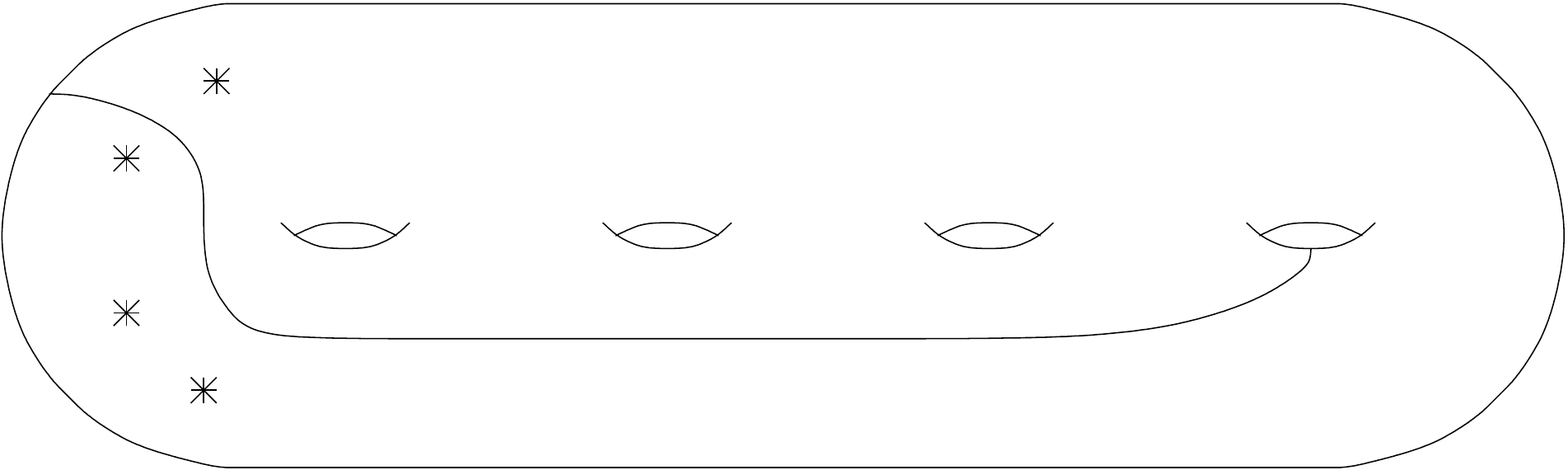} \caption{Examples of $\beta = \beta_{\{2l, \ldots, 2(l+k)\}}^{-}$ with $l = 0$ and $k = 3$. The elements in $\Ef$ are coloured red.} \label{ExamplesChap2Claim6Fig3}
 \end{center}
\end{figure}\\
\indent Using Claim 5 and Proposition \ref{OnlyCUB0}, we obtain that $\eta_{\eps{i}{i+2}}^{\pm 1}(\beta) \in (\Cf \cup \Bf_{0})^{11}$ for $i \in \{0, \ldots, n-2\}$.\\
If $\beta$ is of the form $\beta_{\{2l, \ldots, 2(l+k)\}}^{\pm}$ for $l > l+k$, following the proof of Claim 2, we have that $\beta = \langle C \cup E \cup B\rangle$ with $C \subset \Cf$, $E \subset \Ef$ and $B$ a singleton of a curve disjoint from $\eps{i}{i+2}$. Using the result from Claim 5 we get that $\eta_{\eps{i}{i+2}}^{\pm 1}(C) \subset (\Cf \cup \Bf_{0})^{7}$, and applying Proposition \ref{OnlyCUB0} we obtain that $\eta_{\eps{i}{i+2}}^{\pm 1}(E) \subset (\Cf \cup \Bf_{0})^{10}$. Finally, this implies that $\eta_{\eps{i}{i+2}}^{\pm 1}(\beta) \in (\Cf \cup \Bf_{0})^{11}$, for $i \in \{0, \ldots, n-2\}$.
\section{Rigid sets}\label{chap3}
\indent In this section we suppose $S = S_{g,n}$ with genus $g \geq 3$, and $n \geq  0$ punctures. Here we reintroduce the finite rigid set from \cite{Ara1} (see Subsections \ref{chap3sec1subsec1} for the closed surface case and \ref{chap3sec1subsec2} for the punctured surface case, bellow) and prove Theorem \ref{Xexhausts}.
\subsection{$\X(S)$ for closed surfaces}\label{chap3sec1subsec1}
\indent Let $\Cf$ and $\Bf$ be as in Section \ref{chap1}, and $J$ be a subinterval (modulo $2g+2$) of $\{0, \ldots, 2g+1\}$ such that $|J| < 2g -1$.\\
\indent If $J = \{j, \ldots, j+ 2k-1\}$ for some $j \in \nat$ and $k \in \mathbb{Z}^{+}$, we get the following curve: $$\sigma_{J} \ColonEqq \langle \{\alpha_{j}, \ldots, \alpha_{j+2k-1}\} \cup \{\alpha_{j + 2k +1}, \ldots, \alpha_{j-2}\}\rangle.$$ For examples see Figure \ref{ExamplesSigmaSec3fig1}. 
\begin{figure}[h]
 \begin{center}
  \resizebox{10cm}{!}{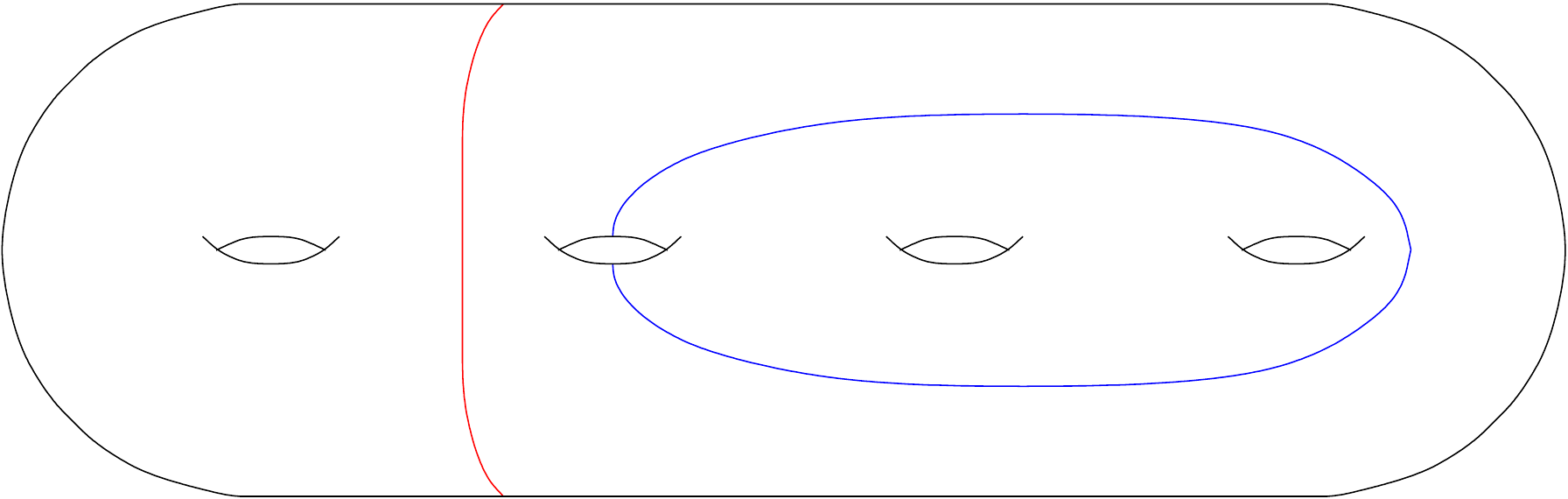} \caption{Examples of $\sigma_{\{0,1\}}$ in red, and $\sigma_{\{4,5,6,7\}}$ in blue.}\label{ExamplesSigmaSec3fig1}
 \end{center}
\end{figure}\\
This way, we define the following set: $$\Sf \ColonEqq \{\sigma_{J} : |J| = 2k \hspace{0.2cm} \mathrm{for} \hspace{0.2cm} \mathrm{some} \hspace{0.2cm} k \in \mathbb{Z}^{+}\}.$$
\indent If $J = \{j, \ldots, j + 2k\}$ for some $j \in \nat$ and $k \in \mathbb{Z}^{+}$, let we get the following curves:$$\mu_{j+2k+1, J}^{+} \ColonEqq \langle \{\beta_{J}^{+},\alpha_{j+2k+1}\} \cup \{\alpha_{j}, \ldots, \alpha_{j+2k-1}\} \cup \{\beta_{\{j+2, \ldots, j+2k+2\}}^{-}\} \cup \{\alpha_{j+2k+3}, \ldots, \alpha_{j-2}\} \rangle,$$ $$\mu_{j+2k+1, J}^{-} \ColonEqq \langle \{\beta_{J}^{-},\alpha_{j+2k+1}\} \cup \{\alpha_{j}, \ldots, \alpha_{j+2k-1}\} \cup \{\beta_{\{j+2, \ldots, j+2k+2\}}^{+}\} \cup \{\alpha_{j+2k+3}, \ldots, \alpha_{j-2}\} \rangle.$$ 
\indent For examples see Figure \ref{ExamplesMuSec3fig1}. 
\begin{figure}[h]
 \begin{center}
  \resizebox{10cm}{!}{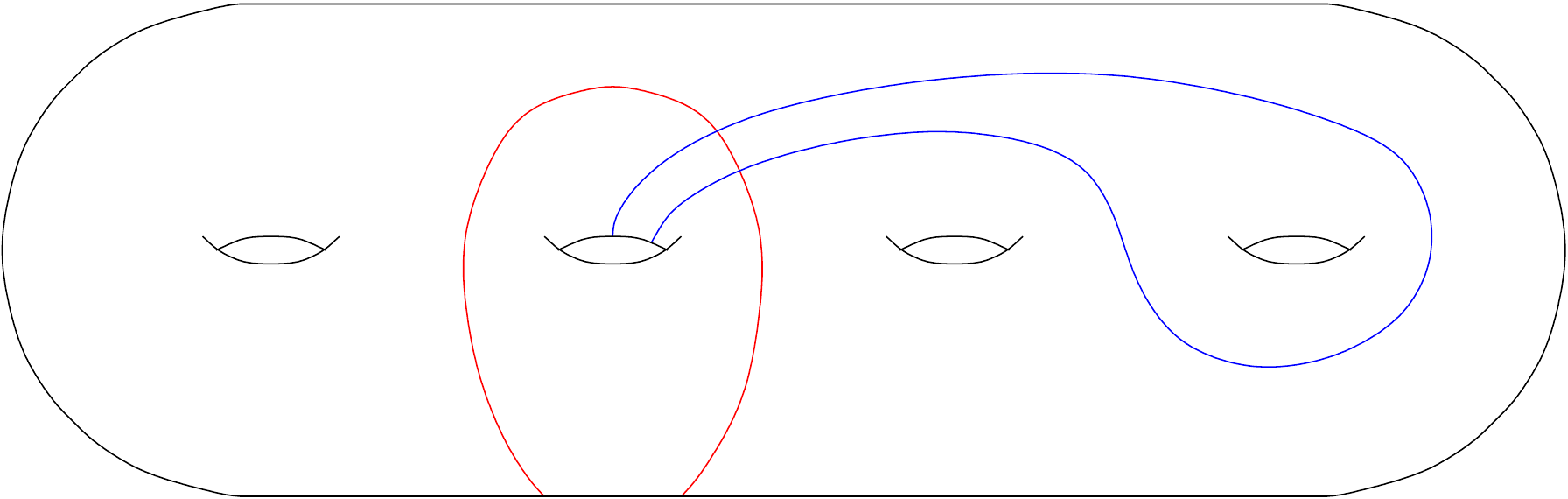} \caption{Examples of $\mu_{3,\{0,1,2\}}$ in red, and $\sigma_{7,\{4,5,6\}}$ in blue.}\label{ExamplesMuSec3fig1}
 \end{center}
\end{figure}\\
\indent This way, we define the following set: $$\Mf \ColonEqq \{\mu_{i,J}^{\pm}: J = \{j, \ldots, j+2k\}  \hspace{0.2cm} \mathrm{for} \hspace{0.2cm} \mathrm{some} \hspace{0.2cm} k \in \mathbb{Z}^{+}, i = j+2k+1\}$$
\indent Finally, we have the set $$\X(S) \ColonEqq \Cf \cup \Bf \cup \Sf \cup \Mf$$
\indent Recall that, as was mentioned in the Introduction, this set was proved to be rigid in \cite{Ara1}, and by construction has trivial pointwise stabilizer in $\EMod{S}$.
\subsection{$\X(S)$ for punctured surfaces}\label{chap3sec1subsec2}
\indent Let $\Cf$, $\Bf_{0}$, $\Df$ and $\Bf_{T}$ be as in Section \ref{chap2}.\\
\indent For $0 \leq i \leq j \leq n$, we denote by $N_{1}^{i, \hspace{0.05cm} j}$ and $N_{2g+1}^{i, \hspace{0.05cm} j}$ closed regular neighbourhoods of the chains $\{\alpha_{0}^{i}, \alpha_{0}^{j}, \alpha_{1}\}$ and $\{\alpha_{0}^{i}, \alpha_{0}^{j}, \alpha_{2g+1}\}$ respectively.\\
\indent Note that $\N{1}{i}{j}$ is a two-holed torus if $j - i \geq 1$ (one of the boundary components will not be essential if $j - i = 1$). Also, $S \backslash \N{1}{i}{j}$ is the disjoint union of a subsurface homeomorphic to an at least once-punctured open disc, and a subsurface homeomorphic to $S_{g-1,n-(j-1)+1}$.\\
\indent If $j - i > 1$, one of the boundary components of $\N{1}{i}{j}$ is the curve $\eps{i}{j}$. On the other hand, for $0 \leq i \leq j \leq n$, we denote by $\sigma_{1}^{i, \hspace{0.05cm} j}$ the boundary component of $\N{1}{i}{j}$ such that one of the connected components of $S \backslash \{\sigma_{1}^{i, \hspace{0.03cm} j}\}$ is homeomorphic to $S_{1,j-i + 1}$.\\
\indent We denote by $\sigma_{2g+1}^{i, \hspace{0.05cm} j}$ the analogous boundary curves of $\N{2g+1}{i}{j}$ (whenever they are essential).\\
\indent Then, we define $$\Sf_{T} \ColonEqq \{\sigma_{l}^{i, \hspace{0.05cm} j} : l \in \{1,2g+1\}, 0 \leq i \leq j \leq n\}$$
\indent Now, let $J$ be a subinterval of $\{0, \ldots, 2g+1\}$ (modulo $2g+2$) such that $|J| \leq 2g$.\\
\indent If $J = \{i, \ldots, i + 2k-1\}$ for some $k \in \mathbb{Z}^{+}$, let $\sigma_{J} = [\alpha_{i}, \ldots, \alpha_{i+2k-1}]$ (with $\alpha_{0} = \alpha_{0}^{1}$ if necessary). We define $$\Sf_{0} \ColonEqq \{\sigma_{J} : J = \{i, \ldots, i + 2k-1\}, k \in \mathbb{Z}^{+}\}.$$
\indent If $J = \{2l, \ldots, 2(l+k)\}$, for some $k \in \mathbb{Z}^{+}$, and $j = 2(l+k) +1$, then $i(\alpha_{j},\beta_{J}^{+}) = i(\alpha_{j},\beta_{J}^{-}) = 1$. Let $\mu_{j,J}^{+}$ be the boundary curve of a regular neighbourhood of $\{\alpha_{j},\beta_{J}^{+}\}$. Analogously, let $\mu_{j,J}^{-}$ be the boundary curve of a regular neighbourhood of $\{\alpha_{j},\beta_{J}^{-}\}$. We define $$\Mf \ColonEqq \{\mu_{j,J}^{\pm} : J = \{2l, \ldots, 2(l+k)\}, k \in \mathbb{Z}^{+}, j = 2(l+k) +1\}.$$
\indent Therefore, we define: $$\X \ColonEqq \Cf \cup \Df \cup \Sf_{T} \cup \Sf_{0} \cup \Bf_{T} \cup \Bf_{0} \cup \Mf.$$
\indent Recall that, as was mentioned in the Introduction, this set was proved to be rigid in \cite{Ara1}, and by construction has trivial pointwise stabilizer in $\EMod{S}$.
\subsection{Proof of Theorem \ref{Xexhausts}}\label{chap3sec2subsec3}
\indent The set $\X(S)$ is studied in \cite{Ara1} and \cite{Ara2}, and it is proven to be a finite rigid set of $\ccomp{S}$ (Theorems 5.1 and 6.1 in \cite{Ara1}). Also, by construction, we have that the principal sets used in Sections \ref{chap1} and \ref{chap2} ($\Cf \cup \Bf$ for closed surfaces, $\Cf \cup \Bf_{0}$ for punctured surfaces) are contained in their respective $\X(S)$, which gives us the proof of Theorem \ref{Xexhausts}:
\begin{proof}[\textbf{Proof of Theorem \ref{Xexhausts}}]
 Since $\Cf \cup \Bf \subset \X(S)$ for $S$ closed (and $\Cf \cup \Bf_{0} \subset \X(S)$ for $S$ a punctured surface), we have that $(\Cf \cup \Bf)^{k} \subset \X(S)^{k}$ for any $k \in \nat$ (analogously $(\Cf \cup \Bf_{0})^{k} \subset \X(S)^{k}$ for any $k \in \nat$). This implies that $\bigcup_{i \in \nat} (\Cf \cup \Bf)^{i} = \bigcup_{i \in \nat} \X(S)^{i}$ (analogously $\bigcup_{i \in \nat} (\Cf \cup \Bf_{0})^{i} = \bigcup_{i \in \nat} \X(S)^{i}$).\\
 \indent This coupled with Theorem \ref{Thm2} and \ref{Thm3} gives us the desired result.
\end{proof}
\indent Recalling that in Proposition 3.5 in \cite{Ara2}, Aramayona and Leininger prove that the rigid expansions of rigid set are themselves rigid, note that Theorem \ref{Xexhausts} gives and alternative proof of Theorem 1.1 in \cite{Ara2}.\\

\end{document}